\documentclass{article}
\usepackage[utf8]{inputenc}
\usepackage{amsmath,amscd,mathtools,cases,amsthm,amsopn,amssymb,soul,cleveref,mathabx,relsize,stmaryrd,upquote}
\usepackage{mathrsfs}
\usepackage{thmtools}
\usepackage{authblk}
\usepackage{thm-restate}
\usepackage{multicol}
\usepackage{tikz}
\usepackage{xcolor}
\usepackage[title]{appendix}
\usepackage{comment}
\usepackage{tikz,pgfplots}
\pgfplotsset{compat=1.18}
\usetikzlibrary{plotmarks,external}
\usepgfplotslibrary{groupplots}
\makeatletter
\pgfkeys{/pgfplots/colormap={viridis reversed}{
	rgb=(0.99324,0.90616,0.14394)
	rgb=(0.84557,0.88733,0.0997)
	rgb=(0.68895,0.86545,0.18272)
	rgb=(0.53561,0.83578,0.2819)
	rgb=(0.39517,0.79747,0.36775)
	rgb=(0.27415,0.75198,0.4366)
	rgb=(0.18065,0.7014,0.48819)
	rgb=(0.12808,0.64775,0.5235)
	rgb=(0.12115,0.59274,0.54465)
	rgb=(0.13777,0.53749,0.5549)
	rgb=(0.1592,0.48224,0.55807)
	rgb=(0.18225,0.42618,0.55711)
	rgb=(0.20862,0.36775,0.55267)
	rgb=(0.23744,0.3052,0.54192)
	rgb=(0.26366,0.23763,0.51877)
	rgb=(0.28026,0.1657,0.4765)
	rgb=(0.28192,0.08966,0.41241)
	rgb=(0.267,0.00487,0.32942)
}}
\def\ticklabelfakedlogalth{$\smash[t]{10^{\pgfmathprintnumber{\tick}}}$\vphantom{1}}
\tikzset{custom matrix plot settings/.style={/pgfplots/.cd,enlargelimits=false,colormap name=viridis reversed,title style={font=\small,at={(xticklabel cs:0.5)},below,align=center,text width=3.1cm},every axis title shift=-16pt,mesh/ordering=colwise,point meta=explicit,axis lines*=left,axis line shift=2.5pt,tick align=outside,separate axis lines=false,tick style={black,thin},tick label style={font=\footnotesize},label style={font=\footnotesize},tickwidth=1.2mm,subtickwidth=0.75mm,major grid style={black!25,thin},minor grid style={black!15,line width=0.3pt},label shift=-1.5pt,colorbar horizontal/.append style={colorbar style={axis line style={line width=0.3pt},xticklabel=\ticklabelfakedlogalth,xlabel={\#\ of graphs},xticklabel pos=upper,xlabel style={font=\footnotesize},axis lines*=box,tick align=inside,axis line shift=0pt,tick style={white,thin},colorbar/width=0.39cm,minor tick num=0,anchor=south west,max space between ticks=28}},axis line style=thin,every axis/.append style={line width=0.1pt}}}
\usetikzlibrary{calc}
\usetikzlibrary{decorations.pathreplacing}
\usetikzlibrary{decorations.markings}
\usetikzlibrary{graphs,graphs.standard,quotes}
\creflabelformat{enumi}{#2(#1)#3}

\tikzstyle{vertex}=[circle, draw, inner sep=0pt, minimum size=4pt,fill=black]

\tikzstyle{whitevertex}=[circle, draw, inner sep=0pt, minimum size=4pt,fill=white]
\tikzstyle{bluevertex}=[circle, draw, inner sep=0pt, minimum size=4pt,fill=black]

\tikzstyle{hollowvertex}=[circle, draw, inner sep=0pt, minimum size=4pt, fill=white]

\tikzstyle{phantomvertex}=[circle, draw, inner sep=0pt, minimum size=4pt,color=white]

\usetikzlibrary{backgrounds}
\usetikzlibrary{arrows.meta}
\tikzset{
  .../.tip={[sep=2pt 2]
    Round Cap[]. Circle[length=0pt 2,sep=2pt] Circle[length=0pt 2,sep=2pt] Circle[length=0pt 2, sep=2pt 2]}}
\tikzset{
  .../.tip={[sep=2pt 2]
    Round Cap[]. Circle[length=0pt 2,sep=2pt] Circle[length=0pt 2,sep=2pt] Circle[length=0pt 2, sep=2pt 2]}}
\tikzset{
      ellipsis/.tip={
Square[length=2pt,sep=0pt,color=white] Circle[length=1pt,sep=0pt,color=black] Square[length=1pt,sep=0pt,color=white]
Circle[length=1pt,sep=0pt,color=black] Square[length=1pt,sep=0pt,color=white]
Circle[length=1pt,sep=0pt,color=black] Square[length=2pt,sep=0pt,color=white]}}
\tikzset{middlearrow/.style n args={2}{
        decoration={markings,
            mark= at position {#2} with {\arrow{#1}} ,
        },
        postaction={decorate}
    }
}
\usetikzlibrary{positioning,datavisualization}

\tikzset{/tikz/data visualization/scientific axes/clean without border/.style={
    /tikz/data visualization/@make axes/.style={
      x axis={
        visualize ticks={common={y axis={goto=padded min},direction axis=y axis,high=0pt},major={tick text at low}},
        visualize axis={y axis={goto=padded min}},
        visualize grid={common={direction axis=y axis}},
        padding=.5em,
      },
      y axis={
        visualize ticks={common={x axis={goto=padded min},direction axis=x axis,high=0pt},major={tick text at low}},
        visualize axis={x axis={goto=padded min}},
        visualize grid={common={direction axis=x axis}},
        padding=.5em
      },
    }
}}

\tikzset{graph styles 1/.style={vertex/.style={black,fill,draw,minimum size=6pt,inner sep=0pt,circle,thin},bold/.style={black,line width=0.6mm},plain/.style={black,thin},bold edges/.style=bold,plain edges/.style=plain,label distance=1mm,every label/.style=text node},graph styles 2/.style={graph styles 1,vertex/.append style={fill=none,minimum size=8pt}},text node/.style={rectangle,fill=none,draw=none},caption node/.style={text node,font=\Large}}
\tikzstyle{vertex_alt}=[circle, draw, minimum size=8pt]

\newtheorem{theorem}[subsubsection]{Theorem}
\newtheorem{corollary}[subsubsection]{Corollary}
\newtheorem{observation}[subsubsection]{Observation}
\newtheorem{proposition}[subsubsection]{Proposition}
\newtheorem{lemma}[subsubsection]{Lemma}

\newtheorem{case}{Case}[theorem]
\theoremstyle{definition}
\newtheorem{definition}[subsubsection]{Definition}
\newtheorem{example}[subsubsection]{Example}
\newtheorem*{unnumberednote}{Note}

\def\GraphSix#1"{\begingroup(\itshape graph6:\nolinebreak[3]\ \verb"\aftergroup\endgroup\aftergroup)}

\usepackage{amsopn,amssymb,amsthm,amsmath,cleveref}
\usepackage{chngcntr}
\counterwithin{figure}{section}
\counterwithin{table}{section}
\DeclareMathOperator{\Z}{Z}
\DeclareMathOperator{\Zp}{Z_+}
\newcommand{\Zs}{\Z_-}
\DeclareMathOperator{\F}{F}
\DeclareMathOperator{\comp}{comp}

\DeclareMathOperator{\term}{Term}
\DeclareMathOperator{\rev}{Rev}

\DeclareMathOperator{\cl}{cl}
\DeclareMathOperator{\ini}{initial}
\DeclareMathOperator{\rd}{rd}
\DeclareMathOperator{\p}{P}
\DeclareMathOperator{\T}{T}
\newcommand{\match}{\mu}

\newcommand{\cartProd}{\mathbin{\Box}}
\newcommand{\symmDiff}{\mathbin{\Delta}}
\newcommand{\TokGrf}[2]{\mathcal{#1}^{\mathrm{#2}}}
\newcommand{\RTokExch}{\TokGrf{R}{TE}}
\newcommand{\RTokSlide}{\TokGrf{R}{TS}}
\newcommand{\GenTokExch}{\TokGrf{R}{TE}_X}
\newcommand{\GenTokSlide}{\TokGrf{R}{TS}_X}
\newcommand{\ZTokExch}{\TokGrf{Z}{TE}}
\newcommand{\ZTokSlide}{\TokGrf{Z}{TS}}
\newcommand{\TokSlide}{\ZTokSlide}
\newcommand{\TokExch}{\ZTokExch}
\newcommand{\Bnull}{B^\emptyset_-}
\newcommand{\Wnull}{W^\emptyset_-}
\newcommand{\SDTwo}{\mathop{SD_2}}
\newcommand{\simrel}{\sim}
\newcommand{\simname}{\mathord{\simrel}}

\newcommand{\OverleafHack}[2]{}

\usepackage{xcolor} 
\usepackage{booktabs}
\usepackage{lineno}
\usepackage[margin=1in, footskip=0.5in]{geometry}
\title{Reconfiguration of Minimum PSD Forcing Sets and Minimum Skew Forcing Sets}
 \author[1]{Novi Bong}
 \affil[1]{University of Delaware, nhbong@udel.edu}
\author[2]{Mary Flagg}
 \affil[2]{University of St. Thomas, Houston, flaggm@stthom.edu}
\author[3]{Mark Hunnell}
\affil[3]{Winston-Salem State University,  hunnellm@wssu.edu}
\author[4]{John Hutchens}
\affil[4]{University of San Francisco, johnd.hutchens@gmail.com}
\author[5]{Ryan Moruzzi}
\affil[5]{California State University East Bay, ryan.moruzzi@csueastbay.edu}
\author[6]{Houston Schuerger}
\affil[6]{Univeristy of Texas Permian Basin, schuerger\_h@utpb.edu}
\author[7]{Ben Small}
\affil[7]{University Place WA, bentsm@gmail.com}

 \date{\textbf{Keywords:} Reconfiguration, Zero Forcing, PSD Forcing, Skew Forcing \\Corresponding Author: Mark Hunnell}

\begin{document}

\maketitle

\begin{abstract}
    {Reconfiguration graphs provide a way to represent relationships among solutions to a problem, and have been studied in many contexts.  We investigate the reconfiguration graphs corresponding to minimum PSD forcing sets and minimum skew forcing sets.  We present results for the structure and realizability of certain graph classes as token exchange and token sliding reconfiguration graphs.  Additionally, we use a universal approach to establish structural properties for the reconfiguration graphs of many common graph parameters under these reconfiguration rules. Finally, we compare results on reconfiguration graphs for zero forcing variants.  }
\end{abstract}

\section{Introduction}

Reconfiguration examines the relationship among solutions to a problem.  The vertices of the reconfiguration graph are the feasible solutions to the problem, and its edges are determined by a reconfiguration rule.  Namely, there exists an edge between two vertices of the reconfiguration graph if the corresponding solutions can be transformed into one another by exactly one application of the reconfiguration rule.  Broadly speaking, reconfiguration is studied from three perspectives: structure (connectedness, Hamiltonicity, diameter, etc.), realizability (which graphs arise as a specific kind of reconfiguration graph), and algorithmic properties (such as finding the shortest path between two solutions quickly).  This article is concerned with questions of the first two types.

Given a graph $G$ with vertices $V(G)$, we are interested in problems whose solutions are subsets of vertices, i.e., of the form $S\subseteq V(G)$.  Each set $S$ having the required property is a vertex of the reconfiguration graph.  Clearly, the reconfiguration rule chosen is important to the structure of the reconfiguration graph. In the literature, three reconfiguration rules are considered most frequently.

\begin{enumerate} 
    \item {\bf Token Exchange (TE):} two vertices $S$ and $S^{\prime}$ are adjacent if and only if $S^{\prime}$ is obtained by exchanging one vertex in $S$ for a vertex in $S^{\prime}$ (note this is also sometimes called token jumping);
    \item {\bf Token Sliding (TS):} two vertices $S$ and $S^{\prime}$ are adjacent if and only if $S^{\prime}$ is obtained by exchanging exactly one vertex in $S$ for a vertex in $S^{\prime}$, and the exchanged vertices are neighbors;
    \item {\bf Token Addition and Removal (TAR):} two vertices $S$ and $S^{\prime}$ are adjacent if and only if $S^{\prime}$ is obtained from $S$ through the addition or removal of exactly one vertex.
\end{enumerate}

Reconfiguration questions have been explored in many contexts, including vertex coloring \cite{beier2016coloring,cereceda2008colourings}, independent sets \cite{avis2023independentsets}, dominating sets \cite{bonamy2021dominating}, zero forcing \cite{geneson2023reconfiguration}, and power domination \cite{bjorkman2022power}.  An overview of reconfiguration on graph colorings and dominating sets is given in \cite{mynhardt2020survey}.  

In this article we are primarily interested in the token exchange and token sliding reconfiguration graphs for two variants of zero forcing: PSD forcing and skew forcing.  (Formal definitions will be given in \Cref{prelim}.) PSD forcing was introduced in \cite{param}, and has since been the subject of much interest, see for instance \cite{ekstrand2013psdzf,psdfort,hogben2023new,psdprop}. Similarly, skew forcing and related graph parameters have been studied extensively, see for example   \cite{allison2010minimum,ansill2016failed,curl2019skew,dealba2014some,kingsley2015skew}.   Token exchange reconfiguration graphs have been studied for both standard zero forcing \cite{geneson2023reconfiguration} and power domination \cite{bjorkman2022power}, hence our focus on PSD forcing and skew forcing.

However, many structural questions about the reconfiguration graphs are independent of the forcing rule used, and furthermore apply to the reconfiguration of many graph parameters not related to zero forcing.  Thus we develop a universal approach in \Cref{universal} to establish results for the token exchange and token sliding reconfiguration rules that are of interest to our forcing variants as well as many other graph parameters. The two variants provide an interesting contrast when considering their structural and realizable properties.  These are developed and discussed for PSD forcing in \Cref{psd} and for skew forcing in \Cref{skew}.  In \Cref{connections}, we consider a pair of instances in which there is a deep connection between certain PSD and skew reconfiguration graphs.  In \Cref{compare}, we use the work in \cite{geneson2023reconfiguration} and our results in this article to compare properties of reconfiguration graphs across zero forcing variants.

\section{Preliminaries and notation}
\label{prelim}
In this section we provide definitions for graphs, zero forcing and its variants, and reconfiguration.  Additional background can be found in \cite{hogben2022book} and \cite{mynhardt2020survey}.

\subsection{Graph terminology}
A {\em graph} $G$ is a pair $(V(G),E(G))$, where $V(G)$ is the set of {\em vertices} and $E(G)$ is a set of 2-element sets of vertices called {\em edges}.  To help differentiate between subsets of $V(G)$ of cardinality two and edges, given two distinct vertices $u,v\in V(G)$, the subset of $V(G)$ composed of these two vertices will be denoted by $\{u,v\}$ while the edge between $u$ and $v$ will be denoted by $uv$ (or $vu$).  All graphs are assumed to be finite and {\em simple}, that is, to have neither loops (edges between a vertex and itself) nor multiple edges between any two distinct vertices.  

If $G$ and $H$ are graphs such that $V(H) \subseteq V(G)$ and $E(H) \subseteq E(G)$, then $H$ is a {\em subgraph} of $G$, denoted $H\leq G$.  If $H$ is a subgraph of $G$ and for any  two vertices $u,v \in V(H)$ we have $uv \in E(H)$ if and only if $uv \in E(G)$, then $H$ is an {\em induced subgraph} of $G$.  Given a subset of vertices $S \subseteq V(G)$, the induced subgraph $H$ of $G$ with vertex set $V(H)=S$ will be denoted $G[S]$ and called the subgraph induced by $S$.  In addition, for a given set of vertices $S \subseteq V(G)$, the notation $G-S$ will be used to denote $G[V(G) \setminus S]$.  

Given a vertex $v \in V(G)$, the {\em open neighborhood} of $v$ in $G$, denoted by $N_G(v)$, is the collection of vertices $u$ such that $vu \in E(G)$.  The {\em closed neighborhood} of a vertex $v$, denoted by $N_G[v]$, is $N_G[v]=N_G(v) \cup \{v\}$.  Likewise, given a set of vertices $S \subseteq V(G)$, the open (respectively, closed) neighborhood of $S$ is defined to be the union of the open  (respectively, closed) neighborhoods of the vertices of $S$ and is denoted by $N_G(S)$ (respectively, $N_G[S]$).  A {\em clique} is a set $S \subseteq V(G)$ such that for each $u,v \in S$ distinct, we have $uv \in E(G)$.  We denote by $\omega(G)$ the size of the largest clique of $G$. 

A \emph{path} is a sequence of distinct vertices $v_1,v_2,\dots,v_m$ such that for each $i$ with $1 \leq i \leq m-1$ we have $v_iv_{i+1} \in E(G)$; the \emph{length} of a path is the number of edges in that path.  Given a pair of vertices $u,v \in V(G)$, a {\em $uv$-path} is a path $v_1,v_2,\dots,v_m$ such that $u=v_1$, $v=v_m$.  A graph $G$ is said to be \emph{connected} if for each pair of vertices $u,v \in V(G)$ there exists a $uv$-path.  A graph is {\em disconnected} if it is not connected.   The maximal connected subgraphs of $G$ are its {\em components}.  The set of these components is denoted $\comp(G)$, and given a vertex $v$, the component of $G$ that contains the vertex $v$ is denoted $\comp(G,v)$.  The {\em distance} between two vertices $u$ and $v$ in the same component of a graph is the length of the shortest $uv$-path.  We also view a path as a graph: the \emph{path} $P_n$ is the graph with $V(P_n)=\{v_1,\dots,v_n\}$ and $E(P_n)=\{{ v_iv_{i+1}}: i=1,\dots,n-1\}$.  We denote the complete graph, complete bipartite graph, and cycle graph on $n=p+q$ vertices by $K_n$, $K_{p,q}$, and $C_n$ respectively.

The {\em join} of two graphs $G$ and $H$, denoted by $G \vee H$, is the graph with vertex set $V(G \vee H)=V(G) \cup V(H)$ and edge set $E(G \vee H) = E(G) \cup E(H) \cup \{gh:g \in G, h \in H\}$ (effectively, it is the disjoint union of $G$ and $H$ with every possible edge between $G$ and $H$ added to the graph).  The {\em Cartesian product} of two graphs $G$ and $H$, denoted $G \cartProd H$ is the graph with vertex set $V(G) \times V(H)$ such that given $(g_1,h_1),(g_2,h_2) \in V(G \cartProd H)$, we have $(g_1,h_1)(g_2,h_2) \in E(G \cartProd H)$ if and only if either $g_1=g_2$ and $h_1h_2 \in E(H)$ or $h_1=h_2$ and $g_1g_2 \in E(G)$.

We conclude this section by highlighting some less familiar definitions used in this article for ease of reference later.
\begin{definition}
    Let $G$ be a graph and $\simname$ be an equivalence relation on $V(G)$, forming the partition $V(G)/\simname$, with the equivalence class of $v \in V(G)$ denoted $[v]$.  The {\em quotient graph} of $G$ with respect to $\simname$ is the graph with vertex set $V(G/\simname)=V(G)/\simname$ and edge set $E(G/\simname)$ such that given $[v],[u] \in V(G/\simname)$ we have $[v][u]\in E(G/\simname)$ if and only if $[v] \neq [u]$ and there exist $w_v \in [v]$ and $w_u \in [u]$ such that $w_vw_u \in E(G)$.
\end{definition}

\begin{definition}
      Given vertex-disjoint graphs $G$ and $H$ and vertices $g_v \in V(G)$ and $h_v \in V(H)$, the {\em vertex sum} of $G$ and $H$ given by the identification of $g_v$ and $h_v$, denoted $G \oplus_v H$ is the graph $(G \cup H)/\simname$ such that given distinct $w,z \in V(G) \cup V(H)$, we have $w \simrel z$ if and only if $\{w,z\}=\{g_v,h_v\}$. Due to the simplistic nature of the partition given by $\simname$, we denote $[g_v]$ by $v$ and we denote $[u] \neq [g_v]$ by $u$.
\end{definition}

\begin{definition}
\label{DuplicationTypes}
Let $G$ be a graph and $v \in V(G)$.  Let the graph $G'$ be the result of adding a vertex $v'$ to $G$ such that $N_{G'}(v')=N_G(v)$. The graph $G'$ is said to be constructed from the \emph{duplication} of $v$, and $v$ and $v'$ are said to be \emph{independent twins}.  (For clarity, we will also use the term \emph{independent duplication} to refer to this construction.)  Let the graph $G''$ be the result of adding a vertex $v''$ to $G$ such that $N_{G''}(v'')=N_G[v]$. The graph $G''$ is said to be constructed from the \emph{join-duplication} of $v$, and $v$ and $v''$ are said to be \emph{adjacent twins}.
\end{definition}

\subsection{Zero forcing}\label{zeroforcing}
In each of the variants of zero forcing discussed here, one starts with a graph and colors every vertex of the graph either blue or white. 

Given a set  of blue vertices (with the other vertices colored white), a process is started during which vertices cause white vertices to become blue.  Each variant $X$ is defined by its color change rule $X$-CCR, which governs under what circumstances a vertex can cause a white vertex to become blue during this process.  When a vertex $u$ causes a white vertex $v$ to become blue, this is referred to as $X$-\emph{forcing} (or simply \emph{forcing} when the context is clear) and denoted $u \rightarrow v$.  There is often a choice as to which vertex is chosen to force $v$ among those that can; for standard zero forcing, PSD forcing, and skew forcing, the final set of blue vertices does not depend on which vertices are chosen to perform the forces {or the order in which valid forces are performed}.   It is worth noting that once a vertex is blue, it will remain blue.   In the variants discussed, the goal is for every vertex in the graph to become blue, so the color change rule will be applied until no further forces are possible.  The relevant color change rules are as follows:  
\begin{itemize}
    \item Standard zero forcing color change rule ($\Z$-CCR): If a blue vertex $u$ has a unique white neighbor $v$, then $u$ can force $v$ to become blue;
    \item Positive semidefinite (PSD) forcing color change rule ($\Zp$-CCR): If $B$ is the set of currently blue vertices, $C$ is a component of $G-B$, and $u$ is a blue vertex such that $N_G(u) \cap V(C)=\{v\}$, then $u$ can force $v$ to become blue;
    \item Skew forcing color change rule ($\Zs$-CCR): If any vertex $u$ has a unique white neighbor $v$, then $u$ can force $v$ to become blue.
\end{itemize}

A {\em standard zero forcing set} of a graph $G$ is a set of vertices $B$ such that if $B$ is the set of initially blue vertices and $\Z$-CCR is applied a sufficient number of times, then all vertices of $G$ become blue.  The {\em standard zero forcing number} of a graph $G$ is $\Z(G)=\min\{\vert B\vert: B \text{ is a zero forcing set of }G\}$. A {\em minimum standard zero forcing set} $B$ of a graph $G$ is a zero forcing set of $G$ such that $\vert B \vert=\Z(G)$. 
A {\em (minimum) PSD forcing set} and the {\em PSD forcing number $\Zp(G)$} of a graph are defined analogously using $\Zp$-CCR as the color change rule. A {\em (minimum) skew forcing set} and the {\em skew forcing number $\Zs(G)$} are also defined analogously using the $\Zs$-CCR color change rule.

We now recall the notion of a PSD fort before introducing the analogous definition for skew forts. 

\begin{definition}[\cite{psdfort}]\label{defPSDfort}
Let $G$ be a graph, $F \subseteq V(G)$ be nonempty, and $\{C_i\}_{i=1}^k$ be the components of $G[F]$.  $F$ is a \emph{PSD fort} of $G$ if for each $v \in V(G) \setminus F$ and each $i \in \{1,2,\dots,k\}$, $|N_G(v) \cap V(C_i)| \neq 1$. 
\end{definition}

\begin{theorem}[\cite{psdfort}]\label{thmPSDfort}
Let $G$ be a graph, and $B \subseteq V(G)$.  $B$ is a PSD forcing set of $G$ if and only if $B \cap F \neq \emptyset$ for each PSD fort $F$ of $G$.
\end{theorem}

It is worth noting that every graph $G$ contains at least one PSD fort, $V(G)$.  We extend the definition of forts to skew forcing with the following definition:

\begin{definition}
Let $G$ be a graph and $F$ be a nonempty subset of $V(G)$.  If for each vertex $v \in V(G)$, $\left|N_G(v) \cap F\right| \neq 1$, then $F$ is a \emph{skew fort} of $G$.
\end{definition}

Not every graph contains a skew fort; in particular, if $\Z_-(G)=0$, then $G$ contains no skew forts.  However, it is worth noting that in many cases $V(G)$ is still a skew fort.  Specifically, if $G$ contains no pendent vertices, then $V(G)$ is a skew fort of $G$, and as we will see from \Cref{skewfort}, this ensures that $\Z_-(G) \geq 1$.

\begin{lemma}\label{skewfortcup}
Let $G$ be a graph and $F_1$,$F_2$ be skew forts of $G$.  Then $F_1 \cup F_2$ is a skew fort of $G$.
\end{lemma}

\begin{proof}
Let $v \in V(G)$, and note that $\max\{|N_G(v) \cap F_1|, |N_G(v) \cap F_2|\} \leq |N_G(v) \cap (F_1 \cup F_2)| \leq {|N_G(v) \cap F_1|}+|N_G(v) \cap F_2|$.  Since neither $|N_G(v) \cap F_1|=1$ nor $|N_G(v) \cap F_2|=1$, it follows that $|N_G(v) \cap (F_1 \cup F_2)| \neq 1$.
\end{proof}

We borrow the concept of the closure of a set of vertices, which provides a concise notation for the result of a certain skew forcing process:

\begin{definition}
Let $G$ be a graph, and let $S\subseteq V(G)$.  Starting with $S$ as the initial set of blue vertices, apply the skew forcing color change rule until no further forces are possible.  The resulting set of blue vertices is the \emph{skew closure} of $S$, denoted $\cl_-(G,S)$.
\end{definition}

\begin{proposition}
\label{SkewClosureComplement}
Let $G$ be a graph and let $S\subseteq V(G)$ be a set of vertices that is not a skew forcing set of $G$.  Then the complement of its skew closure $V(G)\setminus\cl_-(G,S)$ is a skew fort of $G$.
\end{proposition}

\begin{proof}
Since $S$ is not a skew forcing set of $G$, $\cl_-(G,S)\neq V(G)$.  By definition, there are no valid skew forces for $\cl_-(G,S)$.  Thus, for all $v\in V(G)$, $|N_G(v)\cap(V(G)\setminus\cl_-(G,S))|\neq 1$.  Thus $V(G)\setminus\cl_-(G,S)$ is a skew fort of $G$.
\end{proof}

\begin{lemma}
\label{ClosureVsForts}
Let $G$ be a graph, let $S\subseteq V(G)$, and let $F$ be a skew fort of $G$ such that $S\cap F=\emptyset$.  Then $\cl_-(G,S)\cap F=\emptyset$ as well.
\end{lemma}

\begin{proof}
Suppose, by way of contradiction, that $w$ is the first vertex of $F$ to become blue during a skew forcing process beginning with $S$.  Let $S'$ be the set of blue vertices at the time that $w$ is forced.  Then $F\subseteq V(G)\setminus S'$ and there exists $v\in V(G)$ such that $N_G(v)\cap (V(G)\setminus S')=\{w\}$.  However, then \[\{w\}\subseteq N_G(v)\cap F\subseteq N_G(v)\cap (V(G)\setminus S')=\{w\},\] contradicting the fact that $F$ is a skew fort.
\end{proof}

\begin{corollary}
\label{skewderivedset}
Let $G$ be a graph, and let $S\subseteq V(G)$.  Then $V(G)\setminus\cl_-(G,S)=\bigcup\mathbb{F}_S$, where $\mathbb{F}_S$ is the set of skew forts $F$ of $G$ such that $S\cap F=\emptyset$.
\end{corollary}

\begin{proof}
If $F\in\mathbb{F}_S$, then $F\subseteq V(G)\setminus\cl_-(G,S)$, by \Cref{ClosureVsForts}.  On the other hand, $V(G)\setminus\cl_-(G,S)\in \mathbb{F}_S$ since $V(G)\setminus\cl_-(G,S)$ is a skew fort of $G$ by \Cref{SkewClosureComplement}.  Thus, $V(G)\setminus\cl_-(G,S)=\bigcup\mathbb{F}_S$.
\end{proof}

The following theorem is an immediate result of \Cref{skewderivedset}.

\begin{theorem}
\label{skewfort}
Let $G$ be a graph and $S \subseteq V(G)$.  Then $S$ is a skew forcing set of $G$ if and only if $S \cap F \neq \emptyset$ for each skew fort $F$ of $G$.
\end{theorem}

\subsection{Reconfiguration}

\begin{definition}
\label{VertexParameterDef}
Let $X$ be a graph parameter.  If there exists a property $x$ such that either:
\[\text{for every graph $G$, }X(G)=\max\{|B|\colon B \subseteq V(G) \text{ and }  B \text{ has property }x \text{ in }G\},\] 
\[\text{ or }\] 
\[\text{for every graph $G$, }X(G)= \min\{|B|\colon B \subseteq V(G) \text{ and } B \text{ has property }x \text{ in }G\},\]
then $X$ is a \emph{vertex parameter} defined by property $x$.  In the first case we say that $X$ is determined via a maximum and in the second case we say that $X$ is determined via a minimum.
\end{definition}

\begin{definition}
\label{TokReconfigDefs}
Let $X$ be a vertex parameter defined by a property $x$, and let $G$ be a graph.   The \emph{token exchange} $X$-reconfiguration graph of $G$, denoted $\GenTokExch(G)$, and the \emph{token sliding} $X$-reconfiguration graph of $G$, denoted $\GenTokSlide(G)$, are defined by
\[V(\GenTokExch(G))=V(\GenTokSlide(G))=\{S\subseteq V(G)\colon |S|=X(G) \text{ and } S \text{ has property }x \text{ in }G\},\]
\begin{multline*}
E(\GenTokExch(G))=\{S_1S_2\colon S_1,S_2\in V(\GenTokExch(G)) \text{ and there exist } v_1\in S_1\setminus S_2,\ v_2\in S_2\setminus S_1 \\
\text{such that } S_1\setminus\{v_1\}=S_2\setminus\{v_2\}\},
\end{multline*}
\[\text{and}\]
\begin{multline*}
E(\GenTokSlide(G))=\{S_1S_2\colon S_1,S_2\in V(\GenTokSlide(G)) \text{ and there exist } v_1\in S_1\setminus S_2,\ v_2\in S_2\setminus S_1 \\
\text{such that } S_1\setminus\{v_1\}=S_2\setminus\{v_2\} \text{ and } v_1v_2\in E(G)\}.
\end{multline*}

We refer to $G$ as the {\em source graph} of $\GenTokExch(G)$ and $\GenTokSlide(G)$.  When $X$ is a zero forcing variant, we use $\mathcal{Z}$ in place of $\mathcal{R}$ to denote reconfiguration graphs, appending the usual subscripts to denote the particular zero forcing variant; e.g., $\ZTokSlide_+(G)=\RTokSlide_{\Z_+}(G)$.  We refer to $\ZTokSlide_+(G)$ as the PSD token sliding graph; other reconfiguration graphs for zero forcing variants are referred to analogously.
\end{definition}

\begin{unnumberednote}
Token exchange and token sliding reconfiguration graphs can be defined more broadly or in other contexts (for instance, one could remove the $|S|=X(G)$ condition from the vertex set definition); however, \Cref{TokReconfigDefs} is sufficient for the purposes of this paper.
\end{unnumberednote}

\begin{observation}
\label{TS<=TE}
For any graph $G$, $\GenTokSlide(G)\leq\GenTokExch(G)$, with $V(\GenTokSlide(G))=V(\GenTokExch(G))$. 
\end{observation}

The following theorem identifies a nice property of token sliding graphs for any vertex parameter.

\begin{theorem}
Let $G$ be a bipartite graph and $X$ be a vertex parameter.  Then $\GenTokSlide(G)$ is bipartite.
\end{theorem}

\begin{proof}
Let $G$ be bipartite, with partite sets $A$ and $B$, and let $H=\GenTokSlide(G)$.  If $S_1S_2\in E(H)$, then $S_1\symmDiff S_2=\{a,b\}$ (where $P\symmDiff Q=(P\setminus Q)\cup (Q\setminus P)$ is the symmetric difference of $P$ and $Q$) for some $a\in A$ and $b\in B$ (as all edges in $H$ correspond to edges in $G$).  Thus, if $S_1S_2\in E(H)$, $|S_1\cap A|\equiv |S_2\cap A|+1\pmod{2}$ and $|S_1\cap B|\equiv |S_2\cap B|+1\pmod{2}$.  We now proceed by contradiction: assume that $H$ is not bipartite.  Then $H$ contains some cycle $S_1S_2\cdots S_{2k+1}S_1$ of odd length.  However, then $|S_1\cap A|\equiv |S_2\cap A|+1\equiv\cdots\equiv |S_{2k+1}\cap A|+2k\equiv |S_1\cap A|+2k+1\not\equiv |S_1\cap A|\pmod{2}$, a contradiction.
\end{proof}

To streamline arguments similar to the one above, we say that two sets of vertices $S_1$ and $S_2$ {\em differ by an exchange} if $S_1\setminus S_2=\{a\}$ and $S_2\setminus S_1=\{b\}$, and furthermore $S_1$ and $S_2$ {\em differ by a slide} if in addition $ab \in E(G)$.

\subsection{Relaxed chronologies and path bundles}

A new type of forcing process called a relaxed chronology was introduced in \cite{hogben2023new}.  Relaxed chronologies are generalizations of both chronological lists of forces and propagating families of forces.  Due to this, relaxed chronologies have more flexibility when it comes to which forces are performed at which time-step of a zero forcing process, and as such can be built around structural properties of the graphs in which the forcing process is occurring, making the discussion of certain results easier.

Let $G$ be a graph and $X \in \{\Z, \Z_+, \Z_-\}$. Define $S_X(G,B)$ to be the collection of $X$-forces $u \rightarrow v$ such that $u \rightarrow v$ is a valid force according to color change rule $X$-CCR when $B$ is the set of vertices which is currently blue.  A \emph{relaxed chronology} of $X$-forces $\mathcal{F}=\{F^{(k)}\}_{k=1}^K$ is an ordered collection of sets of $X$-forces with certain properties, used to represent a specific type of forcing process.   It has an associated $X$-forcing set $B \subseteq V(G)$ which is chosen as the set of vertices initially colored blue, and an induced \emph{expansion sequence} $\{E_\mathcal{F}^{[k]}\}_{k=0}^K$, where each set $E_\mathcal{F}^{[k]}$ is the set of vertices that are blue after time-step $k$.  We define $E_\mathcal{F}^{[0]}=B$ and $E_\mathcal{F}^{[k]}=E_\mathcal{F}^{[k-1]}\cup\{v\colon u\to v\in F^{(k)}\text{ for some }u\in V(G)\}$ for all $1\leq k\leq K$.  There are then two requirements for $\mathcal{F}$ to be a relaxed chronology for $B$: for each $1\leq k\leq K$, $F^{(k)}\subseteq S_X(G,E_\mathcal{F}^{[k-1]})$; and, for each $v\in V(G)\setminus B$, there exists exactly one $u\in V(G)$ such that $u\to v\in F^{(k)}$ for some $1\leq k\leq K$ (note that this $k$ must be unique, since it is never valid to force an already-blue vertex).

It is worth noting that if at each time-step $F^{(k)}$ is chosen to be maximal, then the relaxed chronology of $X$-forces $\mathcal F$ will be a \emph{propagating family} of $X$-forces, and if at each time-step $|F^{(k)}|=1$, then the relaxed chronology of $X$-forces will be a \emph{chronological list} of $X$-forces.  The {\em round function} is the function $\rd:V(G) \rightarrow \mathbb N$ such that if $v \in B$, then $\rd(v)=0$, otherwise, if $u \rightarrow v \in F^{(k)}$ for some $u \in V(G)$, then $\rd(v)=k$. 

Let $H$ be an induced subgraph of $G$.  We can relate a forcing process on $G$ to a forcing process on $H$ by restricting the forces in the following manner.  Given a set of $X$-forces $F$ between the vertices in $G$, define its \emph{restriction} to $H$, denoted $F|_H$, to be $\{v\to u\in F: u,v\in V(H)\}$.  Likewise, given a relaxed chronology $\mathcal{F}=\{F^{(k)}\}_{k=1}^K$ for some $X$-forcing set $B$ of $G$, define its \emph{restriction} to $H$, denoted $\mathcal{F}|_H$, to be $\{F^{(k)}|_H\}_{k=1}^K$.  For a relaxed chronology $\mathcal F$ of $X$-forces of $B$ on $G$, a vertex $u\in V(H)$ is an \emph{initial vertex} of $\mathcal F$ in $H$ if either $u \in B$ or $v \rightarrow u$ is a force which occurs during $\mathcal F$ but $v \in V(G) \setminus V(H)$.  The following lemma was introduced in \cite{hogben2023new} in the context of standard and PSD forcing; we extend it here to apply in the skew forcing context as well.

\begin{lemma}\label{initialPSD}\label{initialskew}
Let $X \in \{\Z, \Z_+, \Z_-\}$, $G$ be a graph, $H$ be an induced subgraph of $G$, $B$ be an $X$-forcing set of $G$, and $\mathcal{F}$ be a relaxed chronology of $X$-forces of $B$ on $G$. Let $\ini(\mathcal F|_H)$ be the set of initial vertices of $\mathcal F$ in $H$. Then $\ini(\mathcal F|_H)$ is an $X$-forcing set of $H$.
\end{lemma}

\begin{proof}
The case when $X$ is $\Z$ is Lemma 5.4 in \cite{hogben2023new}; the case when $X$ is $\Z_+$ is Lemma 5.3 from the same paper.  We will now treat the remaining case, when $X$ is $\Z_-$.

Let $\mathcal{F}=\{F^{(k)}\}_{k=1}^K$ with corresponding expansion sequence $\{E_{\mathcal{F}}^{[k]}\}_{k=0}^K$. We first show that the skew forces in $F^{(k)}|_H$ are valid when $E_{\mathcal{F}}^{[k-1]}\cap V(H)$ is blue. Consider $v\to u\in F^{(k)}|_H$. Since $v\to u\in F^{(k)}$, $u$ is the unique white neighbor of $v$ in $G$ when $E_{\mathcal{F}}^{[k-1]}$ is blue. Since $H$ is an induced subgraph of $G$ and $u,v \in V(H)$, $u$ is the unique white neighbor of $v$ in $H$ when $E_{\mathcal{F}}^{[k-1]}\cap V(H)$ is blue.

We now claim that if $(E_{\mathcal{F}}^{[k-1]}\cap V(H))\cup \ini(\mathcal F|_H)$ is blue, then all vertices in $(E_{\mathcal{F}}^{[k]}\cap V(H))\cup \ini(\mathcal F|_H)$ will be blue after the forces in $F^{(k)}|_H$ are performed. Consider $v\to u\in F^{(k)}$ with $u\in (E_{\mathcal{F}}^{[k]} \cap V(H))\cup \ini(\mathcal F|_H)$. If $v\notin V(H)$, then $u \in \ini(\mathcal F|_H)$. So we conclude that $u$ is initially blue, and hence also blue after time-step $k$. Otherwise, $v\in V(H)$, so the preceding paragraph implies that $v\to u\in F^{(k)}|_H$ is a valid skew force in $H$, and $u$ will be blue after time-step $k$. Combining these two results with induction on $k$, we conclude that $\mathcal{F}|_H$ is a relaxed chronology of skew forces for $\ini(\mathcal F|_H)$ in $H$ with expansion sequence $\{(E_{\mathcal{F}}^{[k]}\cap V(H))\cup \ini(\mathcal F|_H)\}_{k=0}^K$. 
\end{proof}

Consider a relaxed chronology of standard forces $\mathcal{F}=\{F^{(k)}\}_{k=1}^N$ for a standard zero forcing set $B$ of $G$.  The \emph{terminus} of $\mathcal{F}$, denoted $\term(\mathcal{F})$, is the set of vertices of $G$ that do not perform any forces in $\mathcal{F}$.  The \emph{reversal} of $\mathcal{F}$, denoted $\rev(\mathcal{F})$, is the result of reversing the forces and time-steps in $\mathcal{F}$, i.e., $\rev(\mathcal{F})=(F_{\rev}^{(k)})_{k=1}^K$ with $F_{\rev}^{(k)}=\{v\to u:u\to v\in F^{(K-k+1)}\}$.

The notion of path bundles induced by a vertex was also introduced in \cite{hogben2023new} and provides insight into the process of modifying one PSD forcing set into another, which will play a key role in PSD reconfiguration graphs.   To define path bundles induced by a vertex, let $B$ be a PSD forcing set of $G$, and $\mathcal{F}$ be a relaxed chronology of PSD forces of $B$ on $G$. Fix a vertex $x\in V(G)$.  For $k=0,1,2,\dots,\rd(x)$, define $C^{k}_x=\comp(G-E_{\mathcal F}^{[k]},x)$ to be the component of $G-E_{\mathcal F}^{[k]}$ containing $x$.  The {\em path bundle} of $\mathcal F$ induced by $x$ is the collection of paths $\mathcal Q=\{Q_i\}_{i=1}^{|B|}$ constructed as follows: Let $\mathcal Q^{[0]}$ be the paths given trivially by the vertices in $B$.  At each time-step $k$, consider only the forces $u \rightarrow v \in F^{(k)}$ such that at time-step $k$, $u$ is blue and $v \in V(C^k_x)$.  Append $v$ to the end of the path containing $u$ to form $\mathcal Q^{[k]}$. 
 After time-step $\rd(x)$, let $\mathcal Q=\mathcal Q^{[\rd(x)]}$.    The following theorem introduced in \cite{hogben2023new} identifies the specifics concerning the modification of one PSD forcing set into another via path bundles induced by a vertex.

\begin{theorem}[{\cite{hogben2023new}}]\label{reverse}
Let $G$ be a graph and $B$ be a PSD forcing set of $G$. Let $\mathcal{F}$ be a relaxed chronology of PSD forces for $B$ on $G$, and let $\mathcal{Q}$ be the path bundle of $\mathcal{F}$ induced by $x\in V(G)$. Then $\term(\mathcal{F}|_H)$ is a PSD forcing set of $G$. Furthermore, a relaxed chronology of PSD forces for $\term(\mathcal{F}|_H)$ can be constructed by reversing the forces between vertices in $H$ and preserving all remaining forces. 
\end{theorem}

\section{Universal Parameter Properties}
\label{universal}

Many characteristics of reconfiguration graphs can be derived based purely on properties common among many graph parameters, and in particular many parameters related to zero forcing (examples can be found in \Cref{TypesTable}). In this section, we define two classes into which many important graph parameters fall, as well as some properties of parameters which provide key information about the respective reconfiguration graphs.  We deduce some consequences of these classifications and properties for the reconfiguration graphs, thereby establishing some universal results for reconfiguration graphs. Separately, we determine an upper bound on the size of a maximum clique in the token sliding graph in terms of the clique number of the source graph.

\subsection{Parameter classes}\label{tablesection}

The standard zero forcing number \cite{barioli2013aim}, PSD forcing number \cite{param}, skew forcing number \cite{allison2010minimum}, $k$-forcing number \cite{kforce}, minor monotone floor of zero forcing number \cite{paramlong}, and power domination number \cite{haynes2002powerdom} (with the role of zero forcing evident in \cite{brueni2005pmu}) are each common zero forcing variants and thus central to our conversation.  Since the power domination number can also be considered a domination variant, the domination number, total domination number \cite{totald}, vertex cover number, and independence number, which are each common examples of domination variants (as vertex covers and maximal independent sets are each respectively dominating sets) are also closely related to the topic at hand.  In addition, study has been done examining the relationships between the zero forcing number and the chromatic number \cite{zfvschrom}, the independence number \cite{zfvsindep}, the vertex cover number \cite{zfvsvc}, the domination number \cite{zfvsdom}, and the total domination number \cite{zfvstotd}.  The path cover number is a commonly utilized lower bound on the zero forcing number \cite{param}. Similarly, the tree cover number is a commonly utilized lower bound on the PSD forcing number \cite{bozemantree}, arguments using spider covers are used in the work on power domination \cite{haynes2002powerdom}, and arguments using clique covers are common in the work on zero forcing \cite{barioli2013aim}.

Many graph parameters share certain characteristics that are fundamental to the structure of their reconfiguration graphs.
The first and most fundamental characteristic is whether a given parameter is defined based on properties possessed by subsets of the vertex set, by partitions of the vertex set, or by some other means.  We have already introduced vertex parameters (\Cref{VertexParameterDef}), which are graph parameters defined by subsets of the vertex set.  We establish results for vertex parameters in this section, and defer a discussion of partition parameters to \Cref{part_params}.

A common, and very useful, property of vertex parameters occurs when they operate independently across the connected components of a graph, in some sense.  We make this notion of operating independently rigorous in the following definition and refer to such vertex parameters as summable.  

\begin{definition}\label{def:SummableDefinition}
Let $X$ be a vertex parameter defined by property $x$.  If for every graph $G$,
\[B \subseteq V(G) \text{ has property }x \text{ in }G \]
\[\text{if and only if}\]
\[B= \bigcup_{C \in \comp(G)}B_C \text{, where each }B_C \subseteq V(C) \text{ has property }x \text{ in }C,\]
then $X$ is a \emph{summable vertex parameter}.
\end{definition}

The following result provides computational information about summable vertex parameters. 

\begin{lemma}\label{compsum}
If $X$ is a summable vertex parameter, then for every graph $G$, 
\[X(G)=\sum_{C \in \comp(G)}X(C).\]
\end{lemma}

\begin{proof}
    Let $G$ be a graph and let $C_1,C_2,\dots,C_k$ be the connected components of $G$. For each $B \subseteq V(G)$ with property $x$, by the fact that $x$ is a summable graph parameter, $B=\bigcup_{i=1}^{k}B_i$ for some $B_i \subseteq V(C_i)$ such that $B_i$ has property $x$ in $C_i$. 
\begin{case}
Suppose $X$ is defined by a maximum.  Then,
\[ X(G)=\max \sum_{i=1}^{k} |B_i| = \sum_{i=1}^{k} \max\{|B_i|:B_i\text{ has property } x \text{ in } C_i\} = \sum_{i=1}^{k} X(C_i ).\]
\end{case} 
\begin{case}
Suppose $X$ is defined by a minimum.  Then,
\[ X(G)=\min \sum_{i=1}^{k} |B_i| = \sum_{i=1}^{k} \min\{|B_i|:B_i\text{ has property } x \text{ in } C_i\} = \sum_{i=1}^{k} X(C_i).\qedhere\]
\end{case}
\end{proof}

The token exchange and token sliding reconfiguration graphs of a disconnected graph have a very nice Cartesian product structure if the parameter $X$ is summable. Note that the following theorem is a generalization of Theorem 4.3 in \cite{bjorkman2022power}, which considered parameters defined as the minimum cardinality of vertex sets with property $x$.

\begin{theorem}
\label{disjoint-gives-box}
Let $X$ be a summable vertex parameter, and let $G=\bigsqcup_{i=1}^N G_i$ be a graph.  Then $\GenTokExch(G)\cong\bigbox_{i=1}^N\GenTokExch(G_i)$ and $\GenTokSlide(G)\cong\bigbox_{i=1}^N\GenTokSlide(G_i)$.
\end{theorem}

\begin{proof}
We will assume that $X(G)$ is defined by a maximum; the case where it is defined by a minimum is similar.  By definition, $B\subseteq V(G)$ has property $x$ for $G$ if and only if each $B_i=B\cap V(G_i)$ has property $x$ for $G_i$.  If $B$ has maximum cardinality for $G$, then $X(G)=\sum_i X(G_i)\geq\sum_i|B_i|=|B|=X(G)$, so one must have that $|B_i|=X(G_i)$ for each $i$ (i.e., each $B_i$ has maximum cardinality for $G_i$).  Conversely, if each $B_i$ has maximum cardinality for $G_i$, $X(G)=\sum_i X(G_i)=\sum_i|B_i|=|B|$, so $|B|=X(G)$ and $B$ has maximum cardinality for $G$.  Thus, $V(\GenTokExch(G))=\phi(V(\bigbox_i\GenTokExch(G_i)))$, where $\phi((B_i)_i)=\bigcup_i B_i$ is a bijection between $V(\bigbox_i\GenTokExch(G_i))=\prod_i V(\GenTokExch(G_i))$ and $V(\GenTokExch(G))$ (note that $V(\GenTokExch(G^\star))=V(\GenTokSlide(G^\star))$ for \emph{any} graph $G^\star$).

For the following, $B^k$ (for some $k$) will be used to denote a vertex in $V(\GenTokExch(G))$ and $B^k_i$ (for some $k$) will be used to denote a vertex in $V(\GenTokExch(G_i))$, with the implicit relationships that $B^k=\bigcup_i B^k_i=\phi((B^k_i)_i)$ and $B^k_i=B^k\cap V(G_i)$.

Consider $(B^1_i)_i, (B^2_i)_i \in V(\bigbox_i\GenTokExch(G_i))$. Then $(B^1_i)_i(B^2_i)_i\in E(\bigbox_i\GenTokExch(G_i))$ if and only if there exist some $\ell$ and $p\neq q\in V(G_\ell)$ such that $B^1_\ell\setminus B^2_\ell=\{p\}$, $B^2_\ell\setminus B^1_\ell=\{q\}$, and $B^1_i=B^2_i$ for all $i\neq\ell$.  Thus, if $(B^1_i)_i(B^2_i)_i\in E(\bigbox_i\GenTokExch(G_i))$, then $B^1\setminus B^2=\{p\}$ and $B^2\setminus B^1=\{q\}$, and so $B^1B^2\in E(\GenTokExch(G))$.

Conversely, if $B^1B^2\in E(\GenTokExch(G))$, then $B^1\setminus B^2=\{p\}$ and $B^2\setminus B^1=\{q\}$ for some $p\neq q\in V(G)$.  The fact that $|B^1_i|=|B^2_i|=X(G_i)$ for all $i$ implies that there exists $\ell$ such that $p,q\in V(G_\ell)$ and $B^1_i=B^2_i$ for all $i\neq\ell$.  Thus, $(B^1_i)_i(B^2_i)_i\in E(\bigbox_i\GenTokExch(G_i))$.  Therefore, $\GenTokExch(G)\cong\bigbox_i\GenTokExch(G_i)$.

Now, $B^1B^2\in E(\GenTokSlide(G))$ if and only if $B^1B^2\in E(\GenTokExch(G))$ and $pq\in E(G)$ (letting $p$, $q$, and $\ell$ be as earlier).  This is true if and only if $B^1_i=B^2_i$ for all $i\neq\ell$, $B^1_\ell B^2_\ell\in E(\GenTokExch(G_\ell))$, and $pq\in E(G_\ell)$, since $G=\bigsqcup_i G_i$.  In turn, this is true if and only if $(B^1_i)_i(B^2_i)_i\in E(\bigbox_i\GenTokSlide(G_i))$.  Thus, $\GenTokSlide(G)\cong\bigbox_i\GenTokSlide(G_i)$.
\end{proof}

\begin{corollary}
Let $X$ be a summable vertex parameter such that $X(K_1)=1$.  Then for each pair of graphs $G$ and $H$, with $H \cong G \sqcup mK_1$ for some natural number $m$, $\GenTokSlide(G)\cong\GenTokSlide(H)$ and $\GenTokExch(G)\cong\GenTokExch(H)$.
\end{corollary}

We next define a property of some vertex parameters that, among other things, allows us to investigate the effect of vertex deletion on the reconfiguration graph of the induced subgraph.

\begin{definition}\label{coverable_def}
Let $X$ be a {vertex parameter} defined by property $x$.  If for every graph $G$ and every vertex $v \in V(G)$,
\[B \text{ has property }x \text{ in }G-v\]
\[\text{implies}\]
\[B \cup \{v\} \text{ has property }x \text{ in }G,\]
then $X$ is a \emph{coverable vertex parameter}.
\end{definition}

\begin{theorem}\label{covsubgraph}
Let $X$ be a coverable vertex parameter, let $G$ be a graph, and let $v \in V(G)$.  If $X(G)=X(G-v)+1$, then $\GenTokExch(G-v)\leq\GenTokExch(G)$ and $\GenTokSlide(G-v)\leq\GenTokSlide(G)$.
\end{theorem}
\begin{proof}
Let $x$ be the property defining $X$ and let $G$ be a graph such that $X(G)=X(G-v)+1$.  For each $B \in V(\GenTokSlide(G-v))=V(\GenTokExch(G-v))$, let $B'=B \cup \{v\}$.  Since $X$ is a coverable vertex parameter and $X(G)=X(G-v)+1$, it is clear that for distinct $B_1,B_2 \in V(\GenTokSlide(G-v))=V(\GenTokExch(G-v))$, we will have distinct $B_1',B_2' \in V(\GenTokSlide(G))=V(\GenTokExch(G))$. Furthermore, if $B_1B_2 \in E(\GenTokExch(G-v))$, then $B_1 \setminus B_2=\{u\}$ and $B_2 \setminus B_1=\{w\}$ for some pair of vertices $u,w \in V(G-v)$.  It follows that $B_1' \setminus B_2'=\{u\}$ and $B_2' \setminus B_1'=\{w\}$ and therefore $B_1'B_2' \in E(\GenTokExch(G))$.  Thus $\GenTokExch(G-v)\leq\GenTokExch(G)$.  

If $B_1B_2 \in E(\GenTokSlide(G))$, then we also have that $uw \in E(G-v)$, and it follows that $uw \in E(G)$.  Therefore $B_1'B_2' \in E(\GenTokSlide(G))$, and thus $\GenTokSlide(G-v)\leq\GenTokSlide(G)$.
\end{proof}

The following theorem shows that when vertex parameters are both summable and coverable, even more information can be obtained about their reconfiguration graphs.

\begin{theorem}\label{genconnectplus}
Let $X$ be a vertex parameter that is coverable and summable, let $G$ and $H$ be graphs, with $g_v \in V(G)$ and $h_v \in V(H)$, and let $G \oplus_v H$ be the result of identifying $g_v$ and $h_v$ in $G\sqcup H$.  If $X(G \oplus_v H)=X(G-g_v)+X(H-h_v)+1$, then 
\[\GenTokSlide(G-g_v) \cartProd \GenTokSlide(H-h_v) \cong \GenTokSlide((G \oplus_v H)-v)\leq\GenTokSlide(G \oplus_v H)\text{ and}\]
\[\GenTokExch(G-g_v) \cartProd \GenTokExch(H-h_v) \cong \GenTokExch((G \oplus_v H)-v) \leq \GenTokExch(G \oplus_v H).\]
\end{theorem}

\begin{proof}
    Let $x$ be the property defining $X$.  Since $X$ is a summable vertex parameter and $(G-g_v) \sqcup (H-h_v) \cong (G \oplus_v H)-v$, by \Cref{disjoint-gives-box}, $\GenTokSlide(G-g_v) \cartProd \GenTokSlide(H-h_v)\cong\GenTokSlide((G \oplus_v H)-v)$.  Furthermore, by \Cref{compsum}, we have that $X((G \oplus_v H)-v)=X((G-g_v) \sqcup (H-h_v))=X(G-g_v)+X(H-h_v)$.  Since $X$ is coverable and $X(G \oplus_v H)=X(G-g_v)+X(H-h_v)+1=X((G \oplus_v H)-v)+1$, by \Cref{covsubgraph}, $\GenTokSlide((G \oplus_v H)-v)\leq\GenTokSlide(G \oplus_v H)$.  A similar argument shows that $\GenTokExch(G-g_v) \cartProd \GenTokExch(H-h_v) \cong \GenTokExch((G \oplus_v H)-v) \leq \GenTokExch(G \oplus_v H)$.
\end{proof}

\subsection{Partition parameters}\label{part_params}
The majority of zero forcing variants and domination variants are vertex parameters, and in many cases they are defined by minima.  Note that there are exceptions, such as the independence number, which is defined by a maximum.  However, many important graph parameters are defined via partitions of the vertex set rather than subsets of the vertex set.  We refer to these graph parameters as partition parameters and establish some of their properties in this section.

\begin{definition}
Let $X$ be a graph parameter.  If there exists a property $x$ such that either:
\[\text{for every graph }G,\ X(G)=\max\{|\mathcal Y|: \mathcal Y \text{ is a partition of } V(G) \text{ and } \mathcal Y \text{ has property }x \text{ in }G\},\] 
\[\text{ or }\] 
\[\text{for every graph }G,\ X(G)=\min\{|\mathcal Y|: \mathcal Y \text{ is a partition of } V(G) \text{ and } \mathcal Y \text{ has property }x \text{ in }G\},\] 
then $X$ is a \emph{partition parameter} defined by property $x$.
\end{definition}

Perhaps the most frequently-studied partition parameter is the chromatic number of a graph. 
 However, partition parameters which are of particular interest to the study of zero forcing variants include path covers, tree covers, spider covers, and clique covers.  In each case, the partitions of interest are those for which the subgraphs induced by each partite set belong to a specific graph class.  We refer to these partition parameters as $\mathcal G$-cover numbers.

\begin{definition}
Let $\mathcal G$ be a class of graphs.  Let $G$ be a graph and $\mathcal Y$ be a partition of the vertices of $G$.  Then $\mathcal Y$ is a \emph{$\mathcal G$-cover} of $G$ if and only if for each $Y \in \mathcal Y$, $G[Y] \in \mathcal G$.  The \emph{$\mathcal G$-cover number} of $G$, denoted $\mathcal G(G)$, is the parameter
\[\mathcal G(G)=\min\{|\mathcal Y|: \mathcal Y \text{ is a partition of } V(G) \text{ and for each } Y \in \mathcal Y, G[Y] \in \mathcal G\}.\]
\end{definition}

For a graph $G$, partition parameters are defined by partitions of $V(G)$; reconfiguration rules that are defined instead on subsets of $V(G)$, such as token exchange, are not directly applicable to such parameters.  One resolution to this is to utilize transversals of the respective partitions to create functionally-equivalent vertex parameters.

\begin{definition}
Let $G$ be a graph and $\mathcal Y$ be a partition of $V(G)$.  $B \subseteq V(G)$ is a \emph{transversal} of $\mathcal Y$ in $G$ if $|Y \cap B|=1$ for each $Y \in \mathcal Y$.  If $B$ is a transversal of a partition $\mathcal Y$ then we denote this by $B \in \mathcal T(\mathcal Y)$.
\end{definition}

\begin{definition}
Let $X$ be a partition parameter defined by a property $x$. Then the \emph{transversal number} defined by $X$, denoted $X_\mathcal T$, is given by
\[X_{\mathcal T}(G)=\sigma\left(\{|B|: B \in \mathcal T(\mathcal Y) \text{ for some partition }\mathcal Y \text{ of } V(G) \text{ with property }x\}\right),\]
where the function $\sigma$ takes the maximum cardinality of the set if $X(G)$ is defined by a maximum and takes the minimum cardinality otherwise.
\end{definition}

\begin{observation}
Let $X$ be a partition parameter and $X_{\mathcal T}$ be the transversal number defined by $X$.  Then $X_{\mathcal T}$ is a vertex parameter, and for all graphs $G$, $X_{\mathcal T}(G)=X(G)$. 
\end{observation}

It has been shown for certain classes of graphs that the path cover number, denoted $\p(G)$, is equal to the standard zero forcing number and likewise that for many graphs the tree cover number, denoted $\T(G)$, is equal to the PSD forcing number \cite{barioli2013aim,bozemantree}.

\begin{lemma}[\cite{param}]\label{chainpath}
Let $G$ be a graph, $B$ be a standard zero forcing set of $G$, $\mathcal F$ be a chronological list of forces of $B$ on $G$, and let $\mathcal C$ be the chain set induced by $\mathcal F$.  Then $\mathcal C$ is a path cover of $G$.
\end{lemma}

\begin{lemma}[\cite{ekstrand2013psdzf}]\label{forcingtree}
Let $G$ be a graph, $B$ be a PSD forcing set of $G$, $\mathcal F$ be a chronological list of forces of $B$ on $G$, and let $\mathcal T$ be the collection of forcing trees induced by $\mathcal F$.  Then $\mathcal T$ is a tree cover of $G$.
\end{lemma}

\begin{proposition}
Let $G$ be a graph.
\begin{enumerate}
    \item If $\p(G)=\Z(G)$, then $\TokExch(G)\leq\RTokExch_{\p}(G)$ and $\TokSlide(G)\leq\RTokSlide_{\p}(G)$.
    \item If $\T(G)=\Z_+(G)$, then $\TokExch_+(G)\leq\RTokExch_{\T}(G)$ and $\TokSlide_+(G)\leq\RTokSlide_{\T}(G)$.
\end{enumerate} 
\end{proposition}

\begin{proof}
For (1), suppose $\p(G)=\Z(G)$ and let $B \in V(\TokExch(G))=V(\TokSlide(G))$, $\mathcal F$ be a chronological list of forces of $B$ on $G$, and $\mathcal C$ be the chain set induced by $\mathcal F$.  Then by \Cref{chainpath}, $\mathcal C$ is a path cover of $G$, and since $|\mathcal C|=|B|=\Z(G)=\p(G)$, $\mathcal C$ is a minimum path cover of $G$.  Furthermore, each forcing chain in $\mathcal C$ contains exactly one vertex in $B$, so $B$ is a transversal of $\mathcal C$ and $B \in V(\RTokExch_{\p}(G))=V(\RTokSlide_{\p}(G))$.  Since $V(\TokExch(G)) \subseteq V(\RTokExch_{\p}(G))$, given $B_1B_2 \in E(\TokExch(G))$, it follows that $B_1B_2 \in E(\RTokExch_{\p}(G))$, and so $\TokExch(G)\leq\RTokExch_{\p}(G)$.  Likewise, $\TokSlide(G)\leq\RTokSlide_{\p}(G)$.

For (2), the proof is analogous to that above but utilizes \Cref{forcingtree}.
\end{proof}

The next result is an example of the types of insights that reconfiguration graphs can provide for partition parameters. 

\begin{theorem}\label{thmpartbox}
Let $G$ be a graph, let $X$ be a partition parameter defined by property $x$, and let $\mathcal Y=\{Y_i\}_{i=1}^m$ be a partition of $V(G)$ having property $x$ such that $|\mathcal Y|=X(G)$.  Then the following are equivalent:
\begin{enumerate}
    \item $\GenTokExch(G) \cong \bigbox_{i=1}^m K_{|Y_i|}$;
    \item $\GenTokSlide(G) \cong \bigbox_{i=1}^m G[Y_i]$;
    \item $\mathcal Y$ is the unique partition of $V(G)$ having property $x$ such that $|\mathcal Y|=X(G)$.
\end{enumerate}
\end{theorem}

\noindent Before we turn to the proof of \Cref{thmpartbox}, we introduce a lemma which will be used for that proof.

\begin{lemma}\label{lempartbox}
    Let $G$ be a graph, let $X$ be a partition parameter defined by property $x$, and let $\mathcal Y=\{Y_i\}_{i=1}^m$ be a partition of $V(G)$ having property $x$ such that $|\mathcal Y|=X(G)$.  Then $\bigbox_{i=1}^m K_{|Y_i|}$ is isomorphic to a subgraph of $\GenTokExch(G)$ and $\bigbox_{i=1}^m G[Y_i]$ is isomorphic to a subgraph of $\GenTokSlide(G)$.
\end{lemma}

\begin{proof}
Let $V(\bigbox_{i=1}^m K_{|Y_i|})=V(\bigbox_{i=1}^m G[Y_i])=\mathcal V$, where $\mathcal V$ is the collection of ordered $m$-tuples $(v_{i,j_i})_{i=1}^m$ such that for each $i$, we have $v_{i,j_i} \in Y_i$.  So for each ordered $m$-tuple $(v_{i,j_i})_{i=1}^m$, the corresponding set of vertices $\{v_{i,j_i}\}_{i=1}^m$ is a transversal of $\mathcal Y$ and thus $\{v_{i,j_i}\}_{i=1}^m \in V(\GenTokExch(G))=V(\GenTokSlide(G))$.  Since $\mathcal Y=\{Y_i\}_{i=1}^m$ is a partition of $V(G)$, this correspondence is bijective.  If $(v_{i,j_i})_{i=1}^m(v_{i,k_i})_{i=1}^m \in E(\bigbox_{i=1}^m K_{|Y_i|})$, then $v_{i,j_i}=v_{i,k_i}$ for each $i \neq i_0$.  In this case, $\{v_{i,j_i}\}_{i=1}^m$ and $\{v_{i,k_i}\}_{i=1}^m$ are transversals of $\mathcal Y$ differing by an exchange, and thus $\{v_{i,j_i}\}_{i=1}^m\{v_{i,k_i}\}_{i=1}^m \in E(\GenTokExch(G))$.  It follows that $\bigbox_{i=1}^m K_{|Y_i|}$ is isomorphic to a subgraph of $\GenTokExch(G)$.  Similarly, if $(v_{i,j_i})_{i=1}^m(v_{i,k_i})_{i=1}^m \in E(\bigbox_{i=1}^m G[Y_i]$), then $v_{i,j_i}=v_{i,k_i}$ for each $i \neq i_0$ and $v_{i_0,j_{i_0}}v_{i_0,k_{i_0}} \in E(G)$.  In this case, $\{v_{i,j_i}\}_{i=1}^m$ and $\{v_{i,k_i}\}_{i=1}^m$ are transversals of $\mathcal Y$ differing by a slide, and thus $\{v_{i,j_i}\}_{i=1}^m\{v_{i,k_i}\}_{i=1}^m \in E(\GenTokSlide(G))$. 
 It follows that $\bigbox_{i=1}^m G[Y_i]$ is isomorphic to a subgraph of $\GenTokSlide(G)$.
\end{proof}

\begin{proof}[Proof of \Cref{thmpartbox}]
Let $\mathcal V=V(\bigbox_{i=1}^m K_{|Y_i|})=V(\bigbox_{i=1}^m G[Y_i]))$ be the collection of ordered $m$-tuples $(v_{i,j_i})_{i=1}^m$ such that for each $i$, we have $v_{i,j_i} \in Y_i$, and let $\phi(\mathcal V)$ be the collection of corresponding sets $\{v_{i,j_i}\}_{i=1}^m$.  Since $\mathcal Y=\{Y_i\}_{i=1}^m$ is a partition of $V(G)$, this correspondence is bijective.  Also note that by \Cref{lempartbox}, $\bigbox_{i=1}^m K_{|Y_i|}$ is isomorphic to a subgraph of $\GenTokExch(G)$ and $\bigbox_{i=1}^m G[Y_i]$ is isomorphic to a subgraph of $\GenTokSlide(G)$.

We first show, by way of the contrapositive, that (1) implies (3).  If there exists a partition $\mathcal Y'=\{Y'_i\}_{i=1}^m$ of $V(G)$ having property $x$, with $\mathcal Y' \neq \mathcal Y$ and $|\mathcal Y'|=X(G)$, then there exists a partite set $Y \in \mathcal Y$ and two distinct partite sets $Y_1',Y_2' \in \mathcal Y'$ such that $Y \cap Y_1'\neq \emptyset$ and $Y \cap Y_2' \neq \emptyset$.  Let $v'_i \in Y'_i$ for each $i$, with $v_1' \in Y \cap Y_1'$ and $v_2' \in Y \cap Y_2'$.  It follows that $\{v'_i\}_{i=1}^m$ is a transversal of $\mathcal Y'$.  However, since $v_1',v_2' \in Y \in \mathcal Y$, $\{v'_i\}_{i=1}^m \not \in \phi(\mathcal V)$, and thus $\bigbox_{i=1}^m K_{|Y_i|}$ is isomorphic to a proper subgraph of $\GenTokExch(G)$.  An identical argument shows that (2) implies (3).

We now show that (3) implies (2).  Since $\mathcal Y$ is the unique partition of $V(G)$ having property $x$ such that $|\mathcal Y|=X(G)$ and each $\{v_{i,j_i}\}_{i=1}^m \in V(\GenTokSlide(G))$ must be a transversal of this partition,  $\phi(\mathcal V)=V(\GenTokSlide(G))$.  Furthermore, if $\{v_{i,j_i}\}_{i=1}^m\{v_{i,k_i}\}_{i=1}^m \in E(\GenTokSlide(G))$, then the two vertices $u,v \in \{v_{i,j_i}\}_{i=1}^m\symmDiff\{v_{i,k_i}\}_{i=1}^m$ must be such that $u,v \in Y_{i_0} \in \mathcal Y$ for some $i_0$ and $uv \in E(G)$.  So $(v_{i,j_i})_{i=1}^m$ and $(v_{i,k_i})_{i=1}^m$ are such that $v_{i,j_i}=v_{i,k_i}$ for each $i\neq i_0$ and $v_{i_0,j_{i_0}}v_{i_0,k_{i_0}} \in E(G[Y_{i_0}])$.  Thus, $(v_{i,j_i})_{i=1}^m(v_{i,k_i})_{i=1}^m \in E(\bigbox_{i=1}^mG[Y_i])$ and so $\GenTokSlide(G)$ is isomorphic to a subgraph of $\bigbox_{i=1}^m G[Y_i]$. 
 In particular, since $\bigbox_{i=1}^m G[Y_i]$ is isomorphic to a subgraph of $\GenTokSlide(G)$ it follows that $\GenTokSlide(G) \cong \bigbox_{i=1}^m G[Y_i]$.  A similar argument shows that (3) implies (1).
\end{proof}

In \cite{bjorkman2022power, bong2022isomorphisms}, $X$-set parameters were defined and studied to derive what can be considered universal results for token addition and removal reconfiguration graphs.  Note that a parameter being summable is equivalent to the parameter satisfying criterion (3).

\begin{definition}\label{x_set_def}
An $X$-set parameter is a graph parameter defined to be the minimum cardinality of an $X$-set of $G$, where the $X$-sets of $G$ are subsets of $V(G)$ defined by a given property and satisfy the following conditions:
\begin{enumerate}
    \item If $B$ is an $X$-set of $G$ and $B \subseteq B' \subseteq V(G)$, then $B'$ is an $X$-set of $G$.
    \item The empty set is not an $X$-set of any graph.
    \item $B$ is an  $X$-set of a disconnected graph $G$ if and only if $B$ is the union of an $X$-set of each component of $G$. 
    \item If $G$ has no isolated vertices, then every set of $\left|V(G)\right|-1$ vertices is an $X$-set.
\end{enumerate}
\end{definition}

It is important to note that the skew forcing number is not an $X$-set parameter, since it fails criterion (2). 
 However, \Cref{SkewIsSummable} shows that the skew forcing number is summable.  As illustrated by \Cref{disjoint-gives-box} for example, there are important results which do not require the other three criteria and instead rely solely on summability, and thus apply to a larger family of vertex parameters. 

In \Cref{TypesTable}, we summarize which parameters satisfy the previously-established $X$-set property (\Cref{x_set_def}) and the summable (\Cref{def:SummableDefinition}) and coverable (\Cref{coverable_def}) properties defined herein for several important graph parameters.  While it is obvious that $X$-set parameters must be summable, none of the other implications among these properties hold.
Arguments supporting the conclusions summarized in this table can be found in Appendix \ref{table}.

\renewcommand{\arraystretch}{1.1}
\begin{table}[h!]
\centering
\begin{tabular}{|l|c|c|c|}
\hline
Vertex Parameter & Summable & Coverable & $X$-set Parameter \\
\hline
\hline
Zero Forcing Number & Yes & Yes & Yes\\
\hline
PSD Forcing Number & Yes & Yes & Yes\\
\hline
$k$-forcing Number & Yes & Yes & Yes\\
\hline
Power Domination Number & Yes & Yes & Yes\\
\hline
Domination Number & Yes & Yes & Yes\\
\hline
Path Cover Transversal Number & Yes & Yes & Yes\\
\hline
Spider Cover Transversal Number & Yes & Yes & Yes\\
\hline
Tree Cover Transversal Number & Yes & Yes & Yes\\
\hline
Clique Cover Transversal Number & Yes & Yes & Yes \\
\hline
Skew Forcing Number & Yes & Yes & {\bf No}\\
\hline
Vertex Cover Number & Yes & Yes & {\bf No}\\
\hline
Independence Number & Yes & {\bf No} & {\bf No}\\
\hline
Total Domination Number & Yes & {\bf No} & {\bf No}\\
\hline
Minor Monotone Floor of Zero Forcing Number & {\bf No} & Yes & {\bf No}\\
\hline
Chromatic Transversal Number & {\bf No} & Yes & {\bf No}\\
\hline
\end{tabular}
\caption{Vertex parameters and their characteristics}
\label{TypesTable}
\end{table}

\renewcommand{\arraystretch}{1}

\subsection{Cliques in reconfiguration graphs}
Some structural results on reconfiguration graphs do not require any special properties to be satisfied.  For instance, one can relate cliques in a token sliding graph to cliques in the source graph, using purely set-theoretic arguments.

\begin{lemma}
\label{SetDifferenceLemma}
Let $U$ be a set, and let $W\subseteq\mathcal{P}(U)$ ($|W|>1$; $\mathcal{P}(U)$ is the power set of $U$), such that for all $P,Q\in W$, $P\neq Q$, $|P\setminus Q|=|Q\setminus P|=1$.   Let $Y=\bigcap W$, $Z=(\bigcup W)\setminus Y$.  Then either $W=\{Y\cup\{z\}\colon z\in Z\}$ or $W=\{Y\cup (Z\setminus\{z\})\colon z\in Z\}$.
\end{lemma}
\begin{proof}
Fix $S_0\in W$, and let $W'=W\setminus\{S_0\}$.  Let $A=S_0\cap Z=S_0\setminus Y$, $B=Z\setminus S_0=Z\setminus A$.  For all $S\in W'$, $|S\cap B|=|S\setminus S_0|=1=|S_0\setminus S|=|A\setminus S|$.  Thus, for all $S\in W$, $|S\cap Z|=|A|-|A\setminus S|+|S\cap B|=|A|$ and $|Z\setminus S|=|B|-|S\cap B|+|A\setminus S|=|B|$.  $|A|\neq 0$ (otherwise $W=\{Y\}$); likewise, $|B|\neq 0$ (else $W=\{Y\cup Z\}$).  If $|B|=1$, then $W=\{Y\cup (Z\setminus\{z\})\colon z\in Z\}$, and we are done.  Otherwise, let $|B|\geq 2$.

It is asserted that if $P,Q\in W'$ and $P\cap Q\cap B=\emptyset$, then $P\cap A=Q\cap A$.  Since $|P\cap B|=|Q\cap B|=1$, $P\cap B=\{b_P\}$ and $Q\cap B=\{b_Q\}$ for some $b_P,b_Q\in B$; since $P\cap Q\cap B=\emptyset$, $b_P\neq b_Q$.  Likewise, $|A\setminus P|=|A\setminus Q|=1$, so $A\setminus P=\{a_P\}$ and $A\setminus Q=\{a_Q\}$ for some $a_P,a_Q\in A$.  $a_P=a_Q$, as otherwise $\{a_P,b_Q\}\subseteq Q\setminus P$ and $|Q\setminus P|>1$.  Thus, $P\cap A=A\setminus (A\setminus P)=A\setminus\{a_P\}=A\setminus\{a_Q\}=Q\cap A$.

Let $b_1,b_2\in B$, $b_1\neq b_2$.  There exist $S_1,S_2\in W'$ such that $b_1\in S_1$ and $b_2\in S_2$; since $|S_1\cap B|=|S_2\cap B|=1$, $S_1\cap B=\{b_1\}$ and $S_2\cap B=\{b_2\}$.  Since $S_1\cap S_2\cap B=\emptyset$, $S_1\cap A=S_2\cap A$.  For any $S\in W'$, $|S\cap B|=1$, so either $S\cap S_1\cap B=\emptyset$ or $S\cap S_2\cap B=\emptyset$.  Thus, for all $S\in W'$, $S\supseteq S\cap A=S_1\cap A=S_2\cap A$.  Therefore, $S_1\cap A\subseteq Y\cap Z=\emptyset$, and $|A|=|A\setminus S_1|+|A\cap S_1|=1$.  We then have that $W=\{Y\cup\{z\}\colon z\in Z\}$.
\end{proof}

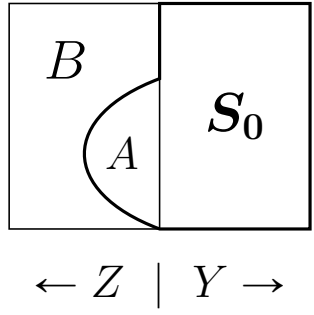
\begin{figure}[hbt]
\centering
\begin{tikzpicture}
\draw[very thick] (0,0) .. controls (-4/3,0.5) and (-4/3,1.5) .. (0,2) -- (0,3) -- (2,3) -- (2,0) -- cycle;
\draw[semithick] (0,0) -- (-2,0) -- (-2,3) -- (0,3) -- cycle;
\node[font=\LARGE] at (0,-0.75) {$\gets Z\ \ |\ \ Y\to$};
{ \boldmath \node[font=\huge] at (1,1.5) {$S_0$}; }
\node[font=\LARGE] at (-0.5,1) {$A$};
\node[font=\huge] at (-1.25,2.25) {$B$};
\end{tikzpicture}
\caption{Illustration of set structure in proof of \Cref{SetDifferenceLemma}.}
\end{figure}

\begin{corollary}
\label{TokExchCSubgraph}
Let $G$ be a graph, $H=\GenTokExch(G)$, and $W\subseteq V(H)$.  If $H[W]\cong K_{|W|}$ is a complete graph, then either $W=\{Y\cup\{z\}\colon z\in Z\}$ or $W=\{Y\cup (Z\setminus\{z\})\colon z\in Z\}$, where $Y=\bigcap W$ and $Z=(\bigcup W)\setminus Y$.
\end{corollary}

\begin{corollary}
\label{TokSlideCSubgraph}
Let $G$ be a graph, $H=\GenTokSlide(G)$, and $W\subseteq V(H)$, with $|W|>1$.  If $H[W]\cong K_{|W|}$ is a complete graph, then either $W=\{Y\cup\{z\}\colon z\in Z\}$ or $W=\{Y\cup (Z\setminus\{z\})\colon z\in Z\}$, where $Y=\bigcap W$ and $Z=(\bigcup W)\setminus Y$.  Moreover, $G[Z]\cong K_{|W|}$.
\end{corollary}

\begin{theorem} \label{thm:tsclique}
Let $G$ be a graph.  Then $\omega(G)\geq\omega(\GenTokSlide(G))$.
\end{theorem}

\section{Reconfiguration for Positive Semidefinite Forcing}\label{psd}
In this section we study reconfiguration graphs under token sliding and token exchange using the positive semidefinite (PSD) color change rule (defined in \Cref{zeroforcing}).  We begin by recalling some basic facts about PSD forcing.  We then consider trees and complete graphs and describe their reconfiguration graphs.  Note that in most common graph families there is no useful general description of the reconfiguration graphs, but in these cases there is.  We consider the realizability of trees and complete graphs. Next we consider a family of graphs that yields arbitrarily many connected components for both token sliding and token exchange graphs.  We begin with some well-known results for PSD forcing, which will be useful throughout.

\begin{theorem}[\cite{ekstrand2013psdzf}]
\label{MinPSDIncExc}
Let $G$ be a graph, and let $v\in V(G)$ be a vertex with at least one incident edge.  There exist minimum PSD forcing sets $A$, $B$ of $G$ such that $v\notin A$ and $v\in B$.
\end{theorem}

\begin{theorem}[\cite{hogben2022book}] \label{Z+=n-1}\label{Z+=1}
    Let $G$ be a graph on $n$ vertices.
\begin{enumerate}
    \item $\Z_+(G)=1$ if and only if $G$ is a tree.
    \item If $G$ has no isolated vertices, then $\Z_+(G)=n-1$ if and only if $G\cong K_n$.
\end{enumerate}
\end{theorem}

\begin{proposition}[\cite{peters2012positive}]\label{migrate}
 Let $G$ be a graph, let $B$ be a PSD forcing set of $G$, let $\mathcal F$ be a chronological list of forces, and let $u \rightarrow v$ be the first force in $\mathcal F$.  Then $B'=(B \cup \{v\})\setminus \{u\}$ is a PSD forcing set of $G$, with $v\to u$ being a valid first force for $B'$.
\end{proposition}

\begin{unnumberednote}
The cited proposition does not explicitly state the validity of $v\to u$ as an initial force for $B'$; however, this is the core of its proof.
\end{unnumberednote}

\begin{lemma}\label{PSDEdgeExistence}
Let $G$ be a graph.  If $G$ contains an edge, then $\ZTokExch_+(G)$ and $\ZTokSlide_+(G)$ both contain an edge.
\end{lemma}

\begin{proof}
Let $xy\in E(G)$.  By \Cref{MinPSDIncExc}, there exists a minimum PSD forcing set $B$ such that $x\notin B$.  Let $\mathcal{F}$ be a chronological list of forces for $B$ on $G$.  Since $x\notin B$, $\mathcal{F}$ cannot be empty; thus, it must have a first force $u\to v$.  By \Cref{migrate}, $B'=(B \cup \{v\})\setminus \{u\}$ is a minimum PSD forcing set.  It thus follows that $BB'\in\ZTokExch_+(G)$; since the validity of $u\to v$ requires $uv\in E(G)$, we also have that $BB'\in\ZTokSlide_+(G)$.
\end{proof}

The next two results are used in the proof of \Cref{noPSDteTree}, but are included here because they provide some information about PSD forcing sets that may be useful outside of the reconfiguration context.

\begin{lemma}\label{compsize}
Let $G$ be a connected graph of order $n \geq 3$ and $B$ be a minimum PSD forcing set of $G$.  If $G \not \cong K_n$, then either
\begin{enumerate}
    \item each component of $G-B$ has only one vertex, and there is a vertex $v \in B$ such that $v$ has at least two neighbors in $V(G-B)$, or
    \item at least one component $C$ of $G-B$ contains at least two vertices.
\end{enumerate}
\end{lemma}

\begin{proof}
We suppose by way of contradiction that each component of $G-B$ has only one vertex, and there is not a vertex $v \in B$ such that $v$ has at least two neighbors in $V(G-B)$.  First note that since $G$ is a connected graph of order $n$ with $G \not \cong K_n$, it follows from \Cref{Z+=n-1} that $\Z_+(G) \leq n-2$. Let $\{w_i\}_{i=1}^{|V(G-B)|}$ be the vertices of $G-B$.  Note that since $B$ is a minimum PSD forcing set of $G$, with $G$ connected and of order $n \geq 3$, each vertex in $B$ is adjacent to some vertex in $V(G-B)$.  As no vertex in $B$ is adjacent to more than one vertex in $V(G-B)$, the vertices of $B$ can be partitioned into the sets $\{N_G(w_i)\}_{i=1}^{|V(G-B)|}$.  No component of $G-B$ contains more than one vertex, so there are no edges between any vertices in the set $\{w_i\}_{i=1}^{|V(G-B)|}$.  Finally, since no vertex in $B$ is adjacent to more than one vertex in $V(G-B)$, but $G$ is connected and $|V(G-B)| \geq 2$, there exist vertices $v_a$ and $v_b$ such that $v_av_b \in E(G)$, $v_a \in N_G(w_a)$, and $v_b \in N_G(w_b)$, with $w_a \neq w_b$.  We will show, by contradiction, that $B'=B \setminus \{v_b\}$ is a PSD forcing set of $G$, and thus $B$ is not of minimum cardinality among the PSD forcing sets of $G$.  Since $G-B$ is a collection of isolated vertices and no vertex in $B$ is adjacent to more than one vertex in $V(G-B)$, it follows that $G-B'$ is a collection of isolated vertices and one component containing two vertices, specifically $v_b$ and $w_b$.  Since $v_aw_b \not \in E(G)$ and $v_av_b \in E(G)$, $v_b$ is the unique neighbor of $v_a$ in the component of $G-B'$ containing $v_b$.  Thus $v_a \rightarrow v_b$ is a valid initial force of $B'$.  Since $B$ is a PSD forcing set of $G$, it follows that $B'$ is a PSD forcing set of $G$.  Since $B' \subsetneq B$, it follows that $B$ cannot be a minimum PSD forcing set of $G$, a contradiction.
\end{proof}

\begin{lemma}\label{PSDinitial}
Let $G$ be a connected graph of order $n \geq 3$.  If $G \not \cong K_n$, then there exist a minimum PSD forcing set $B$ of $G$ and a vertex $v \in B$ such that $v \rightarrow w_1$ and $v \rightarrow w_2$ are both valid initial forces for $B$, with $w_1,w_2$ distinct vertices in $V(G) \setminus B$.
\end{lemma}

\begin{proof}
First note that since $G$ is a connected graph of order $n \geq 3$ and $G \not \cong K_n$, \Cref{compsize} applies.  Thus, for any minimum PSD forcing set $B$ of $G$, either each component of $G-B$ has only one vertex and there is a vertex $v \in B$ such that $v$ has at least two neighbors in $V(G-B)$, or at least one component $C$ of $G-B$ contains at least two vertices.
\begin{case}
     There is a minimum PSD forcing set $B$ of $G$ such that each component of $G-B$ has only one vertex, and there is a vertex $v \in B$ such that $v$ has at least two neighbors in $V(G-B)$.
\end{case}
    
    Choose two vertices $w_1,w_2 \in N_G(v) \cap V(G-B)$.  Since $w_1$ and $w_2$ are the only vertices in their respective components of $G-B$, it follows that both $v \rightarrow w_1$ and $v \rightarrow w_2$ are valid initial forces for $B$.
\begin{case}
There is a minimum PSD forcing set $B$ of $G$ such that at least one component $C$ of $G-B$ contains two vertices.
\end{case}

Let $\mathcal F=\{F^{(k)}\}_{k=1}^{n-|B|}$ be a chronological list of forces of $B$ on $G$ with expansion sequence $\{E_{\mathcal F}^{[k]}\}_{k=0}^{n-|B|}$.  Further, let $x$ be the last vertex of $C$ forced during $\mathcal F$, and let time-step $\ell$ be the last time-step at which the component $C_x^{\ell}$ of $G-E_{\mathcal F}^{[\ell]}$ containing $x$ contains at least two vertices.  Finally, let $u \in C^\ell_x$ such that some vertex $w$ forces $u$ during time-step $\ell+1$.  Since the removal of $u$ from $C_x^\ell$ (which is connected) makes $x$ an isolated vertex, $ux \in E(G)$.  Since $x$ is the only vertex in the component of $G-E_{\mathcal F}^{[\ell+1]}$ containing $x$, it follows that $u \rightarrow x$ is a valid force when $E_{\mathcal F}^{[\ell+1]}$ is blue.  Thus there exists a chronological list of forces $\hat{\mathcal F}$ of $B$ on $G$ with the first $\ell+2$ time-steps being such that $\{\hat{F}^{(k)}\}_{k=1}^{\ell+1}=\{F^{(k)}\}_{k=1}^{\ell+1}$ and $\hat{F}^{(\ell+2)}=\{u \rightarrow x\}$. Let $\mathcal Q_x$ be the path bundle induced by $\mathcal{\hat{F}}$ and $x$, and $\mathcal Q_u$ be the path bundle induced by $\mathcal{\hat{F}}$ and $u$.  By \Cref{reverse}, $B_x=\term(\mathcal{\hat{F}}|_{\mathcal Q_x})$ and $B_u=\term(\mathcal{\hat{F}}|_{\mathcal Q_u})$ are PSD forcing sets of $G$.  Further, by \Cref{reverse}, $u\to w$ is a valid initial force for $B_u$ and $x\to u$ is a valid initial force for $B_x$.  By \Cref{migrate}, $u\to x$ is a valid initial force for $B_u=(B_x \cup \{u\})\setminus \{x\}$, which completes the proof.
\end{proof}

\subsection{Reconfiguration of Trees}

Reconfiguration graphs of source graph families rarely have simple descriptions. However, under the PSD color change rule, the characterizations of the token exchange and token sliding graphs of trees are simple.  This stands in contrast to the corresponding results in \Cref{sec:SkTrees}.

\begin{theorem}\label{treesPSD}
    Let $T$ be a tree with $n$ vertices.  Then $\ZTokExch_+(T) \cong K_n$ and $\ZTokSlide_+(T)\cong T$.
\end{theorem}
\begin{proof}
%    Since $\Z_+(T)=1$ by \Cref{Z+=n-1}, 
    It is well-known that each singleton set is a minimum PSD forcing set of a tree, and the first result is immediate. Under the token sliding reconfiguration rule, the edges of $\ZTokSlide_+(T)$ are therefore precisely the edges of $T$. Therefore, the second result follows. 
\end{proof}

We now consider the realizability of trees as token exchange graphs. It was shown in \cite{geneson2023reconfiguration} that $\ZTokExch(G) \ncong K_{1,r}$ for any $r\geq 2$.  Furthermore, one can construct graphs $G$ such that $\ZTokExch(G) \cong P_r$ for any $r\neq 3$.  The realizability of other trees under the standard color change rule is unresolved. However, no trees of order at least three are realizable as token exchange graphs under the PSD color change rule, as \Cref{noPSDteTree} shows.

\begin{theorem}
\label{noPSDteTree}
    Let $G$ be a graph such that either
    \begin{enumerate}
         \item $G$ has more than one nontrivial component, or
        \item $G$ has a component containing at least three vertices.
    \end{enumerate}
    Then $\ZTokExch_+(G) \not \cong T$ for any tree $T$.
\end{theorem}

\begin{proof}
     First, suppose $G$ has more than one nontrivial component. Consider two such nontrivial components $C$ and $D$.  By \Cref{PSDEdgeExistence}, $\ZTokExch_+(C)$ and $\ZTokExch_+(D)$ each contain an edge. 
 Furthermore, by \Cref{def:SummableDefinition} and \Cref{disjoint-gives-box}, $\ZTokExch_+(C) \cartProd \ZTokExch_+(D)$ is isomorphic to a subgraph of $\ZTokExch_+(G)$.  Thus, $\ZTokExch_+(G)$ contains a 4-cycle and cannot be a tree. Now, assume that $G$ has a nontrivial component $C$ containing at least three vertices. We proceed with cases on the structure of $C$.
    \begin{case}
    $C \cong K_{|V(C)|}$
    \end{case}
   By \Cref{corcomplete}, $\ZTokExch_+(C) \cong K_{|V(C)|}$.  By \Cref{disjoint-gives-box}, $\ZTokExch_+(C)$ is isomorphic to a subgraph of $\ZTokExch_+(G)$.  Since $C$ contains at least 3 vertices, $\ZTokExch_+(G)$ contains a clique of size 3 and therefore cannot be a tree.

    \begin{case}
    $C \not \cong K_{|V(C)|}$
    \end{case}
    By \Cref{PSDinitial} there exist a minimum PSD forcing set $B$ of $C$ and a vertex $v \in B$ such that $v \rightarrow u_1$ and $v \rightarrow u_2$ are both valid initial forces for $B$, with $u_1,u_2$ being distinct vertices in $V(C) \setminus B$.  Thus by \Cref{migrate}, $B_1=(B \cup \{u_1\})\setminus \{v\}$ and $B_2=(B_1 \cup \{u_2\})\setminus \{v\}$ are each minimum PSD forcing sets of $C$.  Since $B$, $B_1$, and $B_2$ each differ by exactly one vertex, they form a 3-cycle in $\ZTokExch_+(C)$.  Since $\ZTokExch_+(C)$ is isomorphic to a subgraph of $\ZTokExch_+(G)$ by \Cref{disjoint-gives-box}, it follows that $\ZTokExch_+(G)$ contains a 3-cycle and cannot be a tree.
\end{proof}

We note that the corresponding result for token sliding graphs does not hold; for example $\ZTokSlide_+(K_{2,3})\cong K_{1,6}$ is a tree.

\subsection{Reconfiguration of Complete Graphs}
Complete graphs are another natural family of graphs whose reconfiguration graphs are easily described.  \Cref{Z+=n-1} allows us to easily deduce the structure of the reconfiguration graph for a complete source graph.

\begin{theorem}\label{psdCompleteTETS}
$\ZTokExch_+(K_n) = \ZTokSlide_+(K_n) \cong K_n$.
\end{theorem}
\begin{proof}
     It is well-known that every set $B \subseteq V(K_n)$ containing $n-1$ vertices is a minimum PSD forcing set of $K_n$; the first result is then immediate. Under the token sliding reconfiguration rule, the edges of $\ZTokSlide_+(K_n)$ are therefore precisely the edges of $K_n$. Therefore, the second result follows. 
\end{proof}

\begin{theorem}
\label{CompleteTokExch}
Let $G$ be a graph with no isolated vertices, such that $\ZTokExch_+(G)\cong K_n$.  Then either $G\cong K_n$, or $G$ is a tree on $n$ vertices.
\end{theorem}

\begin{proof}
Let  $H=\ZTokExch_+(G)$. By \Cref{TokExchCSubgraph}, either (a) $V(H)=\{S_0\cup\{v\}\colon v\in Z\}$ ($S_0\cap Z=\emptyset$) or (b) $V(H)=\{S_0\cup (Z\setminus\{v\})\colon v\in Z\}$ ($S_0\cap Z=\emptyset$).  It also follows that $|Z|=|V(H)|$.  By \Cref{MinPSDIncExc}, for all $v\in V(G)$, there exist minimum PSD forcing sets $A_v,B_v\in V(H)$ such that $v\notin A_v$ and $v\in B_v$.  Thus, for all $v\in V(G)$, $v\notin S_0$; thus $S_0=\emptyset$.  Additionally, in both (a) and (b), $v\in Z$ for all $v\in V(G)$, so $Z=V(G)$.  Thus, $n=|V(H)|=|Z|=|V(G)|$, and either $\Z_+(G)=1$ (in case (a)) or $\Z_+(G)=n-1$ (in case (b)).  Application of \Cref{Z+=n-1} then completes the proof.
\end{proof}

\begin{corollary}\label{corcomplete}
Let $G$ be a graph with no isolated vertices, such that $\ZTokSlide_+(G)\cong K_n$.  Then $G\cong K_n$.
\end{corollary}

\begin{proof}
Let $H=\ZTokSlide_+(G)$. By \Cref{TS<=TE}, $\ZTokExch_+(G)\cong K_n$, so \Cref{CompleteTokExch} implies that either $G\cong K_n$ or $G$ is a tree on $n$ vertices.  \Cref{TokSlideCSubgraph} implies that $K_n\leq G$, so it must be that $G\cong K_n$.
\end{proof}

\begin{proposition}
    If $G$ is connected, then $\ZTokExch_+(G)=\ZTokSlide_+(G)$ if and only if $G \cong K_n$ for some $n$.
\end{proposition}
\begin{proof}
    Suppose $G$ is connected and $G\not\cong K_n$.  Then \Cref{PSDinitial} applies, so there is a PSD zero forcing set $B$ and $v\in B$ such that $v$ can be chosen to force two distinct vertices $u_1$ and $u_2$ in the first time step.  Note that $u_1u_2 \not\in E(G)$.  Then $B\setminus\{v\}\cup\{u_1\}$ and $B\setminus\{v\}\cup\{u_2\}$ are both valid PSD zero forcing sets whose corresponding vertices are adjacent in $\ZTokExch_+(G)$ but not in $\ZTokSlide_+(G)$.  The other direction is immediate from \Cref{psdCompleteTETS}.
\end{proof}

\subsection{Disconnected reconfiguration graphs}

We note that there are many examples of graphs with connected reconfiguration graphs under the token exchange rule. These include trees, cycles, grid graphs, wheels, and many more.  An exhaustive search shows that the smallest graph $G$ with $\ZTokExch_+(G)$ being disconnected is $G=K_{4,4}$.  This leads to a family of graphs that shows that there are reconfiguration graphs with arbitrarily many components, namely, certain complete multipartite graphs.

\begin{proposition}[{\cite{peters2012positive}}]
\label{CompleteMultiZPlus}
For $n_1\geq n_2 \geq\dots \geq n_k$, $\Z_+(K_{n_1,n_2,\dots,n_k})=n_2+n_3+\dots+n_k$.
\end{proposition}

\begin{lemma}\label{completemultiZplusSig}
    Fix an integer $q$ greater than or equal to 4, and let $G = K_{q,q,\dots,q}$, the complete $r$-partite graph on $rq$ vertices.  If $B$ is a minimum PSD forcing set of $G$, then the complement of $B$ in $V(G)$ intersects at most two of the partite sets, and at most one of the intersections can contain more than one vertex. 
\end{lemma}
\begin{proof}
  For $i \in \{ 1, 2, \dots, r\}$, let $S_i$ denote the $i$-th partite set of $V(G)$.  Let $B$ be a minimum PSD forcing set of $G$  (see \Cref{K444}) and denote by $W$ the complement of $B$ in $V(G)$.  Then $W$ can intersect at most two partite sets, since otherwise there would be an induced $K_3$ in $W$ and every $v \in B$ would be adjacent to at least two of the vertices in this induced $K_3$.  If $W$ intersects two partite sets $S_i$ and $S_j$ and $|W \cap S_i| \geq 2$ and $|W \cap S_j |\geq 2$, then $W$ contains an induced $C_4$, and each vertex of $B$ is adjacent to two vertices of this induced $C_4$, thus no forces can take place.
\end{proof}

\begin{theorem}
    Let $q$ be a fixed integer greater than or equal to 4, and consider $G = K_{q,q,\dots,q}$, the complete $r$-partite graph on $rq$ vertices.  Then $\ZTokSlide_+(G)$ and $\ZTokExch_+(G)$ both contain exactly $r$ connected components.
\end{theorem}
\begin{proof}
Let $S_i$ denote the $i$th partite set of $G$.  By \Cref{CompleteMultiZPlus}, all minimum PSD forcing sets $B$ of $G$ satisfy $|B|=q(r-1)$, and $|V(G)\setminus B|=q$.  On the other hand, by \Cref{completemultiZplusSig}, for any minimum PSD forcing set $B$ of $G$, $|(V(G)\setminus B)\cap S_i|>1$ for at most one value of $i$.  Thus, each minimum PSD forcing set $B$ of $G$ satisfies $|(V(G)\setminus B)\cap S_i|\geq q-1$ for exactly one value of $i$.  Let $\mathcal{B}_i$ be the set of minimum PSD forcing sets $B$ of $G$ satisfying $|(V(G)\setminus B)\cap S_i|\geq q-1$.  Let $B_i=V(G)\setminus S_i$.  For each $B\in\mathcal{B}_i$, either $B=B_i$ or $B=(B_i\cup\{v\})\setminus\{w\}$ for some $v\in S_i$ and $w\in V(G)\setminus S_i=N(v)$.  Thus, for all $B\in\mathcal{B}_i$ with $B\neq B_i$, $BB_i\in E(\ZTokSlide_+(G))\subseteq E(\ZTokExch_+(G))$.  Thus, $\ZTokExch_+(G)[\mathcal{B}_i]$ and $\ZTokSlide_+(G)[\mathcal{B}_i]$ are connected and non-empty for each $i$.  Now for $i \neq j$, let $B_1\in\mathcal{B}_i$, $B_2\in\mathcal{B}_j$.  Since $q\geq 4$ and $|B_1\symmDiff B_2|\geq 2q-4$, we have that $B_1B_2\notin E(\ZTokExch_+(G))\supseteq E(\ZTokSlide_+(G))$.  Thus, the sets in each $\mathcal{B}_i$ are in different connected components of $\ZTokExch_+(G)$ and of $\ZTokSlide_+(G)$.
\end{proof}

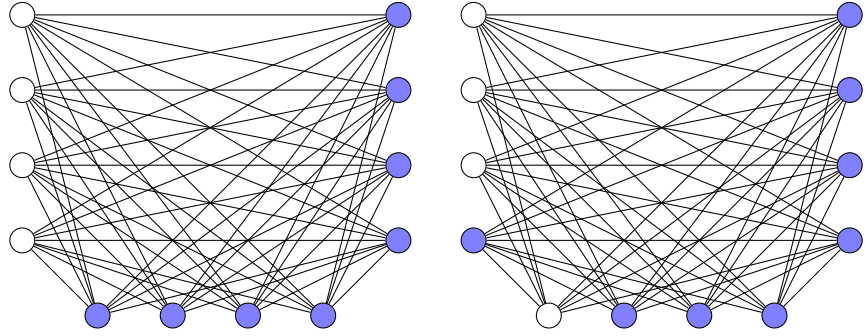
\begin{figure}[ht]
\centering
\begin{tikzpicture}[scale=1, auto,swap]

\foreach \i in {1,2,3,4}
  \node[vertex_alt] (v\i) at (0,-\i) { };

\foreach \i in {5,6,7,8}
  \node[vertex_alt,fill=blue!50] (v\i) at (\i-4,-5) { };

  \foreach \i in {9,10,11,12}
    \node[vertex_alt,fill=blue!50] (v\i) at (5,\i-13) { };

\foreach \i in {1,2,3,4}
  \foreach \j in {5,6,7,8,9,10,11,12}
    \draw (v\i) -- (v\j);
\foreach \i in {5,6,7,8}
  \foreach \j in {9,10,11,12}
    \draw (v\i) -- (v\j);

\begin{scope}[shift={(6,0)}]  
\foreach \i in {1,2,3}
  \node[vertex_alt] (v\i) at (0,-\i) { };
  \node[vertex_alt,fill=blue!50] (v4) at (0,-4) { };

\foreach \i in {6,7,8}
  \node[vertex_alt,fill=blue!50] (v\i) at (\i-4,-5) { };
    \node[vertex_alt] (v5) at (1,-5) { };
  \foreach \i in {9,10,11,12}
    \node[vertex_alt,fill=blue!50] (v\i) at (5,\i-13) { };

\foreach \i in {1,2,3,4}
  \foreach \j in {5,6,7,8,9,10,11,12}
    \draw (v\i) -- (v\j);
\foreach \i in {5,6,7,8}
  \foreach \j in {9,10,11,12}
    \draw (v\i) -- (v\j);
\end{scope}

\end{tikzpicture}
\caption{Up to symmetry, there are two minimum PSD forcing sets for $K_{4,4,4}$.}

\label{K444}
\end{figure}

We remark that there are many constructions which lead to arbitrarily many connected components in PSD token sliding and token exchange graphs.  The uniform complete multipartite graphs are but one example of source graphs which have this property.

\section{Reconfiguration for Skew Forcing}
\label{skew}

In this section we study reconfiguration graphs under token sliding and token exchange using the skew color change rule (defined in \Cref{zeroforcing}).  We begin by recalling some basic facts about skew forcing.  We then consider paths, cycles, and complete graphs, and describe their reconfiguration graphs.  We also consider various aspects of the reconfiguration graphs of trees, and realizability of complete graphs.

The concept of a matching in a graph is very useful for describing skew forcing sets. Recall that a \emph{matching} in a graph $G$ is a subset $M\subseteq E(G)$ of the set of edges such that no vertex is incident to more than one edge in $M$.  The vertices involved in $M$ are called \emph{$M$-saturated}, and the vertices that do not show up in $M$ are called \emph{$M$-unsaturated}.  Furthermore, a \emph{perfect matching} is a matching $M$ where all vertices are $M$-saturated, and a \emph{maximum matching} is a matching with the most possible edges for $G$.  The number of edges in a maximum matching in a graph $G$ is the \emph{matching number} of $G$, denoted $\match(G)$.

The following results demonstrate the connections between matchings and skew forcing sets.

\begin{proposition}[\cite{dealba2014some}]\label{prop:skew0}
    Let $G$ be a connected graph on at least three vertices. If $\Z_-(G)=0$, then $G$ has a unique perfect matching, $|V(G)|$ is even, and $G$ contains at least one leaf. 
\end{proposition}

\begin{proposition}[\cite{dealba2014some}] \label{prop:skewmultipartite}
If $G$ is a connected graph on $n$ vertices, then either $\Z_-(G) = n-2$ or $\Z_-(G)\leq n-4$. Furthermore, $\Z_-(G)=n-2$ if and only if $G$ is a complete multipartite graph $K_{n_1,n_2,\dots,n_s}$ with $s\geq 2$ and $n_i \geq 1$ for all $i$.
\end{proposition}

Let $B^{\emptyset}_-(G)$ be the set of vertices colored blue by applying the skew forcing color change rule to $\emptyset \subseteq V(G)$ until no further forces can be performed (i.e., $B^{\emptyset}_-(G)=\cl_-(G,\emptyset)$ is the skew closure of the empty set), and let $W^{\emptyset}_-(G)=V(G) \setminus B^{\emptyset}_-(G)$.

\begin{proposition}
   Let $G$ be a graph and $\mathbb F$ be the collection of skew forts contained in $G$.  Then $\bigcup \mathbb F=W^{\emptyset}_-(G)$.
\end{proposition}

\begin{proof}
By \Cref{skewderivedset}, $v \in V(G)$ will become blue during a forcing process starting from $\emptyset$ if and only if $v \not \in \bigcup \mathbb F$.
\end{proof}

\begin{observation}
Any minimum skew forcing set $B$ of a graph $G$ satisfies $B \subseteq W^{\emptyset}_-(G)$.
\end{observation}

\begin{theorem}[\cite{hogben2016fractional}]\label{thm:leafstrip}  If $G$ is a graph with leaf $u$ and $v$ is the neighbor of $u$, then $\Z_-(G-\{u,v\}) = \Z_-(G)$.
\end{theorem}

The leaf-stripping algorithm, Algorithm 1 from \cite{hogben2016fractional}, is an application of \Cref{thm:leafstrip} which starts with a graph $G$ and, if $G$ contains any leaves, finds a smaller graph $\widehat{G}$ with the same skew forcing number. It is as follows: \\
\hrule
\ \\
\noindent \textbf{Input:}  Graph $G$ \\
\noindent \textbf{Output:} Graph $\widehat{G}$ with $\delta(\widehat{G}) \neq 1$ or $V(\widehat{G}) = \emptyset$ \\
$\widehat{G} := G$ \\
\textbf{while} $\widehat{G}$ has a leaf $u$ with neighbor $v$ \textbf{do} \\
$|$ \  $\widehat{G} := \widehat{G}-\{u,v\}$ \\
\textbf{end} \\
\textbf{return} $\widehat{G}$ \\
\hrule
\ \\

The leaf-stripping algorithm applies nontrivially to any graph with a leaf, but it allows the direct computation of a minimum skew forcing set for a tree, as we will see in Section \ref{sec:SkTrees}.  The following {observation} is part of Remark 3.23 in \cite{hogben2016fractional}.

\begin{observation}
\label{LeafStripping:TreeToPK_1}
    If $G = T$ is a tree, then $\widehat{G} \cong pK_1$ for some $p\geq 0$ and $V(\widehat{G})$ is a skew forcing set for $T$, thus $\Z_-(T) =|V(\widehat{G})|$.
\end{observation}

\subsection{Paths and Cycles}

We label the vertices of $P_n$ from the set $\{ 1,2, \dots, n\}$ in path order starting from a leaf.  For $C_n$, we adopt the same convention, but any vertex can be the start of the labeling.  The parity of the vertices using this labeling plays an important role in succinctly describing minimum skew forcing sets. 
 Paths are interesting examples for skew forcing because, unlike standard zero forcing and PSD forcing, there is the possibility for empty skew forcing sets.  The skew forcing numbers for $P_n$ are $\Z_-(P_n)=1$ when $n$ is odd and $\Z_-(P_n)=0$ when $n$ is even.  When $n$ is odd, a minimum skew forcing set for $P_n$ consists of any single vertex with an odd label. %Further, for $n$ odd, numbering the vertices in order along the path starting with a degree one vertex, a minimum skew forcing set consists of a single odd vertex. 

\begin{theorem} \label{thm:skewpaths}
     Let $P_n$ be the path graph on $n$ vertices.
     \begin{enumerate}
         \item When $n$ is even, $\ZTokExch_-(P_n) = \ZTokSlide_-(P_n) \cong K_1$.
         \item When $n$ is odd, $\ZTokExch_-(P_n) \cong K_{(n+1)/2}$ and $\ZTokSlide_-(P_n) \cong \frac{n+1}{2}K_1$.
     \end{enumerate}
\end{theorem} 

\begin{proof}
    When $n$ is even, the unique minimum skew forcing set of $P_n$ is the empty set. When $n$ is odd, the even vertices will be forced by a white vertex. Since any single odd vertex forms a minimum skew forcing set and there are $(n+1)/2$ odd vertices, $\ZTokExch_-(P_n) \cong K_{(n+1)/2}$. Exchanging an odd vertex with its even neighbor results in a set that does not force the entire graph, and thus there are no edges in the token sliding graph.
\end{proof}

Cycles also yield reconfiguration graphs whose structure is determined by the parity of the order of the graph.  

\begin{theorem} \label{thm:skewcycles}
    Let $C_n$ be the cycle graph on $n$ vertices.
    \begin{enumerate}
        \item When $n$ is odd, $\ZTokExch_-(C_n) \cong K_n$  and $\ZTokSlide_-(C_n) \cong C_n$.
        \item When $n$ is even, $\ZTokExch_-(C_n) \cong K_{n/2} \cartProd K_{n/2}$ and $\ZTokSlide_-(C_n) \cong \frac{n^2}{4}K_1$.
    \end{enumerate}
\end{theorem}
\begin{proof}
    When $n$ is odd, $\Z_-(C_n)=1$ with the minimum skew forcing sets consisting of any single vertex. Two singleton sets are connected by an edge in the token sliding graph if and only if the vertices they contain are neighbors in $C_n$. For $n$ even, $\Z_-(C_n)=2$ and there are $(\frac{n}{2})^2$ minimum skew forcing sets of the form $\{a,b\}$ with $a$ odd and $b$ even.  Distinct vertices $\{a_1,b_1\}$ and $\{a_2,b_2\}$ are connected in $\ZTokExch_-(C_n)$ if and only if $a_1=a_2$ or $b_1=b_2$, giving the desired Cartesian product. Exchanging a vertex with its neighbor results in a set consisting of either two odd vertices or two even vertices. None of these sets are skew forcing sets. Hence, there are no edges in $\ZTokSlide_-(C_n)$ when $n$ is even.
\end{proof}

\subsection{Complete Graphs}

Recall that for the complete graph $K_n$, we have $\Z_-(K_n) = n-2$. Notice that the token sliding graph is the same as the token exchange graph since all vertices are neighbors. Any subset of  size $n-2$ is a minimum skew forcing set for $K_n$.  So the vertices of $\ZTokExch_-(K_n)$ are the $(n-2)$-subsets and two vertices are adjacent exactly when their intersection is of size $n-3$.  This is the definition of the Johnson graph $J(n,n-2)$. See \cite{alggraph} for more details on Johnson graphs.

\begin{observation}
\label{SkewOfKn}
    $\ZTokExch_-(K_n) = \ZTokSlide_-(K_n) = J(n,n-2)$.
\end{observation}

We now characterize graphs for which $\ZTokExch_-(G) \cong K_m$ for some $m > 1$.

\begin{lemma} \label{skewTEcomplete}
Let $G$ be a graph such that $\Z_-(G) = 1$ and there exist exactly $m$ singleton sets that are skew forcing sets for $G$. Then, $\ZTokExch_-(G)\cong K_m$.
\end{lemma}

\begin{proposition}
     \label{skewstar}
 Let $G = K_{1,m}$ for $m \geq 3$. Then $\ZTokExch_-(G)\cong K_m$.  %%consistency of G \cong K_{1,m} or G = K_{1,m} 
\end{proposition}

\begin{proof}
For $G=K_{1,m}$, the minimum skew forcing sets are all sets of size $m-1$ from the $m$ vertices in the second partite set.  So we get $\ZTokExch_-(G)\cong K_m$.
\end{proof}

\begin{proposition}\label{prop:skewTEisKn}
    Let $G$ be a graph on $n$ vertices. Then $\ZTokExch_-(G) \cong K_n$ if and only if $\Z_-(G)=1$ and each singleton set is a skew forcing set of $G$.
\end{proposition}

\begin{proof}
    First note that by \Cref{skewTEcomplete}, if $\Z_-(G)=1$ and each singleton set is a skew forcing set of $G$, then $\ZTokExch_-(G) \cong K_n$.  Next, suppose $\ZTokExch_-(G) \cong K_n$ and let $W=V(\ZTokExch_-(G))$. By \Cref{TokExchCSubgraph}, there exist sets $Y,Z \subseteq V(G)$ such that either $W=\{Y\cup\{z\}\colon z\in Z\}$ or $W=\{Y\cup (Z\setminus\{z\})\colon z\in Z\}$, where $Y=\bigcap W$ and $Z=(\bigcup W)\setminus Y$. Since $|W|=n$, $|Y|=0$ and $Z=V(G)$. Thus, either each minimum skew forcing set contains only a single vertex, or all sets of size $n-1$ are minimum skew forcing sets. The latter is impossible since $\Z_-(G) \neq n-1$ for any graph. Therefore, $G$ is a graph in which each singleton set is a skew forcing set.
\end{proof}

\Cref{thm:skewcycles} shows that odd cycles have the property that any singleton set is a skew forcing set. The next example is an infinite family of graphs with this property.

\begin{example}\label{trianglesTEisKn}
Given a positive integer $r>1$, let $F_r=K_1 \vee rK_2$ be the friendship graph. Then $\Z_-(F_r)=1$ and any singleton set is a skew forcing set. Label the $K_1$ vertex as $c$. If $c$ is blue, then each copy of $K_2$ may force itself blue by a white vertex force followed by a standard force since each vertex in $K_2$ is the single white neighbor of the other vertex. If one vertex $v$ that is not the center vertex is colored blue, then its neighbor $w$ has one white neighbor $c$, and hence $w$ forces $c$ to turn blue. Then the rest of the graph may be forced as before.
\end{example}

\begin{proposition}\label{prop:skewTSnotKn}
    Let $G$ be a graph on $n$ vertices. Then $\ZTokSlide_-(G) \ncong K_n$.
\end{proposition}

\begin{proof}
    Suppose $\ZTokSlide_-(G) \cong K_n$. By \Cref{thm:tsclique}, $\omega(G) \geq n$, implying that $G\cong K_n$. However, the token sliding graph of $K_n$ is not isomorphic to $K_n$, creating a contradiction.
\end{proof}

\subsection{Disconnected skew reconfiguration graphs}

We now discuss source graphs whose {token exchange} graphs under the skew forcing rule are disconnected.  These results will be useful when we compare reconfiguration across forcing rules in \Cref{compare}.
%\subsubsection{Complete basket graphs}

\begin{definition}
    For $n\geq 4$ and $2\leq k\leq\frac{n}{2}$, the  \emph{$k$-th complete basket graph} on $n+2$ vertices, denoted $KB(n,k)$, is the graph with $K_n$ as an induced subgraph consisting of $\{v_1,\ldots, v_n\}$ and two more vertices $h_1$, $h_2$ with $h_1$ adjacent to each of $\{v_1,\ldots, v_k,h_2\}$ and $h_2$ adjacent to $\{v_{k+1}, \ldots, v_n,h_1\}$.
\end{definition}

\begin{figure}
    \centering
\begin{tikzpicture}[graph styles 1,scale=0.65]
\node["$h_1$",vertex] (h1) at (-2.6,3.6) {};
\node["$h_2$",vertex] (h2) at (2.6,3.6) {} edge (h1);
\graph [clockwise,phase=238,radius=13mm] {
subgraph K_n [V={1,...,7},/tikz/every node/.style={node font=\footnotesize,circle,draw,inner sep=1pt},typeset={$v_\tikzgraphnodename$}];
\foreach \x in {1,...,3} { \x --[looseness/.evaluated={1.55-0.3*\x-0.02*\x*\x},bend left/.evaluated={66-20*\x+\x*\x}] (h1) };
\foreach \x in {4,...,7} { \x --[looseness/.evaluated={0.25+0.04*\x+0.017*\x*\x},bend right/.evaluated={\x*\x+7*\x-35}] (h2) };
};
\end{tikzpicture}
\caption{The graph $KB(7,3)$}
\end{figure}
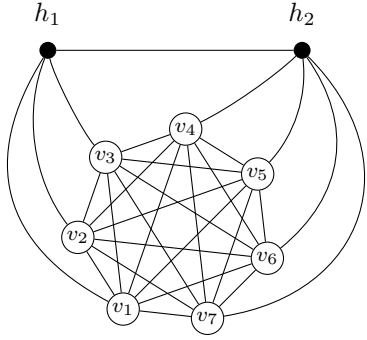

\begin{proposition}

The token exchange graph  $\ZTokExch_-(KB(n,k))$ has two connected components.
\end{proposition}
\begin{proof}
Each minimum skew forcing set must contain either $h_1$ or $h_2${, but} never both.  The skew forcing number is $n-2$.  The minimum skew zero forcing sets for $KB(n,k)$ are exactly the subsets of $V(KB(n,k))$ of the form 
\begin{equation*}\label{KB1}
\{a_1,\ldots, a_k,b_1,\ldots, b_{n-k-2}\}
\end{equation*}
where $a_i \in N(h_1)$ and $b_i \in N(h_2)\setminus\{h_1\},$ or
\begin{equation*}\label{KB2}
\{a_1,\ldots, a_{k-2},b_1,\ldots, b_{n-k}\} 
\end{equation*}
where $a_i \in N(h_1)\setminus\{h_2\}$ and $b_i \in N(h_2)$.
\end{proof}

It is noteworthy that the smallest complete basket graph, $KB(4,2)$ (shown in \Cref{TableGraphs}), is actually the smallest graph with a disconnected skew forcing token exchange graph.

\subsection{Trees}\label{sec:SkTrees}

In this section we use the connections between skew forcing and matchings to determine the structure of token exchange graphs for trees. The leaf-stripping algorithm applied to the tree $T$ outputs a graph $\widehat{T}$ which, by \Cref{LeafStripping:TreeToPK_1}, is either empty or a graph with no edges. The set $V(\widehat{T})$ is a skew forcing set for $\widehat{T}$ and also for $T$. Moreover, the graph $T-\widehat{T}$ has a perfect matching (note that a forest cannot have more than one perfect matching, so it is unique). This is the crux of the proof of the following {proposition}.

\begin{proposition}[\cite{dealba2014some}]\label{perfect_matching}  If $T$ is a tree and $B\subseteq V(T)$, then $B$ is a minimum skew forcing set for $T$ if and only if $T-B$ has a unique perfect matching, which is a maximum matching for $T$.
\end{proposition}

\begin{corollary}\label{ZMinusVsMatchingNumber}
If $T$ is a tree, then $\Z_-(T)=V(T)-2\match(T)$.
\end{corollary}

Applying the leaf-stripping algorithm to a tree $T$ involves choices of the order of the leaves processed. Unless the tree has a perfect matching, different choices result in different minimum skew forcing sets. To reduce the complexity of identifying minimum skew forcing sets, we define the following subgraph.  

\begin{definition}
The \emph{skew-nontrivial subgraph} $\widecheck{G}$ of a graph $G$ is the graph obtained from $G$ in the following manner:

\begin{itemize}
    \item Apply the skew forcing rule on $G$ until no further forces can be performed, with the initial set being the empty set.
    \item If a blue vertex $v$ has only blue vertices as its neighbors, delete it.
    \item Also, delete each edge that has only blue vertices as endpoints.
    \item The resulting subgraph is $\widecheck{G}$.
\end{itemize}      
\end{definition}

The result is a (possibly disconnected) subgraph that determines the token exchange graph of $G$.  Note that the skew-nontrivial subgraph of a tree is equivalent to the core of a tree defined in \cite{hp67}.  If $\widecheck{G} = G$, we say that $G$ is skew-nontrivial.

Skew-nontrivial subgraphs have implications for the reconfiguration of graphs in general, as the next result shows.  However, they are especially useful for analyzing reconfiguration graphs of trees, which will be their primary use in this article.

\begin{theorem}\label{thm:FundThmOfSkewNontrivial}
For any graph $G$, $\ZTokExch_-(\widecheck{G})=\ZTokExch_-(G)$ and $\ZTokSlide_-(\widecheck{G})=\ZTokSlide_-(G)$.
\end{theorem}

\begin{proof}
Let $B^{\emptyset}_-(G)$ be the set of vertices colored blue by skew forcing with the empty set.  Note that if a vertex $v \in B^{\emptyset}_-(G)$ has all blue neighbors, it plays no role in the rest of the skew forcing process. Also, an edge for which both incident vertices are in $B^{\emptyset}_-(G)$ does not affect the remainder of the skew forcing process. Therefore, $S$ is a minimum skew forcing set for $G$ if and only if $S$ is a minimum skew forcing set for the skew-nontrivial subgraph $\widecheck{G}$.  Further, the only edges removed to obtain the skew-nontrivial subgraph are between vertices that are not part of any minimum skew forcing set of $G$, so their removal cannot affect any potential token sliding.  Thus, $\ZTokExch_-(\widecheck{G})=\ZTokExch_-(G)$ and $\ZTokSlide_-(\widecheck{G})=\ZTokSlide_-(G)$.
\end{proof}

\begin{corollary}
\label{nontrivtree}
Given a graph $G$ with $\Zs(G) \geq 1$, let $\widecheck{G} = \bigsqcup_{i=1}^k S_i$, where the $S_i$ are the components of the skew-nontrivial subgraph of $G$. Then 
\[ \ZTokExch_-(G) \cong \bigbox_{i=1}^k \ZTokExch_-(S_i). \]    
\end{corollary}
\begin{proof}
 By applying \Cref{disjoint-gives-box} to the components $S_1, S_2, \dots, S_k$ of $\widecheck{G}$, the skew token exchange graph has the given Cartesian product form.
\end{proof}

\begin{lemma}\label{graphnontriv}
Suppose $G$ is a bipartite graph with a bipartition $\{U_1,U_2\}$ such that every leaf of $G$ is contained in $U_1$.  Then $G$ is skew-nontrivial. 
\end{lemma}

\begin{proof}

Following the skew color change rule, let $\emptyset$ be chosen to be the initial set of blue vertices and assume that after time-step $k \geq 0$, no vertex in $U_1$ is blue.  Note that since no vertex in $U_2$ is a leaf and vertices in $U_2$ are only adjacent to vertices in $U_1$, each member of $U_2$ that is not an isolated vertex has more than one white neighbor. During time-step $k+1$, since no vertices in $U_1$ are adjacent to each other and every vertex in $U_2$ has more than one white neighbor, no vertices in $U_1$ will become blue.  Continuing in this manner, when no further forces are possible, vertices in $U_2$ may be blue, but every vertex in $U_1$ will be white.  Since no isolated vertex can be forced and no vertices in $U_2$ are adjacent to each other, no vertices or edges would be deleted in the construction of $\widecheck{G}$  and $G$ is skew-nontrivial.
\end{proof}

\begin{proposition}\label{whitedge}
Let $T$ be a tree.
\begin{enumerate}
    \item $T[W^{\emptyset}_-(T)]$ is either empty or a graph with no edges.
    \item $T$ is skew-nontrivial if and only if the coloring generated by applying the skew forcing rule to the empty set is a proper $2$-coloring, that is, $\{B_-^{\emptyset}(T), W^{\emptyset}_-(T)\}$ is a bipartition of $T$.
\end{enumerate}
\end{proposition}

\begin{proof}
We proceed by contradiction: assume that $E(T[W^{\emptyset}_-(T)])\neq\emptyset$.  $T[W^{\emptyset}_-(T)]\leq T$, so $T[W^{\emptyset}_-(T)]$ is a forest.  Since $E(T[W^{\emptyset}_-(T)])\neq\emptyset$, $T[W^{\emptyset}_-(T)]$ must contain a leaf $u$ with neighbor $v$.  However, then $u\to v$ is a valid skew force for $B_-^\emptyset(T)$, contradicting the definition of $B_-^\emptyset(T)$.

Now, if $T$ is skew-nontrivial, then after applying the skew forcing rule to the empty set, no vertices in $B_-^{\emptyset}(T)$ are adjacent to each other.  By the first part, no vertices in $W^{\emptyset}_-(T)$ are adjacent each other, and thus $\{B_-^{\emptyset}(T),W^{\emptyset}_-(T)\}$ is a bipartition of $T$.  Conversely, if applying the skew forcing rule to the empty set generates a proper 2-coloring of $T$, then $T=\widecheck{T}$.
\end{proof}

\begin{lemma}
\label{BnullDegree}
Let $T$ be a skew-nontrivial tree, and let $u\in\Bnull(T)$.  Then $\deg(u)\geq 2$.
\end{lemma}

\begin{proof}
Let $u\in V(T)$.  If $\deg(u)=0$, then $u\notin\Bnull(T)$, since no vertex can ever force it.  If $u$ has a single neighbor $w$, then $u\to w$ is a valid skew force for $\emptyset$ in $T$.  Then $w\in\Bnull(T)$, and since $\{\Bnull(T),W^{\emptyset}_-(T)\}$ is a bipartition of $T$ by \Cref{whitedge}, $u\notin\Bnull(T)$.
\end{proof}

\begin{proposition}\label{treenontriv}
A tree is skew-nontrivial if and only if the distance between each pair of leaves is even.
\end{proposition}
\begin{proof}
If the distance between each pair of leaves is even, then by \Cref{graphnontriv}, the tree is skew-nontrivial.  On the other hand, suppose that not every pair of leaves is an even distance apart, and let $\{U_1,U_2\}$ be a bipartition of $T$.  Then, applying the skew forcing rule to the empty set, after the first time-step there are blue vertices in both $U_1$ and $U_2$.  Thus by \Cref{whitedge}, $T$ is not skew-nontrivial.  
\end{proof}

\begin{theorem}
\label{treesSkew}
    Let $T$ be a tree. Then $\ZTokExch_-(T)$ is connected.
\end{theorem}

\begin{proof}
    Let $S_1=\{a_1,\ldots, a_k\}$ and $S_2=\{b_1,\ldots,b_k\}$ be two minimum skew forcing sets of $T$.  Let $M_i$ be the maximum matching of $T$ such that $S_i$ is the set of $M_i$-unsaturated vertices.  Let $H$ be the subgraph of $T$ whose edge set is $M_1 \symmDiff M_2$ and whose vertex set is the set of vertices incident to those edges.

    We claim that $H$ consists of even length alternating paths. To see that $H$ must consist only of disjoint paths, note that a vertex in $H$ is at most incident to one edge in $M_1$ and one edge in $M_2$, hence its degree is one or two. The paths must be alternating because matchings only allow disjoint edges. If there exists an odd alternating path component $P$ in $H$, then without loss of generality assume that the edges at each end of $P$ are in $M_1$. 
Since the endpoints $x$ and $y$ of $P$ are $M_1$-saturated, any edges which they are incident to outside of $P$ cannot be in $M_1$.  Furthermore, said other edges cannot be in $M_2$, either, since this would contradict the maximality of $P$.  Hence, $P$ is an augmenting path for $M_2$, implying that $M_2$ is not maximal by Berge's Theorem. Thus all path components in $H$ are of even length. 

    Note that the endpoints of the paths in $H$ are either $M_1$-unsaturated or $M_2$-unsaturated and hence either in $S_1$ or in $S_2$.  Since each path is even, one endpoint of each path in $H$ is in $S_1\setminus S_2$ and the other is in $S_2 \setminus S_1$. Without loss of generality let us assume that either $a_i =b_i$ or $a_i$ and $b_i$ are ends of the same $(M_1,M_2)$-alternating path in $H$.  To see that $S_1$ and $S_2$ are connected in $\ZTokExch_-(T)$ we can first exchange $a_1$ with $b_1$, if they are distinct, by shifting the matching on the path between them from the edges in $M_1$ to the edges in $M_2$, which reveals the edge between $\{a_1,a_2,\ldots, a_k\}$ and $\{b_1,a_2,\ldots,a_k\}$.  We can do this in turn for each pair $(a_i,b_i)$ until $S_1$ has been reconfigured into $S_2$.
\end{proof}

Since the token exchange graph of a forest is the Cartesian product of the token exchange graphs of its tree components we have the following {corollary}.

\begin{corollary}
The token exchange graph of a forest is connected.
\end{corollary}

\begin{corollary}
\label{SNTconnectedTE}
Let $T$ be a tree.  The induced subgraph of $\ZTokExch_-(T)$ of minimum skew forcing sets that contain only leaves is connected.
\end{corollary}

\begin{proof}
This follows from the proof of \Cref{treesSkew} by letting $S_1$ and $S_2$ consist only of leaves of $T$, and observing that each intermediate set $\{b_1,\dots,b_i,a_{i+1},\dots a_k\}$ between $S_1$ and $S_2$ also consists only of leaves of $T$.
\end{proof}

The following theorem is in direct contrast to \Cref{PSDEdgeExistence} and \Cref{treesPSD}.

\begin{theorem}\label{TslideIsolate}
    The skew token sliding graph $\ZTokSlide_-(T)$ of a tree $T$ contains no edges. 
\end{theorem}

\begin{proof}
    Let $T$ be a tree and suppose that $S_1$ and $S_2$ are minimum skew forcing sets of $T$ that are adjacent in $\ZTokSlide_-(T)$. Let $M_1$ and $M_2$ be the perfect matchings of $V(T)\setminus S_1$ and $V(T)\setminus S_2$, respectively. Let $H$ be the subgraph of $T$ induced by the symmetric difference of $M_1$ and $M_2$. Since {$S_1S_2 \in E(\ZTokSlide_-(T))$}, there exist vertices $v_1$, $v_2$ such that $S_1 \setminus S_2 = \{v_1\}$ and $S_2 \setminus S_1 = \{v_2\}$. By the proof of \Cref{treesSkew}, $H$ is one path from $v_1$ to $v_2$, which is a single edge. This contradicts the fact that $H$ consists of only even-length paths. Therefore, there are no edges in $\ZTokSlide_-(T)$.  
\end{proof}

\begin{lemma}
\label{SpecialForest}
Let $F$ be a forest.  If $\{A,B\}$ is a bipartition of $V(F)$ such that $\deg(b)\geq 2$ for all $b\in B$, then $\match(F)=|B|$.
\end{lemma}

\begin{proof}
We proceed by induction on $|B|$.  If $|B|=0$, then $F$ has no edges, and $\match(F)=0$.  Otherwise, $F$ contains an edge, and therefore a leaf.  Let $v$ be a leaf of $F$, and let $w$ be its neighbor.  $\deg(v)=1$, so $v\in A$ and $w\in B$.  $F'=F-v-w$ is a forest with bipartition $\{A'=A\setminus\{v\},B'=B\setminus\{w\}\}$, and for each $b\in B'$, $\deg_{F'}(b)=\deg_F(b)\geq 2$.  Thus, by the inductive hypothesis, $\match(F')=|B'|=|B|-1$.  Letting $M'$ be a maximum matching of $F'$, we have that $M=M'\cup\{uv\}$ is a matching of $F$ with $|M|=|B|$.  Since any edge in $F$ must be between $A$ and $B$, $\match(F)\leq |B|$, and the proof is complete.
\end{proof}

\begin{corollary}
\label{SkewNontrivialTreeZ-}
Let $T$ be a skew-nontrivial tree.  $\Z_-(T)=|V(T)|-2|\Bnull(T)|$.
\end{corollary}

\begin{proof}
$\{W^{\emptyset}_-(T),\Bnull(T)\}$ is a bipartition of $V(T)$ by \Cref{whitedge}; by \Cref{BnullDegree}, $\deg(b)\geq 2$ for each $b\in\Bnull(T)$.  Thus, by \Cref{SpecialForest}, $\match(T)=|\Bnull(T)|$.  Finally, by \Cref{ZMinusVsMatchingNumber}, $\Z_-(T)=|V(T)|-2|\Bnull(T)|$.
\end{proof}

\section{Connections between PSD and skew reconfiguration graphs}\label{connections}

In this section we consider two different occurrences where there exist connections between PSD and skew reconfiguration graphs.  First, we explore the connection between the skew reconfiguration graphs of skew-nontrivial trees and the PSD reconfiguration graphs of the distance-2 graphs induced by the set of white vertices in a given tree.  Afterwards, we consider independent duplication and join-duplication, and find that independent duplication has a strikingly similar effect on skew reconfiguration to that which join-duplication has on PSD reconfiguration.  Before proceeding, we make a brief observation concerning the reconfiguration graphs for standard, PSD, and skew forcing.

\begin{observation}
    Standard zero forcing sets are also both PSD forcing sets and skew forcing sets, thus $\Z_+(G) \leq \Z(G)$ and $\Z_-(G) \leq \Z(G)$. As a result,
    \begin{itemize}
        \item if $\Z_+(G)=\Z(G)$, then $\ZTokExch(G)$ is an induced subgraph of $\ZTokExch_+(G)$ and $\ZTokSlide(G)$ is an induced subgraph of $\ZTokSlide_+(G)$, and
        \item if $\Z_-(G)=\Z(G)$, then $\ZTokExch(G)$ is an induced subgraph of $\ZTokExch_-(G)$ and $\ZTokSlide(G)$ is an induced subgraph of $\ZTokSlide_-(G)$.
    \end{itemize}
\end{observation}

\subsection{Skew-nontrivial trees and PSD forcing}

We now establish a connection between reconfiguration graphs for PSD and skew forcing when the source graph is a skew-nontrivial tree.

\begin{definition}
Let $T$ be a skew-nontrivial tree.  The \emph{special distance-2 graph} of $T$, defined as \[\SDTwo(T)=D_2(T)[\Wnull(T)]=\bigl(\Wnull(T),\{uw\colon u,w\in \Wnull(T),\,N_T(u)\cap N_T(w)\neq\emptyset\}\bigr),\] is the induced subgraph of the distance-2 graph of $T$ on the set of vertices not skew forced by the empty set.  (See \Cref{SNT:Figure}.)  We also define $\phi_-(uw)\colon E(\SDTwo(T))\to V(T)$ by $\{\phi_-(uw)\}=N_T(u)\cap N_T(w)$ (this is well-defined since $T$ is a tree, and any two vertices in $T$ can have at most one common neighbor).
\end{definition}

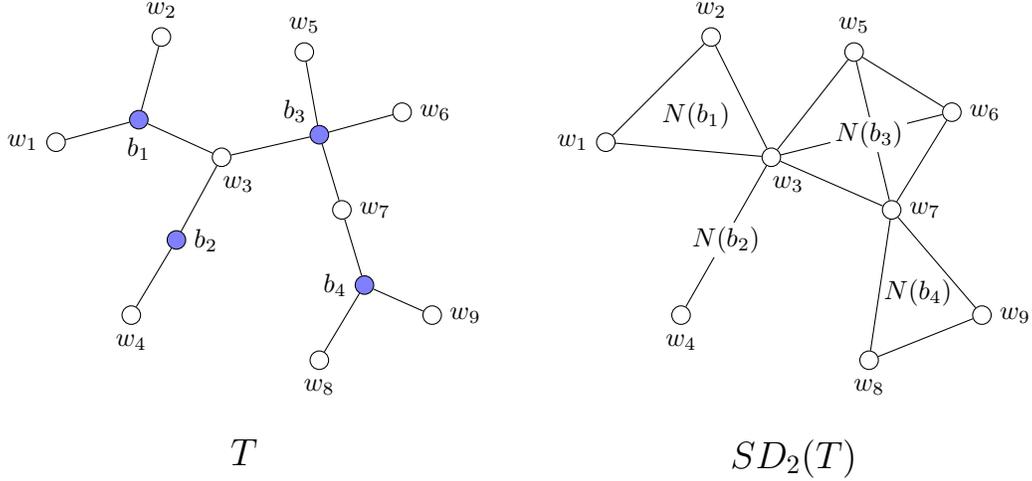
\begin{figure}
\centering
\begin{tikzpicture}[every matrix/.style={graph styles 2,every node/.style={vertex,minimum size=7pt}},blue vertex/.style={fill=blue!50},free text node/.style={text node,inner sep=1.5pt,draw=none}]
\matrix[column sep=0.4in] (illustration) {
\begin{scope}
\node["$w_1$" left] (w1) at  (-2.4, 1.0) {};
\node["$w_2$"] (w2) at       (-1.0, 2.4) {};
\node["-87:$w_3$"] (w3) at   (-0.2, 0.8) {};
\node["$w_4$" below] (w4) at (-1.4,-1.3) {};
\node["$w_5$"] (w5) at       ( 0.9, 2.2) {};
\node["$w_6$" right] (w6) at ( 2.2, 1.4) {};
\node["$w_7$" right] (w7) at ( 1.4, 0.1) {};
\node["$w_8$" below] (w8) at ( 1.1,-1.9) {};
\node["$w_9$" right] (w9) at ( 2.6,-1.3) {};
\node["$b_1$" below,blue vertex] (b1) at      (-1.3, 1.3) {} edge (w1) edge (w2) edge (w3);
\node["$b_2$" right,blue vertex] (b2) at      (-0.8,-0.3) {} edge (w3) edge (w4);
\node["$b_3$" above left,blue vertex] (b3) at ( 1.1, 1.1) {} edge (w3) edge (w5) edge (w6) edge (w7);
\node["$b_4$" left,blue vertex] (b4) at       ( 1.7,-0.9) {} edge (w7) edge (w8) edge (w9);
\node[caption node,below=16pt of current bounding box.south] {$T$};
\end{scope}
&
\begin{scope}
\node["$w_1$" left] (w1) at  (-2.4, 1.0) {};
\node["$w_2$"] (w2) at       (-1.0, 2.4) {} edge (w1);
\node["-87:$w_3$"] (w3) at   (-0.2, 0.8) {} edge (w1) edge (w2);
\node["$w_4$" below] (w4) at (-1.4,-1.3) {} edge (w3);
\node["$w_5$"] (w5) at       ( 0.9, 2.2) {} edge (w3);
\node["$w_6$" right] (w6) at ( 2.2, 1.4) {} edge (w3) edge (w5);
\node["$w_7$" right] (w7) at ( 1.4, 0.1) {} edge (w3) edge (w5) edge (w6);
\node["$w_8$" below] (w8) at ( 1.1,-1.9) {} edge (w7);
\node["$w_9$" right] (w9) at ( 2.6,-1.3) {} edge (w7) edge (w8);
\node[free text node] (b1) at            (-1.2,1.35) {$N(b_1)$};
\node[free text node,fill=white] (b2) at (-0.8,-0.3) {$N(b_2)$};
\node[free text node,fill=white] (b3) at ( 1.1, 1.1) {$N(b_3)$};
\node[free text node] (b4) at            (1.75,-1.0) {$N(b_4)$};
\node[caption node,below=16pt of current bounding box.south] {$\SDTwo(T)$};
\end{scope}
\\
};
\end{tikzpicture}
\caption{A skew-nontrivial tree $T$ and its special distance-2 graph $\SDTwo(T)$.}
\label{SNT:Figure}
\end{figure}

\begin{observation}
If $T$ is a skew-nontrivial tree, then by \Cref{whitedge}, $\{\Bnull(T),\Wnull(T)\}$ is a bipartition of $T$, and it follows that $\phi_-$ actually maps into $\Bnull(T)$.
\end{observation}

\begin{observation}
$\SDTwo(T)$ can also be created by taking each vertex $b$ in $\Bnull(T)$, deleting it and its edges, and adding an edge between every pair of distinct neighbors of $b$, so that $\SDTwo(T)[N_T(b)]\cong K_{|N_T(b)|}$ is a clique in $\SDTwo(T)$.
\end{observation}

\begin{observation}
For each $u\in\Bnull(T)$ and $w\in N_T(u)$, $N_T(u)\subseteq N_{\SDTwo(T)}[w]$.
\end{observation}

\begin{theorem}
\label{NonTrivTree:FS}
Let $T$ be a skew-nontrivial tree, and let $S\subseteq \Wnull(T)$.  Then $S$ is a PSD forcing set of $\SDTwo(T)$ if and only if $S$ is a skew forcing set of $T$.
\end{theorem}

\begin{proof}
\begingroup

For convenience, we let $H=\SDTwo(T)$ throughout this proof.

Let $S$ be a PSD forcing set of $H$.  $\Bnull(T)$ can be skew forced in $T$ starting from $\emptyset\subseteq S$; it thus suffices to prove that $S'=S\cup\Bnull(T)$ skew forces $T$.  Let $\mathcal{F}=\{u_i\to w_i\}_i$ be a chronological list of PSD forces for $S$ in $H$ with expansion sequence $\{E^{[i]}_{\mathcal{F}}\}_i$.  It is asserted that $\widetilde{\mathcal{F}}=\{b_i=\phi_-(u_iw_i)\to w_i\}_i$ is a chronological list of standard (and thus skew) forces for $S'$ in $T$ with expansion sequence $\{\widetilde{E}^{[i]}_{\widetilde{\mathcal{F}}}=E^{[i]}_{\mathcal{F}}\cup\Bnull(T)\}_i$.  Obviously, the expansion sequence is correct with respect to the initial set of blue vertices and chosen forces; it only remains to prove that the chosen forces are valid.  For each $i$, $\comp(H-E^{[i-1]}_{\mathcal{F}},w_i)\cap N_H(u_i)=\comp(H-E^{[i-1]}_{\mathcal{F}},w_i)\cap N_H[u_i]=\{w_i\}$.  We have that $w_i\in N_T(b_i)\setminus \widetilde{E}^{[i-1]}_{\widetilde{\mathcal{F}}}\subseteq (N_H[w_i]\setminus E^{[i-1]}_{\mathcal{F}})\cap N_H[u_i]\subseteq\comp(H-E^{[i-1]}_{\mathcal{F}},w_i)\cap N_H[u_i]=\{w_i\}$ and $b_i\in\Bnull(T)\subseteq \widetilde{E}^{[i-1]}_{\widetilde{\mathcal{F}}}$, so $b_i\to w_i$ is a valid standard force for $\widetilde{E}^{[i-1]}_{\widetilde{\mathcal{F}}}$ in $T$.

Now, let $S$ be a skew forcing set of $T$.  Then $S'=S\cup\Bnull(T)$ is also a skew forcing set of $T$.  Let $\mathcal{F}=\{b_i\to w_i\}_i$ be a chronological list of skew forces for $S'$ in $T$ with expansion set $\{E^{[i]}_{\mathcal{F}}\}_i$.  For all $i$, $w_i\in V(T)\setminus E^{[i-1]}_{\mathcal{F}}\subseteq\Wnull(T)$, and so $b_i\in N_T(\Wnull(T))=\Bnull(T)$.  Now, for each $i$, pick $u_i\in N_T(b_i)\setminus\{w_i\}$ (such a vertex exists by \Cref{BnullDegree}).  It is asserted that $\widetilde{\mathcal{F}}=\{u_i\to w_i\}_i$ is a chronological list of PSD forces for $S$ in $H$ with expansion sequence $\{\widetilde{E}^{[i]}_{\widetilde{\mathcal{F}}}=E^{[i]}_{\mathcal{F}}\setminus\Bnull(T)\}_i$.  Again, it is easy to verify that the expansion sequence is correct for the initial set of blue vertices and chosen forces; we only need to prove that the chosen forces are valid.

Therefore, we fix $i$, and consider whether $u_i\to w_i$ is a valid PSD force in $H$ for $\widetilde{E}^{[i-1]}_{\widetilde{\mathcal{F}}}$.  Since $b_i\to w_i$ is a valid skew force in $T$ for $E^{[i-1]}_{\mathcal{F}}$, $N_T(b_i)\setminus E^{[i-1]}_{\mathcal{F}}=N_T(b_i)\setminus \widetilde{E}^{[i-1]}_{\mathcal{\widetilde{F}}}=\{w_i\}$.  We now wish to show that $\comp(H-u_i,w_i)\cap N_H(u_i)\subseteq N_T(b_i)$.  To show this, we proceed by contradiction: let $z\in(\comp(H-u_i,w_i)\cap N_H(u_i))\setminus N_T(b_i)$.  Let $y=\phi_-(u_iz)$ be the unique vertex in $N_T(u_i)\cap N_T(z)$. Note that $y\neq b_i$, since $z\notin N_T(b_i)$.  $T$ is a tree, so $u_i$ separates $b_i$ from the remainder of $N_T(u_i)$ (including $y$).  Thus, $z\notin\comp(T-u_i,w_i)$.  Let $P=p_0\cdots p_k$ (with $p_0=z$ and $p_k=w_i$) be a $zw_i$ path in $H-u_i$.  $z\notin\comp(T-u_i,w_i)$ and $w_i\in\comp(T-u_i,w_i)$, so there must be some $j$ such that $p_j\notin\comp(T-u_i,w_i)$ and $p_{j+1}\in\comp(T-u_i,w_i)$.  $u_i$ separates $p_j$ and $p_{j+1}$ in $T$, so any $p_jp_{j+1}$ path in $T$ must contain $u_i$.  Now, $p_j\phi_-(p_jp_{j+1})p_{j+1}$ is just such a path, so $\phi_-(p_jp_{j+1})=u_i$, and therefore $u_i\in\Bnull(T)$ -- a contradiction since $\Bnull(T) \cap V(H)=\emptyset$.  Thus, $\comp(H-u_i,w_i)\cap N_H(u_i)\subseteq N_T(b_i)$, and 
 $w_i\in\comp(H-\widetilde{E}^{[i-1]}_{\mathcal{\widetilde{F}}},w_i)\cap N_H(u_i)\subseteq (\comp(H-u_i,w_i)\cap N_H(u_i) )\setminus \widetilde{E}^{[i-1]}_{\mathcal{\widetilde{F}}} \subseteq N_T(b_i) \setminus \widetilde{E}^{[i-1]}_{\mathcal{\widetilde{F}}}=\{w_i\}$.
\endgroup
\end{proof}

\begin{corollary}
\label{NonTrivTree:Rel}
Let $T$ be a skew-nontrivial tree.  Then $\Z_-(T)=\Z_+(\SDTwo(T))$ and $\ZTokExch_-(T)=\ZTokExch_+(\SDTwo(T))$.
\end{corollary}

\begin{proof}
If $S$ is a minimum skew forcing set of $T$, then $S\cap\Bnull(T)=\emptyset$.  The remainder follows from \Cref{NonTrivTree:FS}.
\end{proof}

For a graph $G$, the graph obtained by inserting a single vertex on every edge of $G$ is called a $1$\emph{-subdivision} of $G$.

\begin{theorem}
Let $T$ be the 1-subdivision of a tree.  Then $T$ is a skew-nontrivial tree.
\end{theorem}

\begin{proof}
{Let $T$ be the 1-subdivision of a tree $H$.  It is easily verified that $T$ is a tree.  Moreover, the distance between any pair of leaves of $T$ is twice the distance between those leaves in $H$; thus $T$ is skew-nontrivial by \Cref{treenontriv}.}
\end{proof}

\begin{theorem}
\label{SkewNontrivialTreeEquivs}
Let $T$ be a skew-nontrivial tree.  Then the following are equivalent:

\begin{enumerate}
    \item \label{SNT:Z-=1}$ \Z_-(T)=1$.
    \item \label{SNT:SD2Tree} $SD_2(T)$ is a tree.
    \item \label{SNT:1SubdivTree} $T$ is the 1-subdivision of a tree.
    \item \label{SNT:1SubdivSD2} $T$ is the 1-subdivision of $SD_2(T)$.
\end{enumerate}
\end{theorem}

\begin{proof}
\hfill

\noindent(\ref{SNT:Z-=1} $\Rightarrow$ \ref{SNT:SD2Tree})

Assume that $\Z_-(T)=1$.  By \Cref{NonTrivTree:Rel}, $\Z_+(\SDTwo(T))=1$.  By \Cref{Z+=1}, $\SDTwo(T)$ is a tree.

\noindent(\ref{SNT:SD2Tree} $\Rightarrow$ \ref{SNT:1SubdivSD2})

Assume that $\SDTwo(T)$ is a tree.  A tree has no $K_3$ subgraphs, so each (distinct) edge of $\SDTwo(T)$ must correspond to a (distinct) degree-2 vertex in $\Bnull(T)$.  Moreover, each vertex $v\in\Bnull(T)$ has degree $\geq 2$ by \Cref{BnullDegree}, so there is at least one corresponding edge in $E(\SDTwo(T))$.  Thus, $T$ can be derived from $\SDTwo(T)$ by replacing each edge $uw$ with the path $u\phi_-(uw)w$.

\noindent(\ref{SNT:1SubdivSD2} $\Rightarrow$ \ref{SNT:1SubdivTree})

Assume that $T$ is the 1-subdivision of $\SDTwo(T)$.  If a graph contains a cycle, then so does its 1-subdivision.  Likewise, if a graph is disconnected, then so is its 1-subdivision.  Thus, since the 1-subdivision of $\SDTwo(T)$ is a tree (acyclic and connected), so is $\SDTwo(T)$.

\noindent(\ref{SNT:1SubdivTree} $\Rightarrow$ \ref{SNT:Z-=1})

{Assume that $T$ is the 1-subdivision of a tree $H$ of order $n$.  $\{V(H),V(T)\setminus V(H)\}$ is a bipartition of $V(T)$ such that $\deg_T(v)=2$ for each $v\in V(T)\setminus V(H)$.  Thus, by \Cref{SpecialForest}, $\match(T)=|V(T)\setminus V(H)|=|E(H)|=n-1$.  Moreover, by \Cref{ZMinusVsMatchingNumber}, $\Z_-(T)=|V(T)|-2\match(T)=|V(H)|+|V(T)\setminus V(H)|-2(n-1)=1$.}
\end{proof}

\begin{corollary}
Let $T$ be a skew-nontrivial tree on $n$ vertices with $\Z_-(T)=1$.  Then $\ZTokExch_-(T)\cong K_{(n+1)/2}$ and $\ZTokSlide_-(T)\cong \frac{n+1}{2}K_1$.
\end{corollary}

\begin{proof}
By \Cref{SkewNontrivialTreeZ-}, $1=n-2|\Bnull(T)|$, so $|\Bnull(T)|=\frac{n-1}{2}$.  Then $|V(\SDTwo)|=|\Wnull(T)|=\frac{n+1}{2}$.  By \Cref{SkewNontrivialTreeEquivs}, $\SDTwo(T)$ is a tree.  Thus, by \Cref{treesPSD}, $\ZTokExch_+(\SDTwo(T))\cong K_{(n+1)/2}$.  It follows from \Cref{NonTrivTree:Rel} that $\ZTokExch_-(T)\cong K_{(n+1)/2}$ as well.  Finally, by \Cref{TslideIsolate}, $\ZTokSlide_-(T)$ contains no edges, so $\ZTokSlide_-(T)\cong \frac{n+1}{2}K_1$ (since $V(\ZTokSlide_-(T))=V(\ZTokExch_-(T))$).
\end{proof}

\subsection{Duplication and reconfiguration for PSD and skew forcing}

We now discuss another similarity between PSD and skew forcing reconfiguration. For PSD forcing, join-duplication is well-behaved, increasing the PSD forcing number by exactly one (other than when duplicating an isolated vertex) and producing token reconfiguration graphs that are closely related to those of the original graph.  Adding an independent twin, however, does not behave so predictably.  Twinning an end vertex of $P_3$ yields $K_{1,3}$, with $\Z_+(P_3)=\Z_+(K_{1,3})=1$.  On the other hand, twinning the central vertex yields $C_4$, with $\Z_+(C_4)=2$.

For skew forcing, it is independent duplication (not join-duplication) that behaves well with respect to reconfiguration graphs.  In fact, the behavior of independent duplication with respect to skew forcing is suprisingly similar to that of join-duplication (of a non-isolated vertex) with respect to PSD forcing: it also increases the forcing number by exactly one, and it has exactly the same effect on its respective token exchange reconfiguration graph.  However, the lack of an edge joining the independent twins in the skew case means that there is a difference in the effects on token sliding reconfiguration graphs.  (To see that join-duplication does not behave well with respect to skew forcing, consider the fact that while join-duplicating the central vertex of $P_3$ yields the diamond graph, with $\Z_-(\text{Diamond})=2=\Z_-(P_3)+1$, join-duplicating an end vertex produces the paw graph, with $\Z_-(\text{Paw})=0$.)

In contrast to PSD forcing and skew forcing, neither independent duplication nor join-duplication is well-behaved with respect to standard forcing.  For instance, if $G=K_3$, then independent duplication of any vertex yields $K_4-e$, where $e$ is any edge of $K_4$, and $\Z(K_4-e)=\Z(K_3)=2$.  Similarly, if $G=K_{1,3}$, then join-duplicating a pendent vertex yields a graph $G^+$ with $\Z(G^+)=\Z(G)=2$.  On the other hand, even though $\Z(P_3)=1$, both join-duplication and independent duplication of any vertex of $P_3$ will yield a graph $G$ with $\Z(G)=2$.

Before proceeding to our results on duplication, we include a pair of brief results about PSD forts that will be useful for the proofs to follow.

\begin{observation}
\label{ConnectedPSDFortChecking}
Let $G$ be a graph, and let $F\subseteq V(G)$.  If $G[F]$ is connected, then $F$ is a PSD fort if and only if for all $w\in V(G)\setminus F$, $|N(w)\cap F|\neq 1$.
\end{observation}

\begin{lemma}
\label{PSDFortComponents}
Let $G$ be a graph, let $F$ be a PSD fort of $G$, and let $v\in F$.  Then $F'=V(\comp(G[F],v))$ is also a PSD fort of $G$.
\end{lemma}

\begin{proof}
If $w\in V(G)\setminus F$, $|N(w)\cap F'|=|N(w)\cap V(\comp(G[F],v))|\neq 1$; if $w\in F\setminus F'$, then $N(w)\cap F'=\emptyset$, since $\comp(G[F],v)\neq\comp(G[F],w)$.  By \Cref{ConnectedPSDFortChecking}, $F'$ is a PSD fort of $G$, since $G[F']$ is connected.
\end{proof}

\begin{lemma}\label{dupsets}
Let $\star\in\{+,-\}$, let $G$ be a graph, and let $v\in V(G)$.  If $\star=+$, let $v$ be additionally constrained to have degree at least one (i.e., not be an isolated vertex.)  Let $G^-$ be the graph constructed from the independent duplication of $v$ (with twin $v^-$), and let $G^+$ be the graph constructed from the join-duplication of $v$ (with twin $v^+$).

Then $B\subseteq V(G)$ is a minimum $\Z_\star$-forcing set of $G$ if and only if one of the following is true:
\begin{itemize}
    \item $v \not \in B$ and both $B \cup \{v\}$ and $B \cup \{v^\star\}$ are minimum {$\Z_\star$-}forcing sets of $G^\star$, or
    \item $v \in B$ and $B \cup \{v^\star\}$ is a minimum $\Z_\star$-forcing set of $G^\star$.
\end{itemize}
Furthermore, $\Z_\star(G^\star)=\Z_\star(G)+1$.  (Note, given $B \subseteq V(G)$ with $v \not \in B$, $B \cup \{v\}$ is a minimum $\Z_\star$-forcing set of $G^\star$ if and only if $B \cup \{v^\star\}$ is a minimum {$\Z_\star$-}forcing set of $G^\star$.)
\end{lemma}

\begin{proof}
The PSD and skew forcing numbers are coverable vertex parameters (see \Cref{coverable_def} and \Cref{CoverableParameters}); thus, if $B$ is a $\Z_{\star}$-forcing set of $G$, then $B\cup\{v^\star\}$ is a $\Z_{\star}$-forcing set of $G^\star$.

Now, let $B^\star$ be a minimum $\Z_\star$-forcing set of $G^\star$.  Since $\{v,v^\star\}$ is a $\Z_\star$-fort of $G^\star$ (in the case of PSD forcing, $N_{G^+}(v)\setminus\{v,v^+\}=N_{G^+}(v^+)\setminus\{v,v^+\}$; for skew forcing, $N_{G^-}(v)=N_{G^-}(v^-)$), $\{v,v^\star\}\cap B^\star\neq\emptyset$.  Without loss of generality, we may assume that either $B^\star\cap\{v,v^\star\}=\{v\}$ or $\{v,v^\star\}\subseteq B^\star$ (since the interchange of $v$ and $v^\star$ is an automorphism of $G^\star$).  Pick a chronological list of forces $\mathcal{F}_0=\{u_i^0\to w_i\}_i$ for $B^\star$ in $G^\star$ with expansion sequence $\{E_i\}_i$.  We define a new chronological list of forces $\mathcal{F}=\{u_i\to w_i\}_i$ for $B^\star$ in $G^\star$ with the same expansion sequence by changing any $u_i^0=v^\star$ to $u_i=v$ and leaving $u_i=u_i^0$ otherwise.  This is valid for PSD for all $i$ since $v\in B^+\subseteq E_i$ and when $v^+\in E_i$ (necessary for $v^+\to w_i$ to be valid), {$N_{G^+}(v)\setminus E_i=N_{G^+}(v^+)\setminus E_i$}. It is valid for skew since {$N_{G^-}(v)\setminus E_i=N_{G^-}(v^-)\setminus E_i$}.  By \Cref{initialPSD}, $\ini(\mathcal{F}|_G)=B^\star\setminus\{v^\star\}$ is a $\Z_\star$-forcing set of $G$.  If $v^\star\in B^\star$, we are done.

Otherwise, assume by way of contradiction that $B=B^\star\setminus\{v\}$ is not a $\Z_\star$-forcing set of $G$.  Then there exists a $\Z_\star$-fort $F$ such that $v\in F$ and $F\cap B=\emptyset$, by \Cref{thmPSDfort} and \Cref{skewfort}.  In the PSD forcing case, we can take $G[F]$ to be connected, since for any PSD fort $F_0$ with $v\in F_0$ and $F_0\cap B=\emptyset$, $F=V(\comp(G[F_0],v))$ is also such a fort, by \Cref{PSDFortComponents}, and its induced subgraph $G[F]$ is connected.  Let $F^\star=(F\setminus\{v\})\cup\{v^\star\}$.  It is asserted that $F^\star$ is a $\Z_{\star}$-fort of $G^\star$, which would lead directly to a contradiction, since $F^\star\cap B^\star=\emptyset$.  In the case of skew forcing, for each $x\in V(G^-)$ with $x\neq v^-$, $|N_{G^-}(x)\cap F^-|=|N_G(x)\cap F|\neq 1$; moreover, $|N_{G^-}(v^-)\cap F^-|=|N_G(v)\cap F|\neq 1$, and $F^-$ is thus a skew fort of $G^-$.

The case of PSD forcing is only slightly more complicated.  Since $v$ is not isolated in $G$, there exists $u\in N_G(v)$.  It then follows that $v$ is not isolated in $G[F]$, as that would imply that both $u\notin F$ and $N_G(u)\cap\comp(G[F],v)=\{v\}$, contradicting the fact that $F$ is a PSD fort of $G$.  Thus, there exists $w\in N_G(v)\cap F$; note that $w\in N_{G^+}(v)\cap F^+$ as well.  Now, as $G[F]$ is connected and $N_{G^+}(v^+)\setminus \{v\}=N_G(v)$, $G^+[F^+]$ is also connected.  Therefore, we only need to check that for all $x\in V(G^+)\setminus F^+$, $|N_{G^+}(x)\cap F^+|\neq 1$ by \Cref{ConnectedPSDFortChecking}.   If $x=v$, then we have that $\{v^+,w\}\subseteq N_{G^+}(x)\cap F^+$.  Otherwise, 
\begin{multline*}
|N_{G^+}(x)\cap F^+|=|(N_{G^+}(x)\cap F^+)\setminus\{v^+\}|+|N_{G^+}(x)\cap\{v^+\}|= \\
|(N_G(x)\cap F)\setminus\{v\}|+|N_G(x)\cap\{v\}|=|N_G(x)\cap F|\neq 1.
\end{multline*}

In both cases, $|N_{G^+}(x)\cap F^+|\neq 1$, so $F^+$ is a PSD fort of $G^+$.
\end{proof}

\begin{theorem}
Let $G$ be a graph, let $v \in V(G)$, let $G^-$ be the graph constructed from the independent duplication of $v$ (with twin $v^-)$, let $G^+$ be the graph constructed from the join-duplication of $v$ (with twin $v^+$), let $\star \in \{+,-\}$, and if $\star=+$ suppose $v$ is not an isolated vertex.  Then
\[\ZTokExch_{\star}(G^{\star}) \cong (\ZTokExch_{\star}(G) \cartProd P_2)/\simname \hspace{0.15in} \text{and} \hspace{0.15in} \ZTokSlide _+(G^+) \cong (\ZTokSlide_+(G) \cartProd P_2)/\simname   \hspace{0.15in} \text{and} \hspace{0.15in} \ZTokSlide _-(G^-) \cong (2\ZTokSlide_-(G))/\simname,\]
such that given $B \in V\left(\ZTokExch_{\star}(G)\right)$ we have $(B,v) \simrel (B,v^{\star})$ if and only if $v \in B$, where $V(P_2)=\{v,v^{\star}\}$.
\end{theorem}

\begin{proof}
The proof will proceed by establishing that there are two copies of $\ZTokExch_{\star}(G)$ contained in $\ZTokExch_{\star}(G^{\star})$.  After establishing this fact, we will analyze precisely where the two copies overlap and where edges exist between the two copies.  Let $\mathcal B_v=\left\{B \cup \{v\}: v \not \in B \in V\left(\ZTokExch_{\star}(G)\right)\right\}$, $\mathcal B_{v^{\star}}=\left\{B \cup \{v^{\star}\}: v \not \in B \in V\left(\ZTokExch_{\star}(G)\right)\right\}$, and $\mathcal B_{v,v^{\star}}=\left\{B \cup \{v^{\star}\}: v \in B \in V\left(\ZTokExch_{\star}(G)\right)\right\}$.  By \Cref{dupsets}, the collection of minimum $\Z_{\star}$-forcing sets of $G^{\star}$, and thus the collection of vertices $V\left(\ZTokExch_{\star}(G^{\star})\right)=V\left(\ZTokSlide_{\star}(G^{\star})\right)$, is $\mathcal B_v \cup \mathcal B_{v^{\star}} \cup \mathcal B_{v,v^{\star}}$.  We proceed to examine the different possible cases for edges in $\ZTokExch_{\star}(G^{\star})$:

If $B_1,B_2 \in \mathcal B_{v^{\star}} \cup \mathcal B_{v,v^{\star}}$ we have $B_1B_2 \in E\left(\ZTokExch_{\star}(G^{\star})\right)$ if and only if $B_1'B_2' \in E\left(\ZTokExch_{\star}(G)\right)$, where $B_1'=B_1 \setminus \{v^{\star}\}$ and $B_2'=B_2 \setminus \{v^{\star}\}$, and thus {$\ZTokExch_{\star}(G^{\star})[\mathcal B_{v^{\star}} \cup \mathcal B_{v,v^{\star}}]\cong\ZTokExch_{\star}(G)$.}

Similarly, given $B_1,B_2 \in \mathcal B_v$ we have $B_1B_2 \in E\left(\ZTokExch_{\star}(G^{\star})\right)$ if and only if $B_1'B_2' \in E\left(\ZTokExch_{\star}(G)\right)$, where $B_1'=B_1 \setminus \{v\}$ and $B_2'=B_2 \setminus \{v\}$.

Finally, given $B_1 \in \mathcal B_v$ and $B_2 \in \mathcal B_{v,v^{\star}}$, we let $B_1'=B_1\setminus\{v\}$ and $B_2'=B_2\setminus\{v^{\star}\}$ and observe that the following are equivalent:

\begin{itemize}
\item $B_1B_2\in E(\ZTokExch_{\star}(G^{\star}))$;
\item $B_1\symmDiff B_2=\{w,v^{\star}\}$ for some $w\in V(G)$, $w\neq v$;
\item $B_1'\symmDiff B_2'=\{w,v\}$ for some $w\in V(G)$, $w\neq v$;
\item $B_1'B_2'\in E(\ZTokExch_{\star}(G))$.
\end{itemize}

\noindent Thus $\ZTokExch_{\star}(G^{\star})[\mathcal B_v \cup \mathcal B_{v,v^{\star}}]\cong\ZTokExch_{\star}(G)$.  It has therefore been established that $\ZTokExch_{\star}(G^{\star})$ contains two copies of $\ZTokExch_{\star}(G)$, namely, $\ZTokExch_{\star}(G^{\star})[\mathcal B_v \cup \mathcal B_{v,v^{\star}}]$ and $\ZTokExch_{\star}(G^{\star})[\mathcal B_{v^{\star}} \cup \mathcal B_{v,v^{\star}}]$.

Now we analyze where these two copies of $\ZTokExch_{\star}(G)$ overlap and the edges between them. Given $B_1 \in \mathcal B_v$ and $B_2 \in \mathcal B_{v^{\star}}$, $B_1B_2 \in E\left(\ZTokExch_{\star}(G^{\star})\right)$ if and only if $B_1 \setminus \{v\}=B_2 \setminus \{v^{\star}\}$.  On the other hand, the two copies overlap precisely at $\mathcal B_{v,v^{\star}}$, which occurs when $B \in \ZTokExch_{\star}(G)$ contains $v$. Thus $\ZTokExch_{\star}(G^{\star}) \cong (\ZTokExch_{\star}(G) \cartProd P_2)/\simname$, where for $B \in V\left(\ZTokExch_{\star}(G)\right)$ we have $(B,v) \simrel (B,v^{\star})$ if and only if $v \in B$ (letting $V(P_2)=\{v,v^{\star}\}$).  

We now consider $\ZTokSlide_+(G^+)$ and $\ZTokSlide_-(G^-)$.  Since $vv^+ \in E(G^+)$, a similar argument to that above shows that $\ZTokSlide_+(G^+) \cong (\ZTokSlide_+(G) \cartProd P_2)/\simname$.  On the other hand, since $vv^- \not \in E(G^-)$, another similar argument shows that $\ZTokSlide_-(G^-) \cong (2\ZTokSlide_-(G))/\simname$, that is, there will be no edges between corresponding vertices in the copies of $\ZTokSlide_-(G)$ given by $\ZTokSlide_-(G^-)[\mathcal B_v \cup \mathcal B_{v,v^-}]$ and $\ZTokSlide_-(G^-)[\mathcal B_{v^-} \cup \mathcal B_{v,v^-}]$.
\end{proof}

\section{Comparison of reconfiguration for different color change rules}\label{compare}

The results of \Cref{universal} serve as a summary of some features of reconfiguration graphs that are shared across the standard, PSD, and skew forcing rules (as well as other graph parameters).  In this section we use the results of \Cref{psd}, \Cref{skew}, and \cite{geneson2023reconfiguration} to summarize the contrast in reconfiguration graphs across these zero forcing variants.  We suppress the proofs of results that are determined by direct calculation on finite graphs, which were verified with software.

\subsection{Connectedness of reconfiguration graphs for zero forcing variants}

We begin by discussing the number of connected components for reconfiguration graphs under token exchange for specific graphs $G$ (i.e. $\ZTokExch(G)$, $\ZTokExch_+(G)$, and $\ZTokExch_-(G)$) and token sliding (i.e. $\ZTokSlide(G)$, $\ZTokSlide_+(G)$, and $\ZTokSlide_-(G)$).  In general, the number of connected components is not comparable.  We illustrate this in \Cref{noncomp}, where the representatives are chosen to be graphs of small order.  Images of the nonstandard graphs in the table are illustrated in \Cref{TableGraphs}.  We also note that there are graphs $G$ such that $\ZTokSlide(G)$ and $\ZTokSlide_-(G)$ have isolated vertices (for example, $P_{2k+1}$ for some integer $k>0$).  However, in light of \Cref{migrate}, if $G$ has an edge then $\ZTokSlide_+(G)$ has no isolated vertices.

\begin{figure}
\centering
\begin{tikzpicture}[every matrix/.style={graph styles 2,every node/.style=vertex,scale=0.75}]
\matrix[column sep=0.6in] (top row) {
 \begin{scope}[shift={(-1,1)}]
  % Vertices
  \node (1) at (0,0) {};
  \node (2) at (0,-2) {};
  \node (3) at (-1,-1) {};
  \node (4) at (2,0) {};
  \node (5) at (2,-2) {};  
  \node (6) at (3,-1) {};

  % Edges
  \draw (1) -- (3);
  \draw (2) -- (3);
  \draw (1) -- (4);
  \draw (4) -- (5);
  \draw (2) -- (5);
  \draw (2) -- (1);
  \draw (4) -- (6);
  \draw (6) -- (5);
  
  \node[caption node,below=16pt of current bounding box.south] {$G_6$};
 \end{scope}
&
 \begin{scope}[shift={(-1,1)}]
  % Vertices
  \node (1) at (0,0) {};
  \node (2) at (0,-2) {};
  \node (3) at (-1,-1) {};
  \node (4) at (2,0) {};
  \node (5) at (2,-2) {};  
  \node (6) at (3,-1) {};

  % Edges
  \draw (1) -- (3);
  \draw (2) -- (3);
  \draw (1) -- (4);
  \draw (4) -- (5);
  \draw (2) -- (5);
  \draw (2) -- (1);
  \draw (4) -- (6);
  \draw (6) -- (5);
  \draw (1) -- (5);
  \draw (3) -- (6);
  \draw (2) -- (4);
  
  \node[caption node,below=16pt of current bounding box.south] {$KB(4,2)$};
 \end{scope}
\\
};
\matrix[below=0.35in of top row,column sep=0.75in] {
 \begin{scope}[scale=0.7]
  \node (0) at (-2, -4) {};
  \node (1) at (-2, -2) {};
  \node (2) at (-2,  0) {};
  \node (3) at (-2,  2) {};
  \node (4) at ( 2, -4) {} edge (0) edge (1) edge (2) edge (3);
  \node (5) at ( 2, -2) {} edge (0) edge (1) edge (2) edge (3);
  \node (6) at ( 2,  0) {} edge (0) edge (1) edge (2) edge (3);
  \node (7) at ( 2,  2) {} edge (0) edge (1) edge (2) edge (3);
  \node (8) at ( 0,  4) {} edge (3) edge (7);
 
  \node[caption node,below=16pt of current bounding box.south] {$G_9$};
 \end{scope}
&
 \begin{scope}[scale=0.25,every node/.append style={minimum size=5.4pt}]
  \node (n1) at (-5,-2) {};
  \node (n2) at (-5,5) {} edge (n1);
  \node (n3) at (2,-2) {} edge (n1);
  \node (n4) at (2,5) {} edge (n2) edge (n3);
  \node (n5) at (-2,-5) {} edge (n1);
  \node (n6) at (-2,2) {} edge (n2) edge (n5);
  \node (n7) at (5,-5) {} edge (n3) edge (n5);
  \node (n8) at (5,2) {} edge (n4) edge (n6) edge (n7);
  \node (n9) at (-12,12) {} edge (n1) edge (n8);
  \node (n10) at (9,12) {} edge (n2) edge (n7) edge (n9);
  \node (n11) at (-12,-9) {} edge (n3) edge (n6) edge (n9);
  \node (n12) at (9,-9) {} edge (n4) edge (n5) edge (n10) edge (n11);
  \node (n13) at (-9,9) {} edge (n4) edge (n5) edge (n9);
  \node (n14) at (12,9) {} edge (n3) edge (n6) edge (n10) edge (n13);
  \node (n15) at (-9,-12) {} edge (n2) edge (n7) edge (n11) edge (n13);
  \node (n16) at (12,-12) {} edge (n1) edge (n8) edge (n12) edge (n14) edge (n15);
  
  \node[caption node,below=16pt of current bounding box.south] {$G_{16}$};
 \end{scope}
\\
};
\end{tikzpicture}
\caption{Nonstandard graphs referenced in \Cref{noncomp}.}
\label{TableGraphs}
\end{figure}
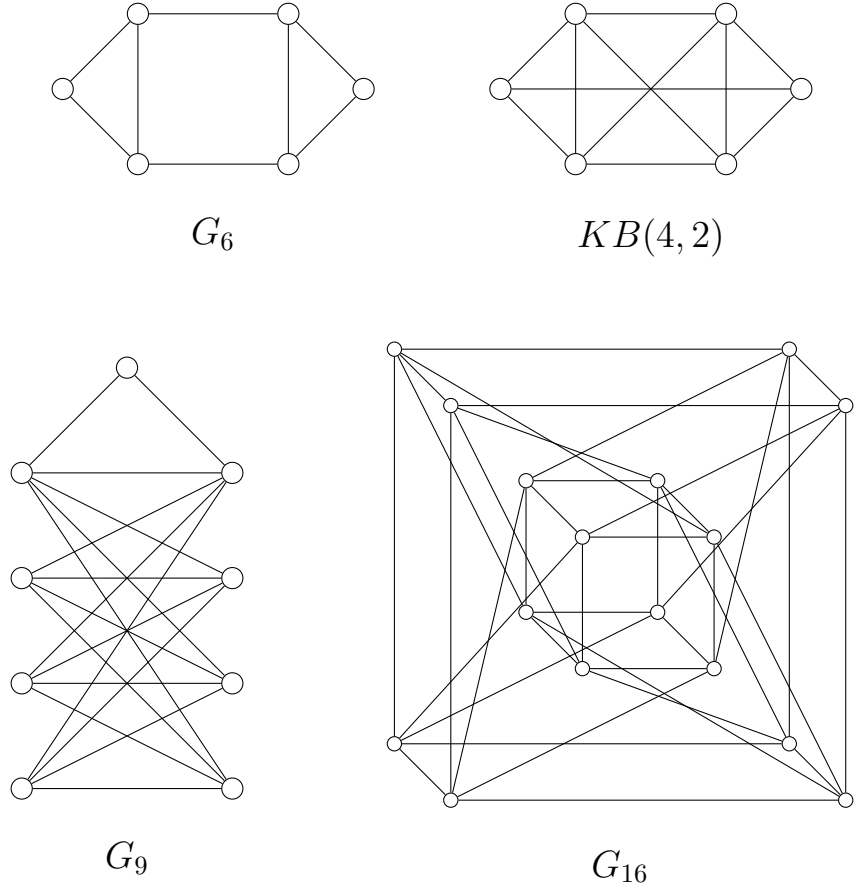

\renewcommand{\arraystretch}{1.25}
\begin{table}[ht]
    \centering
    \begin{tabular}{|c |c ||c|c|c||c|c|c||}\hline
    &&\multicolumn{6}{c||}{Number of connected components} \\
    Graph& Order &\multicolumn{3}{c||}{Token Exchange}&\multicolumn{3}{c||}{Token Sliding}\\
    &&Std &PSD & Skew & Std & PSD & Skew\\
    \hline
%    Blah&3&9&2&9&2&5\\ \hline
    % K_{4,4} \cup K_3 joined on an edge
    $G_9$&9&1&2&1&2&2&1\\ \hline
    $K_{4,4}\vee K_1$&9&1&2&1&1&2&1\\ \hline
    $K_{4,4}$&8&1&2&1&16&2&16\\\hline
    $G_{16}$&16&1&10&1&1&10&11\\\hline
    $KB(4,2)$&6&1&1&2&1&1&4\\\hline
    %Was \GraphSix"C^"
    $K_4-e$&4&1&1&1&2&1&1\\\hline
    $G_6$&6&2&1&1&2&1&1\\\hline
    $P_3$&3&1&1&1&2&1&2\\\hline
    \end{tabular}
    \caption{Graphs illustrating the noncomparability of connected components of reconfiguration graphs across zero forcing variants}
    \label{noncomp}
\end{table}

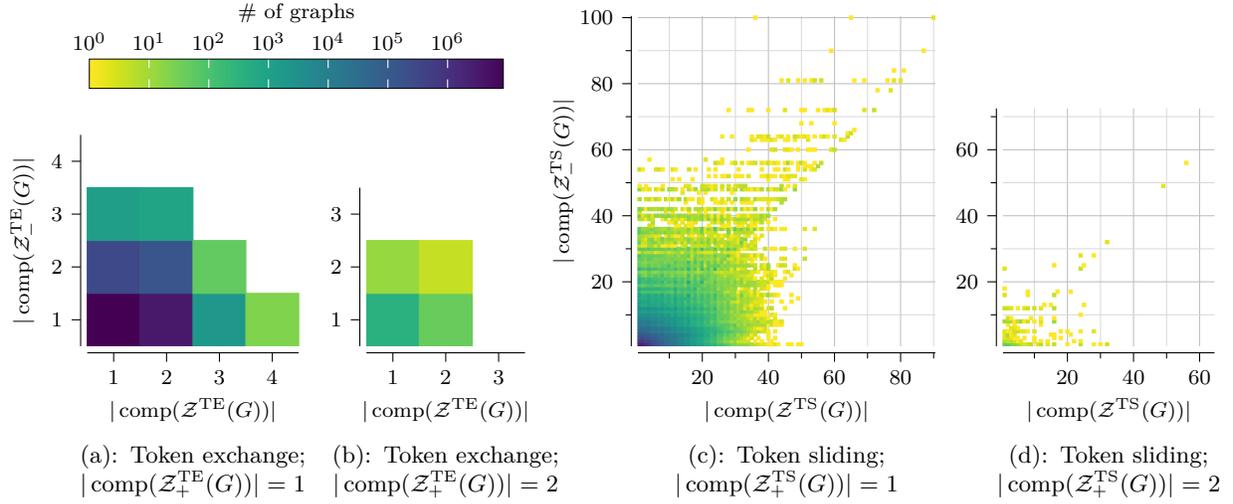
\begin{figure}
\centering
\tikzset{external/export,external/figure name=tgraph10_small2_}
\begin{tikzpicture}[custom matrix plot settings]
\begin{scope}[/pgfplots/.cd,xmin=0.5,xmax=4.5,ymin=0.5,ymax=4.5,point meta min=0,point meta max=6.9330,max space between ticks=18]
\begin{axis}[title={(a): Token exchange; $|\comp(\TokExch_+(G))|=1$},xlabel={$|\comp(\TokExch(G))|$},ylabel={$|\comp(\TokExch_-(G))|$},name=exch 10 psd 1,colorbar horizontal,x=20pt,y=20pt,scale mode=none,colorbar style={width=2*\pgfkeysvalueof{/pgfplots/parent axis width}-0.12cm,at={(0.01,1.22)}}]
\addplot [matrix plot*,mesh/rows=3] table [x=std,y=skew,meta=logc] {
std   skew     count     logc
 1      1   8570979   6.9330
 1      2    256251   5.4087
 1      3      1152   3.0615
 2      1   3044069   6.4835
 2      2    129690   5.1129
 2      3       766   2.8842
 3      1      1788   3.2524
 3      2        47   1.6721
 3      3         0      nan
 4      1        21   1.3222
 4      2         0      nan
 4      3         0      nan
};
\end{axis}

\begin{axis}[at={($(exch 10 psd 1.south east)+(0.9cm,0)$)},anchor=south west,title={(b): Token exchange; $|\comp(\TokExch_+(G))|=2$},xlabel={$|\comp(\TokExch(G))|$},name=exch 10 psd 2,xmax=3.5,ymax=3.5,x=20pt,y=20pt,scale mode=none]
\addplot [matrix plot*,mesh/rows=2] table [x=std,y=skew,meta=logc] {
std   skew     count     logc
 1      1       350   2.5441
 1      2        12   1.0792
 2      1        39   1.5911
 2      2         4   0.6021
};
\end{axis}
\end{scope}

\begin{scope}[/pgfplots/.cd,xmin=0.5,xmax=90.5,ymin=0.5,ymax=100.5,point meta min=0,point meta max=6.9330,minor tick num=1,grid=both,x=1.25pt,y=1.25pt,scale mode=none]
\begin{axis}[at={($(exch 10 psd 2.south east)+(1.5cm,0)$)},anchor=south west,title={(c): Token sliding; $|\comp(\TokSlide_+(G))|=1$},xlabel={$|\comp(\TokSlide(G))|$},ylabel={$|\comp(\TokSlide_-(G))|$},ylabel shift=-6pt,name=slide 10 psd 1]
\addplot [matrix plot*,mesh/rows=100] table [x=std,y=skew,meta=logc] {DataFiles/data_10_slide_psd1.txt};
\end{axis}

\begin{axis}[at={($(slide 10 psd 1.south east)+(0.9cm,0)$)},anchor=south west,title={(d): Token sliding; $|\comp(\TokSlide_+(G))|=2$},xlabel={$|\comp(\TokSlide(G))|$},name=slide 10 psd 2,xmax=64.5,ymax=72.5]
\addplot [matrix plot*,mesh/rows=56] table [x=std,y=skew,meta=logc] {DataFiles/data_10_slide_psd2a.txt};
\end{axis}
\end{scope}
\end{tikzpicture}
\caption{Comparison of the relative frequency of graphs having specific numbers of connected components in their token exchange or token sliding reconfiguration graphs for various color change rules.  (For this figure, all 12,005,168 non-isomorphic graphs on 10 vertices were used.)}
\label{Plots10}
\end{figure}

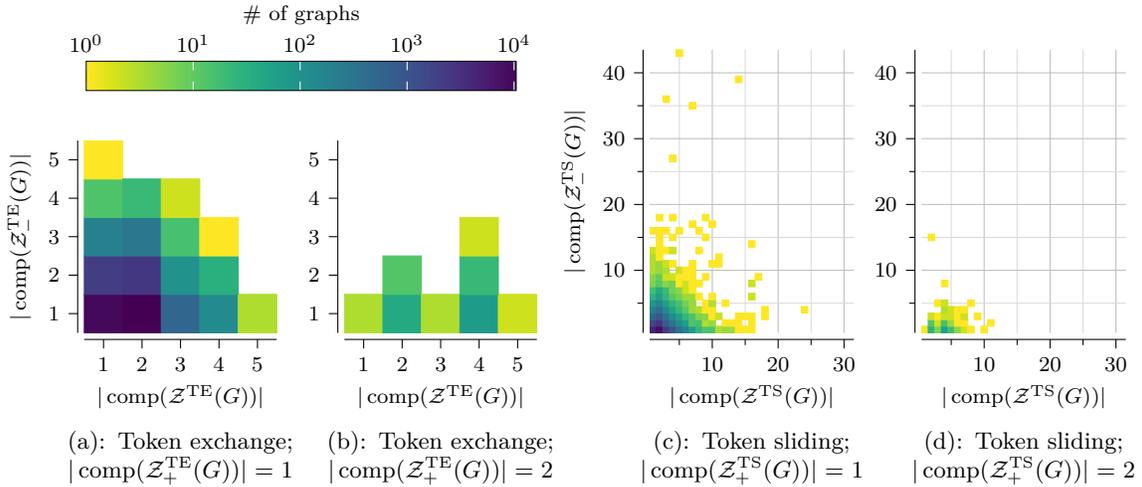
\begin{figure}
\centering
\tikzset{external/export,external/figure name=tgraph20_small2_}
\begin{tikzpicture}[custom matrix plot settings]
\begin{scope}[/pgfplots/.cd,xmin=0.5,xmax=5.5,ymin=0.5,ymax=5.5,point meta min=0,point meta max=4.0201,width=4.8cm,max space between ticks=18]
\begin{axis}[axis equal image,title={(a): Token exchange; $|\comp(\TokExch_+(G))|=1$},xlabel={$|\comp(\TokExch(G))|$},ylabel={$|\comp(\TokExch_-(G))|$},name=exch 20 psd 1,colorbar horizontal,colorbar style={width=2*\pgfkeysvalueof{/pgfplots/parent axis width}+0.6cm,at={(0.01,1.26)}}]
\addplot [matrix plot*,mesh/rows=5] table [x=std,y=skew,meta=logc] {
   std   skew   count     logc
     1      1    8673   3.9382
     1      2    2048   3.3113
     1      3     194   2.2878
     1      4      13   1.1139
     1      5       1   0.0000
     2      1   10473   4.0201
     2      2    2388   3.3780
     2      3     262   2.4183
     2      4      20   1.3010
     2      5       0      nan
     3      1     507   2.7050
     3      2      96   1.9823
     3      3      15   1.1761
     3      4       2   0.3010
     3      5       0      nan
     4      1     126   2.1004
     4      2      31   1.4914
     4      3       1   0.0000
     4      4       0      nan
     4      5       0      nan
     5      1       3   0.4771
     5      2       0      nan
     5      3       0      nan
     5      4       0      nan
     5      5       0      nan
};
\end{axis}

\begin{axis}[at={($(exch 20 psd 1.south east)+(0.9cm,0)$)},anchor=south west,axis equal image,title={(b): Token exchange; $|\comp(\TokExch_+(G))|=2$},xlabel={$|\comp(\TokExch(G))|$},name=exch 20 psd 2]
\addplot [matrix plot*,mesh/rows=3] table [x=std,y=skew,meta=logc] {
   std   skew   count     logc
     1      1       3   0.4771
     1      2       0      nan
     1      3       0      nan
     2      1      39   1.5911
     2      2      12   1.0792
     2      3       0      nan
     3      1       3   0.4771
     3      2       0      nan
     3      3       0      nan
     4      1      65   1.8129
     4      2      21   1.3222
     4      3       2   0.3010
     5      1       2   0.3010
     5      2       0      nan
     5      3       0      nan
};
\end{axis}
\end{scope}

\begin{scope}[/pgfplots/.cd,xmin=0.5,xmax=31.5,ymin=0.5,ymax=43.5,point meta min=0,point meta max=4.0201,width=6.2cm,minor tick num=1,grid=both]
\begin{axis}[at={($(exch 20 psd 2.south east)+(1.5cm,0)$)},anchor=south west,axis equal image,title={(c): Token sliding; $|\comp(\TokSlide_+(G))|=1$},xlabel={$|\comp(\TokSlide(G))|$},ylabel={$|\comp(\TokSlide_-(G))|$},name=slide 20 psd 1]
\addplot [matrix plot*,mesh/rows=43] table [x=std,y=skew,meta=logc] {DataFiles/data_20_slide_psd1.txt};
\end{axis}

\begin{axis}[at={($(slide 20 psd 1.south east)+(0.9cm,0)$)},anchor=south west,axis equal image,title={(d): Token sliding; $|\comp(\TokSlide_+(G))|=2$},xlabel={$|\comp(\TokSlide(G))|$},name=slide 20 psd 2]
\addplot [matrix plot*,mesh/rows=15] table [x=std,y=skew,meta=logc] {DataFiles/data_20_slide_psd2a.txt};
\end{axis}
\end{scope}
\end{tikzpicture}
\caption{Comparison of the relative frequency of graphs having specific numbers of connected components in their token exchange or token sliding reconfiguration graphs for various color change rules.  (For this figure, 25,000 random graphs on 20 vertices were generated, with edge probability $\frac{1}{2}$.)}
\label{Plots20}
\end{figure}

To illustrate the full range of possible number of disconnected components of reconfiguration graphs, the reconfiguration graphs for standard, PSD and skew forcing were computed for two collections of graphs. 
 For each of the graphs considered, the PSD reconfiguration graphs had at most two connected components.  The first heat map in Figure \ref{Plots10} considers graphs from the collection of all 12,005,168 non-isomorphic graphs on 10 vertices for which the PSD token exchange graph is connected, and illustrates the number of connected components of $\ZTokExch(G)$ versus the number of components of $\ZTokExch_-(G)$. The second heat map, makes a similar comparison of graphs for which the PSD token exchange graph has two connected components.  The third and fourth heat maps make analogous comparisons to those made in the first and second for the number of components of $\ZTokSlide(G)$ versus $\ZTokSlide_-(G)$. In each heat map, a tile at position $(x,y)$ indicates that there is at least one graph $G$ in the data set for which $\ZTokExch(G)$ (or $\ZTokSlide(G)$) has $x$ components and $\ZTokExch_-(G)$ (or $\ZTokSlide_-(G)$) has $y$ components, with the color of the tile indicating the approximate number of such graphs in the collection.  Figure \ref{Plots20} shows a similar comparison for the 25,000 random graphs on $20$ vertices.

\subsection{Comparison of reconfiguration graphs corresponding to families of source graphs across zero forcing variants}
\subsubsection{Trees}
Let $T$ be a tree.  Then $\ZTokExch(T)$ is connected and furthermore, if $T$ is not a path then $\ZTokExch(T)$ contains either a $C_3$ or $C_4$ subgraph \cite{geneson2023reconfiguration}.  The PSD color change rule yields $\ZTokExch_+(T) \cong K_r$ where $\vert V(T) \vert = r$ (\Cref{treesPSD}).  Finally, the skew color change rule yields that $\ZTokExch_-(T)$ is connected (\Cref{treesSkew}), with additional information about its reconfiguration given in \Cref{skew}.  For token sliding, we showed that the PSD token sliding graph of a tree is itself (\Cref{treesPSD}), whereas the skew token sliding graph of a tree consists solely of isolated vertices (\Cref{TslideIsolate}).

Now consider the realizability of $T$ as a token exchange graph under the various color change rules.  This article showed that there are no graphs $G$ such that $\ZTokExch_+(G) \cong T$  except in trivial cases (\Cref{noPSDteTree}).  We also conjecture that trees are not realizable as $\ZTokExch_-(G)$ for any graph $G$, but establishing this will require additional tools. 
 However, many different graphs $G$ can satisfy $\ZTokExch(G)\cong P_n$ for a given $n$.  For example, for $n\geq 5$, let $C_n(2)$ denote a cycle with vertices {$v_1,v_2, \dots, v_n$ and the two additional edges $v_1v_3$ and $v_3v_n$.}  Then $\ZTokExch(C_n(2)) \cong P_{n-1}$. Two examples with $P_5$ as their standard token exchange graph, $G_1=C_6(2)$ and $G_2$, are given in \Cref{fig:C6(2)}.  

\begin{figure}
    \centering
\begin{tikzpicture}
  \tikzstyle{vertex}=[circle, draw, inner sep=0pt, minimum size=4mm]

  \foreach \i in {1,2,3,4,5,6}
    \node[vertex] (\i) at ({(\i+1)*360/6}:1cm) {};
  \node[caption node] (b) at (0,-1.5) {$G_1$};
  \draw (1) -- (2);
  \draw (2) -- (3);
  \draw (3) -- (4);
  \draw (4) -- (5);
  \draw (5) -- (6);
  \draw (6) -- (1);
  \draw (1) -- (3);
  \draw (3) -- (6);

  \begin{scope}[shift={(2,1.5)}]

  \node[vertex] (11) at (1, -1) {};
  \node[vertex] (12) at (2, -1) {};
  \node[vertex] (13) at (3, -1) {};
  
  \node[vertex] (21) at (1, -2) {};
  \node[vertex] (22) at (2, -2) {};
  \node[vertex] (23) at (3, -2) {};
  \node[vertex] (24) at (4, -2) {};
  \node[caption node] (a) at (2 , -3) {$G_2$};
  
  % Horizontal edges
  \draw (11) -- (12) -- (13);
  \draw (21) -- (22) -- (23);
  
  % Vertical edges
  \draw (11) -- (21);
  \draw (12) -- (22);
  \draw (13) -- (23);
  \draw (23) -- (24);
  \end{scope}
\end{tikzpicture}
    \caption{Two graphs $G_1,G_2$ satisfying $\ZTokExch(G_i)=P_5$ for $i=1,2$.}
    \label{fig:C6(2)}
\end{figure}

\subsubsection{Complete Graphs}
The reconfiguration graphs of a complete graph $K_n$ are straightforward to calculate, but provide a contrast between the color change rules. For any value of $n$, we have that $\ZTokExch(K_n) = \ZTokExch_+(K_n)\cong K_n$ (\Cref{psdCompleteTETS}), while $\ZTokExch_-(K_n) = J(n,n-2)$ (\Cref{SkewOfKn}).  In all cases the results follow directly from the value of the associated zero forcing parameter and symmetry of $K_n$, where $\Z(K_n) = \Z_+(K_n) = n-1$ and $\Z_-(K_n) = n-2$.  Since all token exchanges can be viewed as token slides in the complete graph, the same results apply to the token sliding graphs.

The realizability question of complete graphs provides a more subtle contrast between the zero forcing rules.   {The results for skew forcing are less precise than those for standard and PSD forcing due to the broad variety of graphs achieving extreme values of the skew forcing number.}

Using the standard color change rule, discussed in \cite{geneson2023reconfiguration}, {if $G$ has no isolated vertices, then} $\ZTokExch(G) \cong K_2$ if and only if $G\cong P_n$ for some $n\geq 2$.  Additionally, if $G$ has no isolated vertices and $r\geq 3$, then $\ZTokExch(G) \cong K_r$ if and only if $G\cong K_r$ or $G\cong K_{1,r}$.  Since all zero forcing sets of $K_{1,r}$ consist of $r-2$ leaves, one readily deduces that for a graph $G$ with no isolated vertices, $\ZTokSlide(G) \cong K_r$ if and only if $G\cong K_r$.

For the PSD color change rule, complete graphs were handled in \Cref{psdCompleteTETS} and \Cref{CompleteTokExch}.  Given a graph $G$ with no isolated vertices, we have $\ZTokExch_+(G) \cong K_m $ if and only if $G\cong K_m$ or $G$ is a tree with $m$ vertices.  It then follows that for a graph $G$ with no isolated vertices, $\ZTokSlide_+(G) \cong K_m$ if and only if $G\cong K_m$.

Using the skew color change rule, one can deduce that if $G=K_{1,m}$ or $\Z_-(G)=1$, then $\ZTokExch(G)\cong K_m$ for some $m$ (\Cref{skewTEcomplete}, \Cref{skewstar}).  In the token sliding case there are ways to construct graphs $G$ such that $\ZTokSlide_-(G) \cong K_m$ for some $m$, but no characterization is currently known. As an example, we define a class of graphs $\{G_n\}_{n\geq 3}$ by labeling the vertices of a complete graph $K_n$ on $n\geq 3$ vertices by $\{v_i\}_{i=1}^n$, and adding a path $v_iu_iv_{i+1}$ between each pair of consecutively labeled vertices to form a new graph $G_n$ (note that $v_1$ and $v_n$ are \emph{not} considered to be consecutive).  Then $\ZTokSlide_-(G_n)\cong K_n$, as the minimum skew forcing sets of $G_n$ are precisely the sets $\{v_i\}$ where $1\leq i\leq n$.

\section{Conclusion}

In this article, we presented a universal framework for token exchange and token sliding reconfiguration before going into more detail for PSD forcing and skew forcing. We found that many properties of the token exchange and token sliding reconfiguration graphs are parameter specific, including connectedness and realizability of trees and complete graphs as reconfiguration graphs. Further investigation is needed to study how these properties behave for the other summable (and coverable) graph parameters listed in Section \ref{universal}. 

\begin{appendices}

\section{Further details concerning vertex parameters}\label{table}

In this appendix, we further discuss the details of the table in Section \ref{tablesection}.  First, we identify the vertex parameters we will be discussing.  Then we identify whether each of these vertex parameters is summable (\Cref{def:SummableDefinition}), coverable (\Cref{coverable_def}), or an $X$-set parameter (\Cref{x_set_def}), providing short arguments for each respective conclusion. 

Many common graph parameters satisfy the definition of vertex parameters, including the independence number \cite{zfvsindep}, vertex cover number \cite{zfvsvc}, zero forcing number \cite{barioli2013aim}, PSD forcing number \cite{param}, skew forcing number \cite{allison2010minimum}, $k$-forcing number \cite{kforce} (denoted $F_k$), minor monotone floor of zero forcing number \cite{paramlong} (denoted $\lfloor\Z\rfloor$), power domination number \cite{haynes2002powerdom} (denoted $\gamma_P$), domination number \cite{zfvsdom} (denoted $\gamma$), and total domination number \cite{totald} (denoted $\gamma_t$).

Path covers, tree covers, spider covers, and clique covers are each examples of $\mathcal G$-covers, and the path cover number \cite{param}, tree cover number \cite{bozemantree}, {spider cover number \cite{haynes2002powerdom}}, and clique cover number \cite{barioli2013aim} are each partition parameters.  The chromatic number \cite{zfvschrom} is also an example of a partition parameter. Due to this, the path cover transversal number, tree cover transversal number, spider cover transversal number, clique cover transversal number, and chromatic transversal number are all vertex parameters.

We now show that the skew forcing number, while not an $X$-set parameter \cite{bjorkman2022power,bong2022isomorphisms}, is both summable and coverable.

\begin{proposition}
\label{SkewIsSummable}
The skew forcing number is a summable vertex parameter.
\end{proposition}

\begin{proof}
Let $G$ be a disconnected graph, and let $B$ be a skew forcing set of $G$.  Also, for each component $C \in \comp(G)$, let $B_C=B \cap V(C)$, observing that $B=\bigcup_{C \in \comp(G)}B_C$.  Note that for each vertex $u \in V(G)$ and each pair of components $C,C' \in \comp(G)$ with $u \in V(C)$ and $u \not \in V(C')$, we have $N_C(u)\setminus B_C=N_G(u)\setminus B$ and $N_G(u) \cap V(C')=\emptyset$.  So it follows that $B$ is a skew forcing set of $G$ if and only if $B_C$ is a skew forcing set of $C$ for each $C \in \comp(G)$.  Thus the skew forcing number is a summable vertex parameter.
\end{proof}

\begin{proposition}
\label{SkewIsCoverable}
The skew forcing number is a coverable vertex parameter.
\end{proposition}

\begin{proof}
Let $G$ be a graph, $v \in V(G)$, and $B$ be a skew forcing set of $G-v$.  Since $N_G(u)\setminus (B \cup \{v\})=N_{G-v}(u)\setminus B$ for each vertex $u \in V(G)$, it follows that $B \cup \{v\}$ is a skew forcing set of $G$.  Thus the skew forcing number is a coverable vertex parameter.
\end{proof}

Similar proofs to those of \Cref{SkewIsSummable,SkewIsCoverable} apply for the parameters in the following observations; we omit the proofs for brevity.  We begin by examining which parameters are summable and which are not.

\begin{observation}
The independence number, vertex cover number, zero forcing number, PSD forcing number, skew forcing number, $k$-forcing number, power domination number, domination number, and total domination number 
are all summable vertex parameters.  This occurs because when working with these parameters components operate independently.  
\end{observation}

\begin{observation}\label{gcovcon}
Let $\mathcal G$ be a class of graphs.  If for each $G \in \mathcal G$, $G$ is connected, then the $\mathcal G$-cover transversal number is a summable vertex parameter, as this ensures that components operate independently.  In particular, the path cover transversal number, tree cover transversal number, spider cover transversal number, and clique cover transversal number are {summable vertex parameters}.
\end{observation}

\begin{observation}
The minor monotone floor of zero forcing number is not a summable vertex parameter because forces can occur between vertices in different components.  For instance, if $G$ is the graph on $n>1$ isolated vertices, then $\lfloor \Z \rfloor(G)=1$, but for each component $C \in \comp(G)$, $\lfloor \Z \rfloor (C)=1$, so $\lfloor \Z\rfloor (G)=1 < n=\sum_{C \in \comp(G)} \lfloor \Z\rfloor (C)$, which would otherwise contradict \Cref{compsum}.
\end{observation}

\begin{observation}
\Cref{compsum} also shows that the chromatic transversal number is not a summable vertex parameter because if $G$ is a disconnected graph, then
\[\chi_{\mathcal T}(G)=\chi(G)=\max_{C \in \comp(G)}\chi(C) < \sum_{C \in \comp(G)} \chi(C)=\sum_{C \in \comp(G)}\chi_{\mathcal T}(C),\] where $\chi(G)$ is the chromatic number of $G$.  Intuitively, one can view this as occurring because each color class in a proper coloring of $G$ may occur in multiple components, but a transversal of the proper coloring would only include one vertex from each color class. 
\end{observation}

We now address which vertex parameters are coverable and which are not.

\begin{observation}
\label{CoverableParameters}
For each parameter $X \in \{\Z, \Z_+, \Z_-, \F_k, \lfloor \Z \rfloor,\gamma_P\}$, $X$ is a coverable vertex parameter because if $B$ is an $X$-forcing set of $G-v$, then $B \cup \{v\}$ is an $X$-forcing set of $G$.
\end{observation}

\begin{observation}
The vertex cover number of a graph is a coverable vertex parameter because if $B$ is a vertex cover of $G-v$, then $B \cup \{v\}$ is a vertex cover of $G$, as by definition each edge in $G-v$ is incident to a member of $B$ and for each $e \in E(G) \setminus E(G-v)$, $e$ is incident to $v$.  If $B$ is a dominating set of $G-v$, then $B \cup \{v\}$ is a dominating set of $G$, as by definition each vertex in $G-v$ is either in $B$ or adjacent to a member of $B$, and clearly $v \in B \cup \{v\}$.  Thus, the domination number of a graph is also a coverable vertex parameter.
\end{observation}

\begin{observation}
The chromatic transversal number of a graph is coverable because if $\mathcal Y=\{Y_i\}_{i=1}^{\chi(G-v)}$ is a proper coloring of $G-v$, then $\mathcal Y'= \mathcal Y \cup \{\{v\}\}$ is a proper coloring of $G$ and for any transversal $B$ of $\mathcal Y$, $B \cup \{v\}$ is a transversal of $\mathcal Y'$. 
\end{observation}

\begin{observation}
The independence number and total domination number are not coverable vertex parameters.  For the independence number, we consider $G=K_2$ with vertex set $V(G)=\{u,v\}$.  $\{u\}$ is an independent set of $G-v$; however, $\{u\}\cup\{v\}$ is not an independent set of $G$.  For the total domination number, we consider $H=P_4$ with vertex set $V(H)=\{v_0,v_1,v_2,v_3\}$ (indexed in path order).  $\{v_0,v_1\}$ is a total dominating set of $H-v_3$; however, $\{v_0,v_1\}\cup\{v_3\}$ is not a total dominating set of $H$.
\end{observation}

\begin{observation}
Let $\mathcal G$ be a class of graphs.  If $K_1 \in \mathcal G$, then the {$\mathcal G$-cover transversal number} is a {coverable vertex parameter}. In particular, the path cover transversal number of a graph, tree cover transversal number of a graph, spider cover transversal number of a graph, and clique cover transversal number of a graph are {coverable vertex parameters}.
\end{observation}

Finally, we address which vertex parameters are $X$-set parameters and which are not.  First, in \cite{bong2022isomorphisms}, the following observation was given.

\begin{observation}[\cite{bong2022isomorphisms}]
The zero forcing number, PSD forcing number, power domination number, and domination number are all $X$-set parameters.  
\end{observation}

\begin{observation}
The $k$-forcing number is an $X$-set parameter.  First, if $B$ is a $k$-forcing set and $\mathcal F$ is a relaxed chronology of $k$-forces of $B$ on $G$, then for any superset $B'$ of $B$ one can construct a relaxed chronology of $k$-forces $\mathcal F'$ of $B'$ on $G$ by deleting any forces in $\mathcal F$ which result in vertices in $B'$ becoming blue.  Second, under the $k$-forcing color change rule, white vertices cannot perform forces, so the empty set is not a $k$-forcing set of any graph.  Finally, if $G$ has no isolated vertices, then for any $u \in  V(G)$, $V(G) \setminus \{u\}$ is a $k$-forcing set because any vertex $v \in N_G(u) \neq \emptyset$ has the single white neighbor $u$ and can perform the necessary force.    
\end{observation}

For the purposes of the following result, a class of graphs is said to be nontrivial if it contains more than one graph.

\begin{proposition}\label{gcov}
If $\mathcal G$ is a nontrivial class of graphs satisfying the property that: 
\[\text{each }G \in \mathcal G \text{ is connected and for any graph } G \in \mathcal G, \text{ if } H \text{ is a connected subgraph of }G,\text{ then }H \in \mathcal G,\]
then the $\mathcal G$-cover transversal number is an $X$-set parameter.
\end{proposition}

\begin{proof}
Let $G$ be a graph.  We will now show that the $\mathcal G$-cover transversal number satisfies the 4 criteria necessary to be an $X$-set parameter.

(1) Let $\mathcal Y$ be a $\mathcal G$-cover of $G$, and $B$ be a transversal of $\mathcal Y$.  Now let $B'=B \cup \{v\}$ for some vertex $v \in V(G) \setminus B$.  Since $\mathcal Y$ is a partition of $V(G)$, there exists $Y \in \mathcal Y$ such that $v\in Y$. Let $u \in B\cap Y$.  If $v$ is a cut-vertex of $G[Y]$, let $H_1$ be the component of $G[Y]-v$ containing $u$ and $H_2=G[Y\setminus V(H_1)]$, noting that $H_2$ contains $v$ and is thus connected.  If $v$ is not a cut-vertex of $G[Y]$, then let $H_2$ be the trivial graph on $\{v\}$ and $H_1=G[Y]-v$, noting that $H_1$ is connected since $v$ is not a cut-vertex.  In either case, since $H_1$ and $H_2$ are connected subgraphs of $G[Y]$, it follows that $H_1,H_2 \in \mathcal G$, and so $\mathcal Y'=(\mathcal Y \setminus Y) \cup \{V(H_1),V(H_2)\}$ is a $\mathcal G$-cover of $G$ and $B'$ is a transversal of $\mathcal Y'$.  Finally, this process can be repeated as necessary to show that any superset of $B$ is a transversal of a $\mathcal G$-cover of $G$.  

(2) Partitions must be non-empty, and as a result so must their transversals.  Thus, the empty set is not a transversal of any $\mathcal G$-cover of $G$.  

(3) By \Cref{gcovcon}, since $\mathcal G$ is a class of connected graphs, the $\mathcal G$-cover transversal number is a summable vertex parameter. 

(4) First note that since $\mathcal G$ is nontrivial and each member of $\mathcal G$ is connected, $\mathcal G$ must contain a graph $H$ with an edge.  Since each connected subgraph of $H$ is a member of $\mathcal G$, it follows that $\mathcal G$ must contain both $K_1$ and $K_2$.  Now suppose that $G$ is a graph of order $n$ with no isolated vertices and consider the set of vertices $B=V(G) \setminus \{w\}$ for some vertex $w \in V(G)$.  Since $G$ contains no isolated vertices, there must exist a vertex $u \in V(G)$ such that $uw \in E(G)$.  Since $uw \in E(G)$ and $K_1,K_2 \in \mathcal G$, the partition $\{u,w\} \cup \{\{v\}\}_{v \in V(G)\setminus \{u,w\}}$ is a $\mathcal G$-cover of $G$, and thus $B$ is a $\mathcal G$-cover transversal of $G$.
\end{proof}

\begin{observation}
Since the graph classes of paths, trees, cliques, and spiders each satisfy the hypotheses of \Cref{gcov}, it follows that the path cover transversal number, the tree cover transversal number, the clique cover transversal number, and the spider cover transversal number are each $X$-set parameters.
\end{observation}

\begin{observation}
The minor monotone floor of the zero forcing number and chromatic transversal number are not $X$-set parameters because they are not summable and thus do not satisfy criterion (3).
\end{observation}

\begin{observation}
The skew forcing number is not an $X$-set parameter because $\emptyset$ is a skew forcing set of $K_2$.  Similarly, the vertex cover number is not an $X$-set parameter because $\emptyset$ is vacuously a vertex cover of $K_1$.  Thus they both fail criterion (2). 
\end{observation}

\begin{observation}
The total domination number is not an $X$-set parameter because $\gamma_t(K_2)=2$, thus failing criterion (4).
\end{observation}

\begin{observation}
The independence number fails criterion (4) and is thus not an $X$-set parameter because $K_2$ is the only connected nontrivial graph for which every set of $|V(G)|-1$ vertices is an independent set; it also fails criterion (1) since supersets of maximal independent sets are not independent sets.
\end{observation}

\end{appendices}

\bibliographystyle{plain}
\bibliography{Reconfig}

\section*{Statements and Declarations}
This group was made possible by American Institute of Mathematics virtual research community Inverse Eigenvalue Problem for a Graph and Zero Forcing. We would like to thank AIM for supporting collaboration.
\subsection*{Funding}
The work of Mary Flagg was partially supported by the National Science Foundation through grant no.2331634. \\
\noindent
The authors have no relevant financial or non-financial interests to disclose.
\subsection*{Data Availability}
All data generated or analysed during this study are included in this published article (and its supplementary information files).
\end{document}